\documentclass[11pt]{amsart}
\usepackage{tikz,amsthm,amsmath,amstext,amssymb,amscd,epsfig,euscript, mathrsfs, dsfont,pspicture,multicol,graphpap,graphics,graphicx,times,enumerate,subfig,sidecap,wrapfig,color}

\usepackage{amsmath}
\usepackage{amssymb}
\usepackage{pgf}
\usepackage{stackrel}
\usepackage[colorlinks,citecolor=red]{hyperref}

\newcommand{\R}{{\mathbb R}}       % Field of real numbers
\newcommand{\N}{{\mathbb N}}       %
\newcommand{\HH}{{\mathcal H}}

\newcommand{\LL}{{\mathcal L}}

\def\one{\mathds{1}}

\newcommand{\diam}{\mathop{\rm diam}}
\newcommand{\dist}{{\rm dist}}

\newcommand{\rf}[1]{{(\ref{#1})}}

\newcommand{\supp}{\operatorname{supp}}

\newcommand{\ve}{{\varepsilon}}
\newcommand{\vv}{{\vspace{2mm}}}
\newcommand{\vvv}{\vspace{4mm}}
\newcommand{\wt}[1]{{\widetilde{#1}}}

\newcommand{\HD}{{\mathsf{HD}}}
\newcommand{\LD}{{\mathsf{LD}}}

\def\Xint#1{\mathchoice
{\XXint\displaystyle\textstyle{#1}}%
{\XXint\textstyle\scriptstyle{#1}}%
{\XXint\scriptstyle\scriptscriptstyle{#1}}%
{\XXint\scriptscriptstyle\scriptscriptstyle{#1}}%
\!\int}
\def\XXint#1#2#3{{\setbox0=\hbox{$#1{#2#3}{\int}$ }
\vcenter{\hbox{$#2#3$ }}\kern-.58\wd0}}

\def\avint{\Xint-}

\newtheorem{theorem}{Theorem}[section]
\newtheorem{lemma}[theorem]{Lemma}
\newtheorem{remark}[theorem]{Remark}

\newtheorem*{lemma*}{Lemma}
\newtheorem*{theorem*}{Theorem}

 \newtheorem{main}{Theorem}

\theoremstyle{definition}

\theoremstyle{remark}
\newtheorem{rem}[theorem]{\bf Remark}

\def\ND{{\rm ND}}
\def\LD{{\rm LD}}
\def\HD{{\rm HD}}
\def\BA{{\rm BA}}

\numberwithin{equation}{section}

\newcommand{\cnj}[1]{\overline{#1}}
\newcommand{\RRem}{\begin{rem}}
\newcommand{\erem}{\end{rem}}

\newcommand{\xavi}[1]{}%\marginpar{\Bl \scriptsize \textbf{Xa:} #1}}

\makeatletter
\def\@tocline#1#2#3#4#5#6#7{\relax
  \ifnum #1>\c@tocdepth % then omit
  \else
    \par \addpenalty\@secpenalty\addvspace{#2}%
    \begingroup \hyphenpenalty\@M
    \@ifempty{#4}{%
      \@tempdima\csname r@tocindent\number#1\endcsname\relax
    }{%
      \@tempdima#4\relax
    }%
    \parindent\z@ \leftskip#3\relax \advance\leftskip\@tempdima\relax
    \rightskip\@pnumwidth plus4em \parfillskip-\@pnumwidth
    #5\leavevmode\hskip-\@tempdima
      \ifcase #1
       \or\or \hskip 1em \or \hskip 2em \else \hskip 3em \fi%
      #6\nobreak\relax
    \dotfill\hbox to\@pnumwidth{\@tocpagenum{#7}}\par
    \nobreak
    \endgroup
  \fi}
\makeatother

\def\cH{{\mathcal{H}}}

\def\cM{{\mathcal{M}}}

\def\bR{{\mathbb{R}}}

\newcommand{\ps}[1]{\left( #1 \right)}

\newcommand{\ck}[1]{\left\{#1 \right\}}
\newcommand{\av}[1]{\left| #1 \right|}

\newcommand{\isif}[1]{\left\{\begin{array}{cc} #1
\end{array}\right.}

\def\gec{\gtrsim}
\def\lec{\lesssim}

\def\grad{\nabla}

\def\bR{\mathbb{R}}
\def\lec{\lesssim}
\def\Lip{\textrm{Lip}}

\newcommand{\jonas}[1]{}%\marginpar{\color{magenta} \scriptsize \textbf{Jonas:} #1}}

\newcommand{\hdis}[1]{{\mbox{\dist}}_{#1}}

\begin{document}

\title[Characterization of 
rectifiable measures in terms of $\alpha$-numbers]{Characterization of 
rectifiable measures in terms of $\alpha$-numbers}

\author{Jonas Azzam}
\address{School of Mathematics, University of Edinburgh, JCMB, Kings Buildings,
	Mayfield Road, Edinburgh,
	EH9 3JZ, Scotland.}
\email{j.azzam "at" ed.ac.uk}

\author{Xavier Tolsa}
\address{
ICREA, Passeig Llu\'{\i}s Companys 23 08010 Barcelona, Catalonia, and\\
Departament de Matem\`atiques and BGSMath
\\
Universitat Aut\`onoma de Barcelona
\\
08193 Bellaterra, Catalonia
}
\email{xtolsa@mat.uab.cat}

\author{Tatiana Toro}
\address{University of Washington \\
	Department of Mathematics \\
	Seattle, WA 98195-4350, USA}
\email{toro@uw.edu}

\thanks{X.T. was supported by the ERC grant 320501 of the European Research Council, and also partially supported by the grants 2017-SGR-395 (Catalonia), MTM-2016-77635-P and  MDM-2014-044 (MICINN, Spain). TT was partially supported by the Craig McKibben \& Sarah Merner Professor in Mathematics and by NSF grant number DMS-1664867}

\def\Lip{{\rm Lip}}
	\def\lec{\lesssim}
\def\gec{\gtrsim}

\begin{abstract}
	We characterize Radon measures $\mu$ in $\R^{n}$ that are $d$-rectifiable in the sense that their supports are covered up to $\mu$-measure zero by countably many $d$-dimensional Lipschitz graphs and $\mu \ll \HH^{d}$. The characterization is in terms of a Jones function involving the so-called $\alpha$-numbers. This answers a question left open in a former work by Azzam, David, and Toro. 
\end{abstract}

\maketitle

\section{Introduction}

\xavi{Added ``countably many" in the abstract}

A Borel measure $\mu$ in $\R^{n}$ is called {\it $d$-rectifiable} if there are countably many Lipschitz images $\Gamma_{i}$ of $\R^d$ such that 
\begin{equation}\label{e:rectifiable-support}
\mu\biggl(\R^{n}\setminus \bigcup_i \Gamma_{i}\biggr) =0
\end{equation}
and additionally $\mu\ll \cH^{d}$, where $\HH^d$ denotes the $d$-dimensional Hausdorff measure. A set $E$ is called $d$-rectifiable if $\cH^{d}|_{E}$ is a $d$-rectifiable measure.

The goal of this paper is to give sufficient conditions for the $d$-rectifiability of a  Borel measure $\mu$ in the above sense. Such conditions are desirable since rectifiable measures and sets enjoy many useful properties and are ubiquitous in analysis. %For example, the work in \cite{GT16-prep} provides a sufficient condition for the rectifiability of a measure. This criteria has been essential in solving of a two phase problem for harmonic measure while providing a characterization of  mutual absolute continuity of the \emph{complementary} harmonic measures with surface measure in disjoint domains \cite{AMT16,AMTV16}.
\jonas{I removed a citation to harmonic measure here, thought it wasn't too relevant.}
Characterizations of rectifiability usually arise from the study of certain properties that are trivial for the Lebesgue measure in Euclidean space. These properties do not necessarily hold for rectifiable sets and measures except in an approximate way. For example, the property that a measure $\mu$ satisfies $\mu(B(x,r))=r^{d}$ for all $x\in \supp \mu$ and $r>0$ is trivially satisfied by Lebesgue measure, though not for general rectifiable measures. However, the weaker property that $\lim_{r\rightarrow 0} \frac{\mu(B(x,r))}{r^{d}}\in (0,\infty)$ for $\mu$-almost every $x$ is satisfied by rectifiable measures, and this also implies $d$-rectifiability by the amazing work of Preiss \cite{Pr87}. See also \cite{TT15} and \cite{Tol17} for related characterizations in terms of densities. 

In this paper, we will study $d$-rectifiability from the perspective of how well a measure resembles  $d$-dimensional Lebesgue measure at various scales and locations. It is a classical result that if $\mu$ is $d$-rectifiable, then for $\mu$-almost every $x\in \R^{n}$, the measures $\mu_{x,r}$ defined by 
\[
\mu_{x,r}(A) = r^{-d} \mu(rA+x)
\]
converge weakly to a constant times Lebesgue measure restricted to a $d$-dimensional plane (see \cite{DeL08} and \cite{Pr87}). In particular, the distance between these rescaled measures and the class of $d$-dimensionally ``flat" measures tends to zero. We can make this distance more precise as follows. For measures $\mu$ and $\nu$ and an open ball $B$ we define
$$F_B(\sigma,\nu):= \sup\Bigl\{ \Bigl|{\textstyle \int \phi \,d\sigma  -
	\int \phi\,d\nu}\Bigr|:\, \phi\in \Lip_{1}(B) \Bigr\},$$
where 
\[
\Lip_{1}(B)=\{\phi:{\rm Lip}(\phi) \leq1,\,\supp f\subset
B\}
\]
and $\Lip(\phi)$ stands for the Lipschitz constant of $\phi$.

It is easy to check that this is indeed a distance in the space of finite Borel measures supported in the open ball $B$. See [Chapter 14, Ma] for other properties of this distance. In fact, this is a variant of the well known Wasserstein $1$-distance from mass transport theory. 

For a measure $\mu$ and $d\in \N$, we define
\begin{equation}\label{alpha-def}
\alpha_\mu^{d}(B) := \frac1{r_{B}\,\mu(B)}\,\inf_{c\geq0,L} \,F_{B}(\mu,\,c\HH^{d}|_{L}),
\end{equation}
where the infimum is taken over all $c\geq0$ and all $d$-dimensional planes $L$. Also, { if $\mu(B^{\circ})>0$,} we denote by $c_B$ and $L_B$ a constant and a plane such that, if we set
\begin{equation}\label{L-def}
\LL_{B} :=c_B\HH^{d}|_{L_B}, 
\end{equation}
then
\begin{equation}\label{cL-ppts}
\alpha_\mu^{d}(B) =  \frac1{r_{B}\mu(B)}\,F_{B}(\mu,\,\LL_B). %<2\alpha_\mu^{d}(B).
\end{equation}
Let us remark that $c_B$ and $L_B$ (and so $\LL_B$) may be not unique. Moreover, we may (and will) assume that 
$L_B\cap B\neq\varnothing$.
When $B=B(x,r)$, we will also write $\alpha_{\mu}^d(B)=\alpha_{\mu}^d(x,r)$, and $c_{B}=c_{x,r}$. 
{Further we may drop the superindex $d$ quite often, to shorten notation.} \xavi{Added this sentence}

These are the so-called $\alpha$ coefficients from \cite{Tol09}. If $\mu$ is $d$-rectifiable, the convergence of $\mu_{x,r}$ to $d$-dimensional Lebesgue on a $d$-plane as $r\rightarrow 0$ for a.e.\ $x$ implies the weaker property that 
\begin{equation}\label{e:arightarrow0}
\lim_{r\rightarrow 0} \alpha_{\mu}^{d}(x,r)=0 \quad\mbox{for $\mu$-a.e. }x\in \R^{n}. 
\end{equation}
However, this limit being zero is not enough to imply rectifiability. This can be seen by considering a variant of the Von Koch snowflake such that if $K_k$ denotes the $k$-th stage of the construction, $K_{k+1}$ is obtained from $K_k$ by introducing new edges that make an angle equal to  $\frac{1}{\sqrt{k}}$ with the previous edges, and then let $\mu_{k} = [\cH^{1}(K_{k})]^{-1}\cH^{1}|_{K_{k}}$. These measures converge weakly to a measure $\mu$ for which \eqref{e:arightarrow0} holds (with $d=1$) yet the measure is singular with respect to $\cH^{1}$. Thus, it is a natural question to ask what additional information is needed aside from \eqref{e:arightarrow0} to imply rectifiability. 

In \cite{ADT16}, the first author, David and the third author considered \xavi{I changed the wording in this paragraph}
some variant of the $\alpha$ coefficients. Denote $T_{x,r}(y)= (y-x)/r$ and let
$W_{1}$ be the $1$-Wasserstein distance between probability measures and the infimum is taken over all $d$-planes. Then one sets
\begin{equation}
\label{e:tol15}
\wt{\alpha}_{\mu}^{d}(x,r) = \inf_{L} W_{1}\bigl(\mu(B(x,r))^{-1}\,T_{x,r}[\mu] , \,\cH^{d}(L\cap B(0,1))^{-1}\,\HH^{d}|_{L\cap B(0,1)}\bigr),
\end{equation}
where the infimum is taken over all $d$-planes. 
In \cite{ADT16} it was shown that if $\mu$ is doubling and 
\[
\int_{0}^{1} \wt{\alpha}_{\mu}^{d}(x,r)\frac{dr}{r} < \infty \quad\mbox{ for $\mu$-a.e. }x\in \R^{n},
\] 
then $\mu$ is $d$-rectifiable.
In \cite{ADT16} it was also conjectured that the same result should be true if $\wt{\alpha}_{\mu}^{d}(x,r)$ were replaced with $\wt{\alpha}_{\mu}^{d}(x,r)^{2}$. 

In \cite{Orp17}, Orponen showed the conjecture is true for $n=d=1$. In fact, he proved that if $\mu$ and $\nu$ are two Radon measures on the real line (where $\nu$ is doubling) then $\mu\ll \nu$ if $\int_{0}^{1}\wt{\alpha}_{\mu,\nu}^1(x,r)^{2}\frac{dr}{r}<\infty$ holds $\mu$-almost everywhere, where now $\wt{\alpha}_{\mu,\nu}^1$ measures the $1$-Wasserstein distance between $\mu$ and $\nu$, normalized appropriately.

If one assumes absolute continuity a priori, then there are other some related results in the literature. Define the Jones' $\beta$-numbers
\[
\beta_{\mu,p}^{d}(x,r)^{p}  = \inf_{L} \frac{1}{r_{B}^{d}} \int_{B(x,r)}\ps{\frac{\dist(y,L)}{r}}^{p} d\mu(y),
\]
where the infimum is over all $d$-dimensional planes $L$.
In a sense, these coefficients are weaker than the $\alpha$-numbers that we described above since they only measure how close the measure is to lying on a $d$-plane, not how much it resembles $d$-dimensional Lebesgue measure (so for example, if $\mu$ is supported in a plane but not supported on a portion inside the ball $B(x,r)$ with positive area in this plane, then the $\beta$-number of $B(x,r)$ is zero while the $\alpha$-number is positive). If $\mu\ll\HH^d|E$ for some set $E$ of finite $\HH^d$-measure, it has been shown recently by the first and second authors \cite{AT15} that $\mu$ is rectifiable if 
\begin{equation}\label{e:Bjones}
\int_{0}^{1}\beta_{\mu,2}^{d}(x,r)^{2} \frac{dr}{r} d\mu(x)<\infty,
\end{equation}
for $\mu$-almost every $x\in \R^{n}$. More recently, Edelen, Naber, and Valtorta \cite{ENV16} have obtained a related 
result of more quantitative nature. 

\xavi{I changed a little the wording in this paragraph}
The converse to the result obtained in \cite{AT15} also holds, as shown by the second author \cite{Tol15-I}. That is, if $\mu$ is $d$-rectififable, then \eqref{e:Bjones} holds. 
Further in the same work it is shown that if $\mu$ is $d$-rectififable, then \eqref{e:Bjones} is satisfied  with $\beta_{\mu,2}^{d}(x,r)$ replaced by $\alpha_\mu^d(x,r)$.
This fact motivated the above conjecture about the characterization of rectifiability in terms of the $\alpha$-numbers. 

In this paper, we confirm this conjecture for measures that are pointwise doubling. More precisely, we prove the following:

\begin{main}\label{thmi}
	Let $\mu$ be a Radon measure in $\R^{n}$, $0<d\leq n$, and $E$ a Borel set with $\mu(E)>0$ such that 
	\begin{equation}\label{e:main}
	J_{\alpha,1}(x):=\int_{0}^{1} \alpha_{\mu}^{d}(x,r)^{2}\frac{dr}{r} <\infty \mbox{ for all }x\in E
	\end{equation}
	and 
	\begin{equation}\label{e:asymdub}
	\limsup_{r\rightarrow 0}\frac{\mu(B(x,2r))}{\mu(B(x,r))} <\infty \mbox{ for all }x\in E.
	\end{equation}
	Then $\mu|_{E}$ is $d$-rectifiable. 
\end{main}

As stated above, in \cite{Tol15-I} it is shown that if $\mu$ is any $d$-rectifiable measure (not necessarily doubling), then \eqref{e:main} holds. Thus combining this result with  the theorem above we obtain a characterization of rectifiable measures in terms of their $\alpha$-coefficients and the doubling condition
\rf{e:asymdub}. 

It is not hard to see using the definition of Wasserstein distance that 
$\alpha_{\mu}^d(x,r) \leq \wt{\alpha}_{\mu}^d(x,r)$, and so Theorem \ref{thmi} implies the conjecture from \cite{ADT16} for measures satisfying \rf{e:asymdub}.

The doubling condition \eqref{e:asymdub} is necessary as shown by the following result.

\begin{main}\label{teocount1}
	There exists a  Radon measure $\mu$ in $\R^2$ which satisfies
	\begin{equation*}
	\int_0^1\alpha_\mu^1(x,r)^2\,\frac{dr}r<\infty\quad \mbox{for all $x\in \supp\mu$},
	\end{equation*}
	and such that 
	$$\lim_{r\to 0}\frac{\mu(B(x,r))}r =0\quad\mbox{ for all $x\in\supp\mu$.}$$
	In particular, $\mu$ is not $1$-rectifiable. 
\end{main}

\xavi{I added this paragraph}
We remark that a related phenomenon occurs for the $\beta_p$ coefficients when $p<2$ in the absence of doubling conditions. Indeed, it is
has been shown recently in \cite{Tolsa-betap} that there exists a set $E\subset \R^2$ with $\HH^1(E)<\infty$ which is not $1$-rectifiable and such that, for all $1\leq p<2$, 
\begin{equation*}
	\int_0^1\beta_{\HH^1|_E,p}^1(x,r)^2\,\frac{dr}r<\infty\quad \mbox{ for $\HH^1$-a.e. $x\in E$}.
	\end{equation*}
On the other hand, by a result due to Pajot \cite{Pajot97} it follows that, for all $p\in[1,2]$, the above condition implies the  rectifiability of
$E$ under the additional assumption that $$\liminf_{r\to0} \frac{\HH^1(E\cap B(x,r))}r>0
\quad \mbox{ for $\HH^1$-a.e. $x\in E$}, $$ which is stronger than the pointwise doubling assumption \rf{e:asymdub} (for  $\mu=\HH^1|_E$ with $\HH^1(E)<\infty$).

We should also mention that there are results that provide necessary and/or sufficient conditions for 
a different notion of rectifiability of measures introduced by Federer. This notion of rectifiability only asks that condition \rf{e:rectifiable-support} hold, and does not require the absolute continuity with respect to $\HH^d$. The charaterization of Federer rectifiability is a more difficult
problem. Part of the interest in this topic was motivated by an example of Garnett, Killip, and Schul \cite{GKS10} of a doubling measure $\mu$ with $\supp \mu=\R^{2}$ that satisfies \eqref{e:rectifiable-support}. This was a surprising result since doubling measures are considered to be well-behaved apart from possibly being singular, so it was anticipated that, if a doubling measure has support equal to $\R^{2}$, then it should give zero measure to any rectifiable curve. 
Later on  Badger and Schul \cite{BS16}  characterized the measures in Euclidean space that can be covered up to measure zero by Lipschitz curves, assuming a positive lower density condition on the measure.  Also, the first author and Mourgoglou showed in \cite{AM16}
that if a measure is doubling with connected support and positive $1$-dimensional lower density, then it is $1$-rectifiable. Previously, in \cite{Ler03}, Lerman gave sufficient conditions for $1$-rectifiability in terms of $\beta$-type numbers without any lower density assumption. Thus far, the most general necessary conditions for this kind of rectifiability using $\beta$-type numbers is given in \cite{BS15}. Unfortunately, these necessary conditions are not sufficient, as shown by an example of Martikainen and Orponen \cite{MO16}. However, see \cite{BS17} for a characterization for measures with positive lower density using a different $\beta$-type quantity.

\vv
% *******************************************************************************************

\section{Notation}\label{sec:notation}
 We will write $A\lesssim B$ if $A\leq CB$ for some universal constant $C$. Throughout this paper, we will assume all such implicit constants depend on the dimension $n$; otherwise, we will write $A\lesssim_{t}B$ if the constant $C$ depends on some parameter $t$. We will write $A\approx B$ to mean $A\lesssim B\lesssim A$ and define $A\approx_{t}B$ similarly.
 
We denote by $B(x,r)$  the open ball centered at $x$ of radius $r>0$ in $\R^n$. If $B$ is a ball, we write $x_{B}$ for its center and $r_{B}$ for its radius. If $B=B(x,r)$ and $\lambda>0$, we will write
\[
\lambda B= B(x,\lambda r),
\]
that is, the ball with same center but $\lambda$-times the radius. 

For a measure $\mu$ in $\R^n$ and a ball $B=B(x,r)$, we write
\[
\Theta_{\mu}^{d}(x,r) = \Theta_{\mu}^{d}(B) = \frac{\mu(B(x,r))}{r^{d}}=\frac{\mu(B)}{r_{B}^{d}} . \]
Given $E,F\subset \R^n$ closed sets, $d_H(E,F)$ stands for the Hausdorff distance between $E$ and $F$. For $x\in
\R^n$ and $r>0$ we also consider the following local scale invariant version of Hausdorff distance
$$\hdis{x,r}(E,F) = \frac1r\,\max\Bigl(\,\sup_{y\in E\cap B(x,r)}\dist(y,F); \sup_{y\in F\cap B(x,r)}\dist(y,E)\Bigr).$$

Given two $d$-planes $L_{1}$ and $L_{2}$,  let $L_{1}', L_2'$ be the respective parallel $d$-planes passing through the origin.
Then we denote
\[
\angle(L_{1},L_{2})=\dist_{H}(L_{1}'\cap B(0,1),L_{2}'\cap B(0,1)).
\]
In a sense, $\angle(L_{1},L_{2})$ is the angle between $L_1$ and $L_2$.

\vv 

% **********************************************************************************************************************

\section{Preliminaries}

%{\color{blue} Note that the properties of what is called later a Besicovitch covering are stated in more detail and the name is given}

Below we use constants $A,\tau,C_{1}$, and $\ve>0$. We choose them so that 
\begin{equation}\label{constants-1}
\tau\ll1\ll \min\{A,C_1\}\hbox{      and      }
\ve\ll\min\{A^{-1},\tau^4,C_{1}^{-1}\},
\end{equation}

We recall  Besicovitch covering lemma as we will use it frequently. There exists $N=N(n)$ depending only on $n$ such that for any bounded set $E\subset \R^n$, and any collection of closed balls 
$\{B(x,r(x)): x\in E\}$ with $\sup\{r(x):x\in E\}<\infty$ there are 
$\mathcal G_1\cdots, \mathcal G_N$ countable disjoint subcollections 
 such that
\begin{equation}\label{besicovitch-1}
E\subset \bigcup_{j=1}^N\bigcup_{B\in \mathcal G_j}B \quad\hbox{   consequently   }\quad
\chi_{E}\leq \sum_{j=1}^N \chi_{\mathcal B_j}\lec_{n} 1,
\end{equation}
where $\mathcal B_j=\bigcup_{B\in\mathcal G_j}B$.
In particular, for a measure $\mu$, there is $j_0\in\{1,\cdots, N\}$ such that 
\begin{equation}\label{e:besicovitch}
\mu(E) \leq \sum_{j}\mu(\mathcal B_{j}) \le N \mu( \mathcal B_{j_0})= N \mu\biggl(\bigcup_{B\in \mathcal B_{j_0}}B\biggr)
\end{equation}
Such covering will often be referred to as a Besicovitch subcovering of the collection $\{B(x,r(x)): x\in E\}$.

We now go over some basic facts about $\alpha$ numbers. Some of them are proven in \cite{Tol09} for $d$-AD-regular measures and in \cite{Tol17} for general measures. However we supply some more precise estimates here.

\begin{lemma}\label{l:alpha-monitone}
	For $x,y\in \R^{n}$, if $B(x,r)\subset B(y,s)$, then
	
	\begin{equation}\label{e:alpha-monitone-general}
	\alpha_{\mu}(x,r)\leq  \frac{s\,\mu(B(y,s))}{r\,\mu(B(x,r))}\,\alpha_{\mu}(y,s).
	\end{equation}
\end{lemma}

%{\color{blue} Please note that below many x's were changed to y's and r's to s's in this proof}
\begin{proof}
	Let  $\ve>0$ and pick $\LL_{y,s}=c\HH^{d}|_{L}$ so that 
	\[
	\frac{1}{s\mu(B(y,s))}F_{B(y,s)}(\mu,\LL_{y,s})\leq(1+\ve)\alpha_{\mu}(y,s).
	\]
	For $\phi\in \Lip_{1}(B(x,r))$
	\[ 
	\av{\int\phi \,d\mu - \int \phi\, d\LL_{y,s}}
	\leq s\mu(B(y,s)) (1+\ve)\alpha_{\mu}(y,s).
	\]
	Taking the supremum over $\phi\in \Lip_{1}(B(x,r))$ and using \eqref{alpha-def} we have
	\[
	\alpha(x,r)r\mu(B(x,r))\le F_{B(x,r)}(\mu, \LL_{y,s})\leq(1+\ve)s\mu(B(y,s))\alpha_{\mu}(y,s).
	\]
	Hence,
	\[
	\alpha_{\mu}(x,r) 
	\leq \frac{s\mu(B(y,s))}{r\mu(B(x,r))}\,(1+\ve) \,\alpha_{\mu}(y,s)
	\]
	and letting $\ve\rightarrow 0$ we obtain \eqref{e:alpha-monitone-general}.
\end{proof}

\begin{lemma} \label{l:LcapB} For $x\in \mathbb{R}^{n}$, if $y\in B(x,r/2)$, $B(y,2s)\subset B(x,r)$, $\LL$ is a measure supported on a $d$-plane $L$, and {$F_{B(x,r)}(\mu,\LL)<{s\,\mu(B(y,s))}$,} then
	\begin{equation}
	\label{e:LcapB-general} 
	L\cap B(y,2s)\neq\varnothing.
	\end{equation} 
	In particular, if, $\alpha_{\mu}(x,r)<\frac{\mu(B(x,r/8))}{8\,\mu(B(x,r))}$, then 
\[
	L_{x,r}\cap B(x,r/4)\neq\varnothing.
\]
\end{lemma}

{
\begin{proof}
	Let $\phi(z) = (2s-|y-z|)_+$. Note that $\phi\in\Lip_{1}(B(x,r))$
	and $\phi\geq s$ on $B(y,s)$. If $L\cap B(y,2s)=\varnothing$, then
\begin{equation*}
	s\,\mu(B(y,s)) 
	\leq \int \phi\, d\mu \leq \left|\int \phi \,d(\mu-\LL)\right| 
	\leq
	F_{B(x,r)}(\mu,\LL)
	<{s}\,\mu(B(y,s)),
\end{equation*}
which is a contradiction. Thus, $\dist(y,L)<2s$.
%	$tr=1-\dist(y,L)$, and $z\in L\cap B(x,r)$ be closest to $x$. 
%	First assume that $t\leq 3/4$. Let 
%	\[
%	\phi(y) =  ((1-t)r-|x|)_{+}
%	\]
%	Then $\phi\in \Lip_{1}(B(x,r))$ and $\phi \geq r/8$ on $B(x,r/8)$. Hence,
%	\[
%	\frac{r}{8} \mu(B(x,r/8))
%	\leq \int \phi d\mu
%	\leq r\mu(B(x,r)) \alpha_{\mu}(x,r) + \int \phi d\LL_{x,r}
%	<\frac{r}{8} \mu(B(x,r/8)).
%	\]
%	which is a contradiction. Thus, $t>3/4$ and so $\dist(x,L)<1/4$, and this proves the lemma. 
%	
\end{proof}
}

\begin{lemma}\label{l:cb-theta}
	For $x\in \mathbb{R}^{n}$, if $\LL=c\cH^{d}|_{L}$ and { $F_{B(x,r)}(\mu,\LL)<\frac r8\,\mu(B(y,\frac r8))$,} then
	\begin{equation}\label{e:cb-theta-gen}
	\Theta^d_{\mu}(x,r/2) \lec  c \lec \Theta^d_{\mu}(x,r).
	\end{equation}
	In particular, if $\alpha_{\mu}(x,r) < \frac{\mu(B(x,r/8))}{8\,\mu(B(x,r))}$, then 
	\begin{equation}\label{e:cb-theta}
		 \Theta^d_{\mu}(x,r/2) \lec  c_{x,r}:=c_{B(x,r)} \lec \Theta^d_{\mu}(x,r).
	\end{equation}
\end{lemma}

\begin{proof}
	Let $\phi(x) = (r-|x-y|)_+$, so that
	 $\phi\in \Lip_{1}(B(x,r))$ and $\phi\geq r/2$ on $B(x,r/2)$. Since $L\cap B(x,r/4)\neq\varnothing$ by the previous lemma, we have		
	\[
	c\,r^{d+1}
	\approx r\LL(B(x,r/2))
	\lec \int \phi\, d\LL
	\leq F_{B(x,r)}(\mu,\LL) +  \int \phi \,d\mu
	\leq 2r \mu(B(x,r))
	\]
	and hence $c \lesssim \Theta^d_{\mu}(x,r)$. A similar computation reversing the roles of $\mu$ and $\LL$ yields $c \gec  \Theta^d_{\mu}(x,r/2)$. 
\end{proof}

%{\color{blue} Below a number of s's were changed to r/2 ,  hypothesis changed $B(x,2r)\subset B(y,s)$ to make first case work. Added something where to the "geometry" part of the argument please make sure this is 
%what you had in mind.I also added some estimates about the angles to be able to quote them later}

\begin{lemma}\label{l:angle-general}
	Let $x,y\in \mathbb{R}^{n}$, $B(x,2r)\subset B(y,s)$, and $\LL_{i}= c_{i}\HH^{d}|_{L_{i}}$ for $i=1,2$.
	If  $F_{B(x,r)}(\mu,\LL_{1})< \frac r8\,\mu(B(x,r/8))$, and $L_2\cap B(y,s)\neq\varnothing$, then
		\begin{equation}\label{e:angle-general1}
	\angle(L_1,L_2)+	\hdis{x,r/2}(L_1,L_2)\lec \frac{ F_{B(x,r)}(\mu,\LL_{1}) + F_{B(y,s)}(\mu,\LL_{2})}{{ r\,\mu(B(x,r/2))}}.
	\end{equation}
	 In particular, if $\alpha_{\mu}(x,r)<\frac{\mu(B(x,r/8))}{8\,\mu(B(x,r))}$, then
	\begin{equation}\label{e:angle-general}
	\angle(L_{x,r},L_{y,s})+ \dist_{x,r/2}(L_{x,r},L_{y,s})\lec \frac{s\,\mu(B(y,s))}{{ r\,\mu(B(x,r/2))}}\, \alpha_{\mu}(y,s).
	\end{equation}
\end{lemma}

\begin{proof}
Suppose first that $L_2\cap B(x,2r)= \varnothing$. Let $\phi_0(z)=(2r-|x-z|)_+$.
Then we have
$$F_{B(y,s)}(\mu,\LL_{2})\geq \int \phi_0\,d(\mu-\LL_2)= \int \phi_0\,d\mu\geq r\,\mu(B(x,r)).$$
It is also immediate that 
$\dist_{x,r/2}(L_1,L_2)\lec 1$, and thus
$$\angle(L_1,L_2)+\dist_{x,r/2}(L_1,L_2)\lec \frac{ F_{B(y,s)}(\mu,\LL_{2})}{{r\, \mu(B(x,r))}},$$
and so \eqref{e:angle-general1} holds in this case.

Suppose now that $L_2\cap B(x,2r)\neq \varnothing$.
	Let $\Phi$ be a $\frac{2}{r}$-Lipschitz function that equals $1$ on $B(x,r/2)$ and $0$ outside $B(x.r)$. Also set 
	\[
	\phi(z) = \Phi (z)\cdot  \dist(z,{ L_2}).
	\]
	Using that  $\dist(z,{ L_2})\leq3r$ on $\supp\Phi$,
	it is immediate to check that $\phi$ is { $7$-Lipschitz} on $B(x,r)$. 
	By Lemma \ref{l:LcapB}, {$L_1\cap B(x,r/4)\neq\varnothing$}, and so $\cH^{d}({ L_1}\cap B(x,r/2))\approx r^{d}$. Thus, using that $\phi$ vanishes on $L_2$,
	\begin{align}\label{int-est}
	\avint_{B(x,r/2)} \frac{\dist(z,L_{2})}{r} \,d\cH^{d}|_{L_{1}}
	& \stackrel{\eqref{e:cb-theta-gen}}{\lec}  \frac1{\Theta^d_{\mu}(x,r/2)\, r^{d+1}}\int \phi \,d\LL_{1} 
	= \frac{1}{r\mu(B(x,r/2))}\int \phi \,d\LL_{1} \\
	& \lec \frac{F_{B(x,r)}(\mu,\LL_{1})}{r\,\mu(B(x,r/2))}
	+ \frac{1}{r\,\mu(B(x,r/2))} \int  \phi \,d(\mu-\LL_2) \nonumber\\
	& \lesssim \frac{ F_{B(x,r)}(\mu,\LL_{1}) + F_{B(y,s)}(\mu,\LL_{2})}{r\,\mu(B(x,r/2))}.\nonumber
	\end{align}
	
	Let $z_0\in B(x,r/2)\cap L_1$ be such that $\dist(z_0,L_2)=\inf\{\dist(z,L_2):z\in B(x,r/2)\cap L_1$. Then since $L_1$ and $L_2$ are $d$-planes, for $z\in B(x,r/2)\cap L_1$ we have
	\begin{equation}\label{dist-1}
	\dist(z,L_2)= \dist(z_0,L_2) + \dist(z-z_0, L_2-z_0)=\dist(z_0,L_2) + |z-z_0|\angle(L_1,L_2)
		\end{equation}
	Integrating \eqref{dist-1} over $B(x,r/2)\cap  L_1$ and using \eqref{int-est} we obtain that
	\begin{equation}\label{dist-2}
\dist(z_0,L_2) + r\angle(L_1,L_2)\lesssim \frac{ F_{B(x,r)}(\mu,\LL_{1}) + F_{B(y,s)}(\mu,\LL_{2})}{r\,\mu(B(x,r/2))}.
\end{equation}
Thus \eqref{dist-1} and \eqref{dist-2} yield
	\begin{eqnarray}\label{dist-3}
	\frac{1}{r}\sup\{\dist(z,L_2):z\in B(x,r/2)\cap L_1\}&\le &\frac{1}{r}\dist(z_0,L_2) + 2\angle(L_1,L_2)\nonumber\\
	&\lesssim &\frac{ F_{B(x,r)}(\mu,\LL_{1}) + F_{B(y,s)}(\mu,\LL_{2})}{\mu(B(x,r/2))},
	\end{eqnarray}
	and 
	\begin{equation}\label{angle-1}
	\angle(L_1,L_2)\lesssim \frac{ F_{B(x,r)}(\mu,\LL_{1}) + F_{B(y,s)}(\mu,\LL_{2})}{r\,\mu(B(x,r/2))}.
	\end{equation}
	Since $L_1$ and $L_2$ are planes this is enough to conclude \eqref{e:angle-general1}.
	 Now \eqref{e:angle-general} follows from \eqref{e:angle-general1} and \eqref{e:alpha-monitone-general}, by taking $L_1= L_{x,r}$ and $L_2=L_{y,s}$.
Indeed,	we derive 
	\begin{align}\label{e:3894}
	F_{B(x,r)}(\mu,\LL_{x,r}) + F_{B(y,s)}(\mu,\LL_{y,s}) &\lec \alpha_\mu(x,r) \,r\,\mu(B(x,r))
	+ \alpha_\mu(y,s) \,s\,\mu(B(y,s))\\
	& \lec %\frac{s\mu(B(y,s))}{r\mu(B(x,r))}\,\alpha_{\mu}(y,s)\, \,r\,\mu(B(x,r))+
	 \alpha_\mu(y,s) \,s\,\mu(B(y,s)).\nonumber
	\end{align}
	Plugging this estimate into \rf{e:angle-general1}, we obtain \rf{e:angle-general}.	
\end{proof}

\begin{lemma}\label{l:cx-cy-general}
	Let $x,y\in \mathbb{R}^{n}$ be such that $B(x,2r)\subset B(y,s)$, $\LL_{i}= c_{i}\HH^{d}|_{L_{i}}$,  { $F_{B(x,r)}(\mu,\LL_{1})< \frac r8\,\mu(B(x,\frac r8))$ and  { $F_{B(y,s)}(\mu,\LL_{2})< \frac s8\,\mu(B(y,\frac s8))$}}. Then
	\begin{equation}
	\label{e:cx-cy-general1}
	{ |c_{1}-c_{2}| \lec  \frac{ F_{B(x,r)}(\mu,\LL_{1}) + F_{B(y,s)}(\mu,\LL_{2})}{r^{d+1}} \biggl(1+ \frac{\Theta^d_{\mu}(y,s)}{\Theta^d_{\mu}(x,r/2)}\biggr)\frac sr.}
	\end{equation}
	In particular, if  $\alpha_{\mu}(x,r)<\frac{\mu(B(x,r/8))}{8\,\mu(B(x,r))}$
	and $\alpha_{\mu}(y,s)<\frac{\mu(B(y,s/8))}{8\,\mu(B(y,s))}$, then
	\begin{equation}
	\label{e:cx-cy-general}
	{ |c_{x,r}-c_{y,s}| \lec  \alpha_\mu(y,s)\,\Theta^d_\mu(y,s) \biggl(1+ \frac{\Theta^d_{\mu}(y,s)}{\Theta^d_{\mu}(x,r/2)}\biggr)\,\frac{ s^{d+2}}{r^{d+2}}.}
	\end{equation}
\end{lemma}

\begin{proof}
	Let $\phi(z) =(r-|x-z|)_{+}$. Then, by \eqref{e:angle-general1} and
	\eqref{e:cb-theta-gen}, since $2r\le s$
	\begin{align*}
	r^{d+1} |c_{1} -c_{2} |
	&  \lec \av{\int \phi \,c_{1}\, d\cH^{d}|_{L_{1}}  -\int \phi\, c_{2}\, d\cH^{d}|_{L_{1}} }\\
	& \leq \av{\int \phi \,c_{1} d\cH^{d}|_{L_{1}}  -\int \phi \,d\mu}
	+\av{\int \phi \,d\mu -  \int \phi \,c_{2}\, d\cH^{d}|_{L_{2}} }\\
	& \qquad		+c_{2} \av{\int \phi \, d\cH^{d}|_{L_{2}}  -\int \phi\, d\cH^{d}|_{L_{1}} }\\
	& \lec F_{B(x,r)}(\mu,\LL_{1}) + F_{B(y,s)}(\mu,\LL_{2}) \\
	&\quad + 
	\Theta^d_\mu(y,s)\, \frac{ F_{B(x,r)}(\mu,\LL_{1}) + F_{B(y,s)}(\mu,\LL_{2})}{\mu(B(x,r))} \,r^{d-1}\,s\\
	& \lesssim \Bigl( F_{B(x,r)}(\mu,\LL_{1}) + F_{B(y,s)}(\mu,\LL_{2})\Bigr) \biggl(1+ \frac{\Theta^d_{\mu}(y,s)}{\Theta^d_{\mu}(x,r)}\biggr)\frac sr,
%	c_{y,s} \alpha_{\mu}(s,r) s^{d+1}\\
%	& \stackrel{\eqref{e:alpha-monitone}}{\lec} \alpha_{\mu}(y,s)s^{d+1}.
	\end{align*}
	which yields \rf{e:cx-cy-general1}.
	
	To get \rf{e:cx-cy-general} we apply \rf{e:3894} using the fact $B(x,2r)\subset B(y,s)$and then we obtain
	\begin{align}
	\frac{ F_{B(x,r)}(\mu,\LL_{x,r}) + F_{B(y,s)}(\mu,\LL_{y,s})}{r^{d+1}} 
		& \leq %\frac{s\mu(B(y,s))}{r\mu(B(x,r))}\,\alpha_{\mu}(y,s)\, \,r\,\mu(B(x,r))+
	 \alpha_\mu(y,s) \,\frac{s\,\mu(B(y,s))}{r^{d+1}} = \alpha_\mu(y,s) \,\Theta^d_\mu(B(y,s))\,\frac{s^{d+1}}{r^{d+1}}
	 .\nonumber
	\end{align}
	Plugging this estimate into \rf{e:cx-cy-general1}, we derive \rf{e:cx-cy-general}.	
\end{proof}

\vv

% *******************************************************************************************
%{\color{blue} Note that the section ``Outline of the proof " has moved to after the preliminaries. The explanation of how we get to the properties of $E_0$ used in section 5 appears here in detail also}
\section{Outline of proof}

\xavi{In this section I made some cosmetic changes: added commas, separated or joining equations, corrected a typo...}
In order to present an outline of the proof to Theorem \ref{thmi} we first explore the consequences of the hypotheses. Consider a Radon measure $\mu$ and a Borel set $E$, with $\mu(E)>0$ and satisfying 
\rf{e:main} and \rf{e:asymdub}. Let $E_1=E\cap B(0,R)$ with $R$ large enough so $0<\mu(E_1)<\infty$. By  \rf{e:asymdub} for $M>1$ large enough there exists a closed set $\wt E\subset E_1$ such that $\mu(\wt E)>0$ and for all
$x\in \wt E$
%\begin{align}\label{limit-hyp}
$$	\lim_{k\to\infty}\sup_{0<r<2^{-k}}\frac{\mu(B(x,2r))}{\mu(B(x,r))} \le \frac{M}{2}\quad \text{ and }\quad
\lim_{k\to\infty}\int_{0}^{2^{-k}} \alpha_{\mu}^{d}(x,r)^{2}\frac{dr}{r}=0.$$
By Egoroff, there exists a closed set $\wt E_0\subset \wt E$, with $\mu(\wt E_0)\ge \frac{9}{10}\mu(\wt E)>0$ so that for $\ve\in (0,10^{-3})$ there is $k_0=k_0(M,\ve)>1$ so that for $k\ge k_0$ and $x\in \wt E_0$
\begin{align}\label{quantified-hyp}
	\sup_{0<r<2^{-k}}\frac{\mu(B(x,2r))}{\mu(B(x,r))} \le M\quad \text{ and }\quad
\int_{0}^{2^{-k}} \alpha_{\mu}^{d}(x,r)^{2}\frac{dr}{r}<\ve^2.
	\end{align}
Since $0<\mu(\wt E_0)<\infty$, for $\mu$-a.e. $x\in \wt E_0$ (\cite[Corollary 2.14]{Mattila}),
\begin{equation}\label{rel-density}
\lim_{r\to 0}\frac{\mu(B(x,r)\cap \wt E_0)}{\mu(B(x,r))}=1.
\end{equation}	
By Egoroff, once again, given $\delta\in (0,\frac{1}{10})$ there exists a closed set $\wt F_0\subset \wt E_0$, with $\mu(\wt F_0)\ge (1-\delta)\mu(\wt E_0)\ge\frac{81}{100}\mu(\wt E)>0$ so that for $\ve\in (0,10^{-3})$ there is $k_1=k(\ve,\delta)>1$ so that 
for $r<2^{-k_1}$ and $x\in \wt F_0$
\begin{equation}\label{quant-rel-den}
\mu(B(x,r)\backslash \wt E_0)\le \ve\mu(B(x,r)).
\end{equation}
%{\color{blue} Note that this property allows us  to assume that all points considered are of high density that is ND is empty}

Summarizing we have that given $M>1$ large enough, $\delta\in (0,\frac{1}{10})$ and $\ve\in (0,10^{-3})$ there exist closed sets $\wt F_0\subset \wt E_0\subset E\cap B(0,R)$ and $\rho_o>0$ such that  $\mu(\wt F_0)\ge(1-\delta)\mu(\wt E_0)>0$ and for 
every $x\in \wt E_0$ and every $0<r<\rho_o$ (see \rf{quantified-hyp})  
\begin{align}\label{working-hyp-1}
\mu(B(x,2r))\le M\mu(B(x,r)),
\end{align}
\begin{align}
&J_{\alpha, {\rho_o}}(x):=\int_{0}^{\rho_o} \alpha_{\mu}^{d}(x,r)^{2}\frac{dr}{r}<\ve^2,\label{working-hyp-2}
\end{align}
and for every $x\in \wt F_0$  and every $0<r<\rho_o$ (see  \rf{quant-rel-den})
\begin{equation}\label{working-hyp-3}
\mu(B(x,r)\backslash \wt E_0)\le \ve\mu(B(x,r)).
\end{equation}

Without loss of generality we may assume that $0\in \wt F_0$. Moreover note that if $\wt\mu_{r}(A) =\mu(B(0,r))^{-1}\mu(r A)$ and $c>0$ then for $y=\frac{x}{r}$ with $x\in \wt E_0$
 \begin{equation}\label{rescale-mu}
 \alpha_{c\wt\mu_r}^{d}(y,s)=\alpha_{\mu}^d(x,sr).
\end{equation}
Letting $\rho_o=4C_1r_0$ where $C_1$ is as in \rf{constants-1} and replacing $\wt E_0$ by $E_0=\frac{1}{r_0}\wt E_0$, $\wt F_0$ by $F_0=\frac{1}{r_0}\wt E_0$, and $\mu$ by 
$\mu(B(0,3C_1r_0))^{-1}\wt\mu_{r_0}$ and relabeling it $\mu$ we have that $0\in F_0$
\begin{equation}\label{meas=1}
\mu(B(0,3C_1))=1,
\end{equation}
and for given $M>1$ large enough, $\delta\in (0,\frac{1}{10})$, and $\ve\in (0,10^{-3})$ there exist closed bounded sets $ F_0\subset  E_0\subset \frac{1}{r_0} E$  such that  $\mu(F_0)\ge(1-\delta)\mu( E_0)>0$ and for 
every $x\in  E_0$ and every $0<r<4C_1$ (see \rf{working-hyp-1} and \rf{working-hyp-2})  
\begin{align}\label{working-hyp-1A}
&\mu(B(x,2r))\le M\mu(B(x,r)),\end{align}
\begin{align}
&J_{\alpha, {4C_1}}(x):=\int_{0}^{4C_1} \alpha_{\mu}^{d}(x,r)^{2}\frac{dr}{r}<\ve^2,\label{working-hyp-2A}
\end{align}
and for every $x\in  F_0$ and every $0<r<4C_1$ (see \rf{working-hyp-3})
\begin{equation}\label{working-hyp-3A}
\mu(B(x,r)\backslash E_0)\le \ve\mu(B(x,r)).
\end{equation}
%{\color{blue} I have the impression that \rf{working-hyp-3A} which is satisfied in $F_0$ with $\mu(F_0)\ge(1-\delta)\mu( E_0)>0$ should allow us to disregard the set ND as we already know it is small, the issue is that 
%we would need to adjust it to be small in the ball $B_0$. I was not able to do it easily so I am leaving things as they are}
\jonas{That could be the case, but yeah, probably best to not tinker too much with the machine at this point}

Note that \rf{working-hyp-2A}  ensures that for $x\in E_0$ and $r\in (0, 2C_1)$ there exits $t\in [r,2r]$ such that 
\begin{equation}\label{pointwise-a-bound-1}
\alpha_{\mu}^{d}(x,t)^{2}\le 2\int_{r}^{2r} \alpha_{\mu}^{d}(x,s)^{2}\frac{ds}{s}<2\,\ve^2.
\end{equation}
Then \rf{pointwise-a-bound-1} and \rf{working-hyp-1} combined with \rf{e:alpha-monitone-general} ensure that for $x\in E_0$ and $r\in (0, 2C_1)$ 
\begin{equation}\label{pointwise-a-bound-2}
\alpha_{\mu}^{d}(x,r)\le \frac{t}{r}\frac{\mu(B(x,t))}{\mu(B(x,r))}\alpha_{\mu}^{d}(x,t)\le 2\frac{\mu(B(x,2r))}{\mu(B(x,r))}\sqrt{2}\ve\le 4M\ve.
\end{equation}

Now we outline the plan for the rest of the proof: Note that on the set $E_0$ the rescaled measure $\mu$ is doubling on a range of scales \rf{working-hyp-1A}, the Jones function $J_\alpha$ and $\alpha$-numbers corresponding to 
$\mu$ and also small (see \rf{working-hyp-2A} and \rf{pointwise-a-bound-1}). For each point in $E_0$ we consider the supremum over all radii, less than a fraction of $4C_1$, 
for which $\mu$ does not behave like an Ahlfors regular 
measure above these scales. Hence, for most of these scales, the measure is either too large or too small. Our goal is to show that the subset of $E_0$ for which this supremum is not 0, is small.
To do this we use techniques from \cite{DT12}, to build a Lipschitz graph  which approximates $\supp\mu$ at every good scale and location. Upon this graph, we 
construct a projection $\nu$ of the measure $\mu$. The nice estimates on the $\alpha$-numbers for $\mu$, yield even nicer estimates for $\nu$. The advantage now is that we have a surrogate $\nu$ for $
\mu$, supported on a graph, and $\nu$ is Ahlfors regular (see \rf{e:nu-AD}). 
To estimate the set where the density drops at small scales we use techniques that come from \cite{Leger} and which were also used in other works, such as \cite{AT15}. To control the measure of the set where the density increases too much at small scales we use our $\alpha$-number estimates to estimate the $L^{2}$-norm of the density of $\nu$ into the domain of the graph (that is, $\R^{d}$). 
 This idea is newer and comes from \cite{Tol17}.
  Altogether, these techniques give us control on the total mass of the area where $\nu$ (and thus $\mu$) can have low or high  density with respect to surface measure. This will show that in fact in most places the density of $\mu$ stays bounded away from $0$ and $\infty$, implying absolute continuity with respect $d$-dimensional Hausdorff measure and rectifiability. It is important to note that this argument proves 
  the rectifiability if a rescaled version of $\mu$, namely $[\mu(B(0,3C_1r_0)]^{-1}\wt\mu_{r_0}$ restricted to the set $\frac{1}{r_0}E$, which is equivalent to the rectifiability of our original $\mu$ restricted to the set 
  $E$.

% **********************************************************************************************************************

\section{The stopping time}\label{sec5}

The rest of the paper will be devoted to proving the following lemma, which implies Theorem \ref{thmi} by an exhaustion argument. 

\begin{lemma}
	\label{l:main}
	With the assumptions of Theorem \ref{thmi}, there is $E'\subset E$ with $\mu(E')>0$ such that $\mu|_{E'}$ is $d$-rectifiable.
\end{lemma}

%\begin{lemma} \label{l:preiss} \cite[Lemma 2.4]{Pr87}
%	For $\mu$ a Radon measure, for $\mu$-almost every $x\in \R^{n}$ and for any $C_{1}>0$, there is $M>0$ and a  sequence $r_{j}\downarrow 0$ so that 
%	\begin{equation}
%	\label{e:doubling}
%	\mu(B(x,C_{1}r_{j}))\leq M\mu(B(x,r_{j})).
%	\end{equation}
%\end{lemma}

Let $E$ and $\mu$ be as in Theorem \ref{thmi}. We assume that there is no set $E'\subset E$ as in the lemma. Using the notation introduced in the previous section we obtain a contradiction as follows.
For $\tau$ and $A$ as in \rf{constants-1}, and $B_0=\overline{B(0,1)}$ let
\begin{equation}\label{e:goodA}
	G=\{x\in {{E_{0}}}\cap B_0: \Theta^d_{\mu}(x,r)\in [2^{-d}\tau,2^dA] \,\mbox{ for all }r\in (0,C_{1})\}.
	\end{equation}
	Under the hypothesis on $E_0$, $\mu|G$ is $d$-rectifiable (see proof of Lemma \ref{l:good}). Therefore $\mu(G)=0$ by the contradiction assumption. Using \cite{DT12} we construct an approximating Lipschitz surface $\Sigma$ near $E_0$ (see Section
	\ref{approximating}). We then construct an Ahlfors regular measure $\nu$ on $\Sigma$ which captures the behavior of $\mu$ on $E_0$ (see Section \ref{meas-approx}). This allows us to conclude in Section
	\ref{end} that $\mu(G)$ is proportional to $\mu(E_0\cap B_0)\ge (1-\ve)\mu(B_0)\ge C(\ve, M, C_1)\mu(B(0,3C_1)>0$ (see \rf{working-hyp-1A}, \rf{working-hyp-3A} and \rf{meas=1}), which contradicts the fact that $\mu(G)=0$.

%Because of the doubling condition \rf{e:asymdub}, there exists $M>0,r_0>0$, and a subset 
%$\wt E\subset E$ with $\mu(\wt E)>0$ such that
%$$\frac{\mu(B(x,2r))}{\mu(B(x,r))} \leq M\quad \mbox{ for all $x\in \wt E$ and all $0<r\leq r_0$.}$$
%Let $0<\ve<1/10$ and  $C_{1}>100$ be constants to be decided later (possibly depending on $M$). By  
%considering a density point of $\wt E$ and applying the
%Lebesgue differentiation theorem for measures \cite[Corollary 2.14]{Mattila}, we may find $\delta>0$, a closed set $E_{0}\subset \wt E$, and a ball $B_{0}$ centered on $E_{0}$ of radius $r_{B_0}$ such that
%\begin{enumerate}
%	\item $E_{0}\subset C_{1}B_{0}$,
%	\item $\mu(C_{1} tB_{0}\setminus E_{0})<\ve \,\mu(C_{1}tB_{0})$ for all $0<t\leq 3$,
%	\item $\mu(B(x,2r))\leq M \mu(B(x,r))$ for all $x\in E_{0}$ and $0<r\leq  4C_{1} \,r_{B_0}$, and
%	\item $\int_{0}^{4C_{1}r_{B_0}} \alpha_{\mu}(x,r)^{2} \frac{dr}{r} <\ve^{2}$ for all $x\in E_{0}$,
%\end{enumerate}
%Without loss of generality, we will assume $B_{0}=B(0,1)$,  and $\mu(3C_{1}B_{0})=1$ (to this end,
%note that $\alpha_\mu(x,r) = \alpha_{c\,\mu}(x,r)$ for any $c>0$).

%Let $A\gg1\gg\gamma\gg\tau\gg\ve^{\frac14}>0$. 
For $x\in E_{0}\cap B_0$, we define $\delta(x)$ to be the supremum over all radii $0<r\leq C_{1}$ such that either
the density ratio of $B(x,r)$ is either too big or too small or the angle between $L_{x,r}=L_{B(x,r)}$ as in \rf{L-def} and \rf{cL-ppts}
and $L_{B_0}$ is too big, that is:
%{\color{blue} It might be good to have the stuff below displayed in the center but I could not figure out how to do it}
\begin{center}
\begin{itemize}
\item[\ND:] $\mu(B(x,r)\setminus E_{0})\geq\ve^\frac{12} \mu(B(x,r))$,
\item[\LD:] %$x\not\in \ND$ and 
	$\Theta_{\mu}^{d}(B(x,r))\leq\tau$,
	\item[\HD:]  %$x\not\in \ND$ and %
	$\Theta_{\mu}^{d}(B(x,r))\ge A$, or
	\item[\BA:] %$x\not\in \ND\cup\LD\cup\HD$ and %
	$\angle(L_{x,r},L_{B_{0}})\geq\ve^{\frac{1}{4}}$. 

	%\item[(a)] $\Theta^d_{\mu}(x,r)\not\in [\tau,A]$,
	%\item[(b)] $\mu(B(x,r)\setminus E_{0})\geq \ve^{\frac12}\mu(B(x,r))$, or
	%\item[(c)] $\angle(L_{x,r},L_{B_{0}})\geq \ve^{\frac14}.$
\end{itemize}
\end{center}
\noindent The abbreviations stand for ``not dense",``low density", ``high density", and ``big angle", respectively. Note that by \rf{working-hyp-3A} if $x$ is such that $\mu(B(x,r)\setminus E_{0})\geq\ve \mu(B(x,r))$ for some $r\in (0,4C_1) $ then $x\in F_0^c$.

%Recall that the notation $\angle(L_{x,r},L_{B_{0}})$ was introduced in Section \ref{sec:notation}.

%Let $x\in E_{0}$ and $\delta(x)\leq r<2C_{1}$. Then we have
%$$\int_{r}^{2r} \alpha_{\mu}(x,t)^{2}\,\frac{dt}{t} \leq \int_{0}^{4 C_{1}} \alpha_{\mu}(x,t)^{2}\,\frac{dt}{t}\leq \ve^{2}.$$
%So there exists $t\in [r,2r]$ such that $\alpha_{\mu}(x,t)\lec\ve$.
%Then, from \eqref{e:alpha-monitone-general} it follows that
%\begin{equation}
%\label{e:alpha<e}
%\alpha_{\mu}(x,r)\lec_M \ve\quad \mbox{for } x\in E_{0} \mbox{ and }\delta(x)<r<C_{1}.
%\end{equation}

%Let $C_{0}>1$ to be fixed below. 
For $x\in \R^{n}$ define
\begin{equation}\label{def-d}
d(x) = \inf_{y\in E_{0}\cap  B_{0}} \bigl\{\delta(y) +|x-y|\bigr\}.
\end{equation}
Note that $d$ is a continuous function. Indeed, this is a $1$-Lipschitz function since this is defined
as an infimum over the family of $1$-Lipschitz functions $\bigl\{\delta(y) +|\cdot -y|: y\in E_{0}\cap  B_{0}\bigr\}$.
\xavi{I would erase the red text, since the new explanation I wrote before in terms of $1$-Lipschitz functions i
s simpler, in my opinion.}
%{In fact given $x,\, z\in\R^n$, and $\delta>0$ there exist  $y_x,\,y_z\in E_0\cap B_0$ such that 
%\begin{equation}\label{d-cont}
%\delta(y_x)+|y-y_x|-\delta\le d(x)\le \delta(y_x)+|y-y_x|\hbox{  and  }\delta(y_z)+|y-y_z|-\delta\le d(z)\le \delta(y_z)+|y-y_z|.
%\end{equation}
%Using \rf{def-d} and \rf{d-cont} we have
%\begin{equation}\label{d-cont-2}
%d(x)-d(z)\le \delta(y_z)+|x-y_z| -\delta(y_z)-|y-y_z| +\delta\le |x-z| +\delta.
%\end{equation} 
%Reversing the roles of $x$ and $z$ and letting $\delta$ tend to 0 we conclude that $d$ is Lipschitz continuous.}

\begin{lemma}\label{l:good-range}
For $A$ and $\tau^{-1}$ large enough, depending on $C_{1},$ and $M$, 
\begin{equation}
\label{e:delta0}
d(x)\leq \delta(x)\leq 10^{-3} \quad\mbox{ for all }x\in E_{0}\cap B_{0}.
\end{equation}
Moreover, for all $r$ such that $ \delta(x)\leq r<2C_{1}$ and $x\in E_{0}\cap B_{0}$,
\begin{equation}
\label{e:good-range}
\Theta^d_{\mu}(x,r)\in [\tau,A], 
\quad\frac{\mu(B(x,r)\setminus E_{0})}{ \mu(B(x,r))}\leq\ve^{\frac{1}{2}}, \;
\;\mbox{ and } \;\;  \angle(L_{x,r},L_{B_{0}})\leq \ve^{\frac14}.
\end{equation}
\end{lemma}

\begin{proof}
First note that since $C_1>1$, if $x\in E_{0}\cap B_{0}$ and $10^{-3}\leq r<2C_{1}$, then $B(x,r)\subset 3C_{1}B_{0}$. Hence,
\begin{equation}\label{density-upper-bound}
\Theta^d_{\mu}(x,r)
\leq 10^{3d}\,\mu(3C_{1} B_{0}) \lec 1
\end{equation}
and since $3C_1B_0\subset B(x,4C_1)$ \rf{working-hyp-1A} yields
\begin{equation}\label{density-lower-bound}
\Theta^d_{\mu}(x,r)\gec_{M,C_{1}} \Theta^d_{\mu}(x,4C_{1}) 
\gec \frac{\mu(3C_{1}B_{0})}{(4C_{1})^{d}} \gec C_{1}^{-d}.
\end{equation}
Thus, for $A,\tau^{-1}$ large enough depending on $C_{1}$, and $M$, \rf{density-upper-bound} and \rf{density-lower-bound} imply
\begin{equation}\label{e:ugh}
\Theta^d_{\mu}(x,r)\in [\tau,A]\quad \mbox{ for all } x\in E_{0}\cap B_{0} \mbox{ and }10^{-3}\leq r<2C_{1}.\end{equation}

Furthermore, by \eqref{pointwise-a-bound-2}, \eqref{e:angle-general}, and \eqref{e:ugh}, for the same choice of $x$ and $r$,
\[
\angle(L_{x,r},L_{B_{0}}) \lec_{C_{1},A,\tau,M} \alpha(2C_{1}B_{0}) \]
and so for $\ve>0$ small enough, we can guarantee that $\angle(L_{x,r},L_{B_{0}}) <\ve^{\frac14}$ for all $x\in E_{0}\cap B_0$ and $10^{-3}\leq r<2C_{1}$. 

Finally, for $x\in E_{0}\cap B_{0}$ and $10^{-3}\leq r<2C_{1}$, by \eqref{e:ugh},
\[
\mu(B(x,r)\setminus E_{0})
\leq \mu(3C_{1}B_{0}\setminus E_{0})<\ve\, \mu(3C_{1}B_{0})
\approx_{M,C_{1}} \ve \,\mu(B(x,r)),\]
and so $\mu(B(x,r)\setminus E_{0})<\ve^{\frac{1}{2}}\mu(B(x,r))$ for $\ve$ small enough. These facts imply that $\delta(x)\leq 10^{-3}$, and \eqref{e:good-range} follows immediately.
\end{proof}
%
%Note that $\Lip(d)\leq C_0$. Also, for $A$ and $\tau^{-1}$ large enough (depending on $C_{0}$), we can guarantee that 
%\begin{equation}
%
%d(0)\leq C_{0} \delta(0)< 1
%\end{equation}
%and so $d(x)\leq 1+C_0\,|x|$ for all $x\in \R^{n}$. 

\begin{rem}\label{remark1}
Using \rf{working-hyp-1A} and a similar argument to the one  that appears in the proof of Lemma \ref{l:good-range} we deduce that for any given constant $0<c_0\leq1$, given $r$ such that $ c_0\delta(x)\leq r<2C_{1}$ and $x\in E_{0}\cap B_{0}$, we have
\begin{equation}
\label{e:good-range'}
\tau\lesssim_{c_0,M}\Theta^d_{\mu}(x,r)\lesssim_{c_0,M} A, 
\quad\frac{\mu(B(x,r)\setminus E_{0})}{ \mu(B(x,r))}\lesssim_{c_0,M}\ve^{\frac{1}{2}}, \;
\;\mbox{ and } \;\;  \angle(L_{x,r},L_{B_{0}})\lesssim_{c_0,M} \ve^{\frac14}.
\end{equation}
\end{rem}

%{\color{blue} In the Lemmas below 4 and 32 changed to 2, second estimate modified taking into account \rf{working-hyp-1A}}
 
\begin{lemma} 
For $x\in \mathbb{R}^{n}$ and $ 2d(x) \le r< C_{1}$,
\begin{equation}\label{e:muAD}
2^{-d}\tau r^{d}\leq \mu(B(x,r))\leq 2^{d}Ar^{d}
\end{equation}
and
\begin{equation}\label{eqe0*}
\frac{\mu(B(x,r)\setminus E_{0})}{ \mu(B(x,r))}\lesssim_{M}\ve^{\frac{1}{2}}.
\end{equation}
\end{lemma}

\begin{proof} If $d(x)=0$, since $E_0$ is closed by \rf{def-d} $x\in E_0\cap B_0$ thus by \rf{working-hyp-1A} and \rf{e:good-range}, \rf{e:muAD} and \rf{eqe0*} hold.
Suppose that $d(x)>0$.
	Let $y\in E_{0}\cap B_{0}$ be such that 
	\[
	\delta(y) + |x-y|\le \frac{1}2\,r.
\]
	Then $r/2\ge\delta(y)$ and $|x-y|\le r/2$. Recalling that $r<C_1$, we deduce that $\Theta^d_{\mu}(y,r/2)\geq \tau$ and 
	$\Theta^d_{\mu}(y,3r/2)\leq A$ (this follows from the definition of $\delta(y)$ if $3r/2<C_1$ and from
	the fact that $\mu(3C_{1}B_{0})=1$ otherwise). Hence,
\begin{equation}\label{eqe1*}
	\mu(B(x,r))\geq \mu(B(y,r-|x-y|)) \geq \mu(B(y,\tfrac12r))\geq 2^{-d} \tau r^{d},
\end{equation}
	and also
\begin{equation}\label{eqe2*}
	\mu(B(x,r))\leq \mu(B(y,r+|x-y|))\leq  \mu(B(y,(1+\tfrac12)r))\leq 2^{d}Ar^{d}.
\end{equation}

On the other hand, arguing as in the preceding estimate, using also  \rf{working-hyp-1A},
we have
\begin{align*}
\mu(B(x,r)\setminus E_0)& \leq  \mu(B(y,\tfrac32r)\setminus E_0)\leq
\ve^\frac{1}{2}\, \mu(B(y,\tfrac32r))\\
&\lesssim_M \ve^\frac{1}{2}\, \mu(B(y,\tfrac12r))
\lesssim_M\ve^\frac{1}{2}\,\mu(B(x,r)).
\end{align*}

\end{proof}

The following is an immediate consequence of \eqref{e:muAD} and Lemma \ref{l:alpha-monitone}.

\begin{lemma}\label{a-comparison}
	For $x,y\in \R^{n}$, $2d(x)<r<s<C_{1}$, if $B(x,2r)\subset B(y,s)$,

\begin{equation}\label{e:alpha-monitone}
\alpha_{\mu}(x,r)\lec_{A,\tau} \ps{\frac{s}{r}}^{d+1}\alpha_{\mu}(y,s).
\end{equation}
\end{lemma}

%
%
%\begin{proof}
%	Let $\phi\in \Lip_{1}(B(x,r))$. Then
%\[ 
%\av{\int\phi d\mu - \int \phi d\LL_{y,s}}
%\leq s\mu(B(y,s)) 2\alpha_{\mu}(y,s).
%\]
%Hence,
%\[
%\alpha_{\mu}(x,r) 
%\leq \frac{s\mu(B(y,s))}{r\mu(B(x,r))}2 \alpha_{\mu}(y,s)
%\stackrel{\eqref{e:muAD}}{\leq}  4^{d} \frac{As^{d+1}}{\tau r^{d+1}} 2\alpha_{\mu}(y,s).\]
%
%
%
%\end{proof}

%The next lemma follows from \eqref{e:muAD}, and Lemmas \ref{l:LcapB}, \ref{l:cb-theta}, \ref{l:angle-general}, and \ref{l:cx-cy-general}.

\begin{lemma} \label{lem5.6}
For $\ve>0$ small enough, $x\in \mathbb{R}^{n}$ and $2d(x)\le r<C_{1}$, 
\begin{equation}\label{e:LcapB}
L_{x,r}\cap B(x,r/4)\neq\varnothing,
\end{equation} 

	\begin{equation}\label{e:cb-mub}
c_{x,r} \approx \Theta^d_{\mu}(x,r) \approx_{A,\tau} 1.
\end{equation}
%\end{lemma}
%
%\begin{proof}
%Let $L=L_{x,r}$, $tr=1-\dist(x,L)$, and $z\in L\cap B(x,r)$ be closest to $0$. 
%First assume that $t\leq 7/8$. Let 
%\[
%\phi(x) = \min\{\dist(x,B(x,(1-t)r)^{c}), (3/4-t)r\}
%\]
%Then $\phi\in \Lip_{1}(B(x,r))$ and $\phi\geq (\frac{3}{4}-t )r\geq r/8$ on $B(x,r/4)$. Since $r>16C_{0}d(x)$, we know $r/4>4C_{0}d(x)$ and hence
%\[
%\tau^{d}r^{d+1}
%\stackrel{\eqref{e:muAD}}{\lec} r \mu(B(x,r/4))
%\lec \int \phi d\mu
%r\mu(B(x,r)) \alpha_{\mu}(x,r) + \int \phi d\LL_{x,r}
%\stackrel{\eqref{e:muAD}}{\lec} A^{d} r^{d+1}\ve + 0
%\]
%which is a contradiction for $\ve>0$ small enough. Thus, $t>7/8$ and so $\dist(x,L)<1/8$, and this proves the lemma. 
%
%\end{proof}
%
%\begin{lemma}
%	For $\ve>0$ small, if $x\in \mathbb{R}^{n}$ and $32d(x)/C_{0}<r<C_{1}$, then 
%	\begin{equation}\label{e:cb-mub}
%		c_{x,r} \approx \theta_{\mu}(x,r) \approx_{A,\tau} 1.
%	\end{equation}
%\end{lemma}
%
%\begin{proof}
% Note $\mu(B(x,r/2)) \approx_{A,\tau} r^{d}$. If 
% \[
% \phi(x) = \min\{r/2,\dist(x,B(x,r)^{c})\}\]
% Then $\phi\in \Lip_{1}(B(x,r))$ and $\phi\geq r/2$ on $B(x,r/2)$, so
% \[
%c_{x,r}r^{d+1}
%\stackrel{\eqref{e:LcapB}}{\approx}
%r\LL_{x,r}(\frac{1}{2} B)
%\leq \int \phi d\LL_{x,r}
%<2\alpha_{\mu}(x,r)\mu(B(x,r))r +  \int \phi d\mu
%\leq (2\ve +1)r \mu(B(x,r))
% \]
% and hence $c_{x,r} \lec \Theta^d_{\mu}(x,r)$. Similar estimates give $c_{x,r} \gec  \Theta^d_{\mu}(x,r)$. 
%\end{proof}
and if $B(x,2r)\subset B(y,s)$, then
	\begin{equation}\label{e:angle}
	\dist_{x,r/2}(L_{x,r},L_{y,s})\lec_{A,\tau} \frac{s^{d+1}}{r^{d+1}}\,\alpha_{\mu}(y,s).
	\end{equation}
%
%\begin{proof}
%	Let 
%	\[
%	\phi(z) = \max\{\dist(z,B(x,r)^{c}), \dist(z,L_{y,s}).
%	\]
%	Then
%	\begin{align*}
%	\avint_{B(x,r/2)} \frac{\dist(z,L_{x,r})}{r} d\cH^{d}|_{L_{x,r}}
%	& \stackrel{\eqref{e:muAD} \atop \eqref{e:cb-mub},\eqref{e:LcapB}}{\lec}  \int \phi d\LL_{x,r} 
%  \lec \alpha_{\mu}(x,r)
%	+ \frac{1}{r\mu(B(x,r))} \int  \phi d\mu \\
%	& \stackrel{\eqref{e:muAD}}{\lec} \alpha_{\mu}(x,r)
%	+ \frac{s\mu(B(y,s))}{r\mu(B(x,r))} \alpha_{\mu}(y,s) + \frac{1}{r\mu(B(x,r))} \int \phi d\LL_{y,s} \\
%	& \stackrel{\eqref{e:muAD},\eqref{e:alpha-monitone}}{\lec} \ps{\frac{s}{r}}^{d+1} \alpha_{\mu}(y,s).
%	\end{align*}
%	Using some geometry, the lemma now follows from this estimate 
%\end{proof}

Further, if $x,y\in \mathbb{R}^{n}$, $2\,\max\{d(x),d(y)\}<r<s<C_{1}$, and $B(x,2r)\subset B(y,s)$, then,
		\begin{equation}
		\label{e:cx-cy}
		|c_{x,r}-c_{y,s}| \lec_{A,\tau} \ps{\frac{s}{r}}^{d+2}\alpha_{\mu}(y,s).
		\end{equation}
\end{lemma}

This lemma follows from \eqref{e:muAD}, and Lemmas \ref{l:LcapB}, \ref{l:cb-theta}, \ref{l:angle-general}, and \ref{l:cx-cy-general}. In fact note that if $2d(x)<r<C_1$, there is $z\in E_0\cap B_0$ such that 
$\delta(z)+|x-z|\le r/2$, then $B(x,r)\subset B(z,2r)$ and $\alpha_\mu^d(z,2r)\le 4M\ve$ by \rf{pointwise-a-bound-2} then as in Lemma \ref{a-comparison}, $\alpha_{\mu}(x,r)\lec_{A,\tau}\ve$,  which by
\eqref{e:muAD} ensures that the conclusions to Lemmas \ref{l:LcapB}, \ref{l:cb-theta}, \ref{l:angle-general}, and \ref{l:cx-cy-general} hold.

%{\color{blue} A minor explanation was added above}
%
%\begin{proof}
%	Let $\phi(z) = \max\{0,|z-x|\}$. Then
%	\begin{align*}
%	r^{d+1} |c_{x,r} -c_{y,s} |
%	&  \lec \av{\int \phi c_{x,r} d\cH^{d}|_{L_{x,r}}  -\int \phi c_{y,s} d\cH^{d}|_{L_{x,r}} }\\
%	& \leq \av{\int \phi c_{x,r} d\cH^{d}|_{L_{x,r}}  -\int \phi d\mu}
%	+\av{\int \phi d\mu -  \int \phi c_{y,s} d\cH^{d}|_{L_{y,s}} }\\
%& \qquad		+c_{y,s} \av{\int \phi  d\cH^{d}|_{L_{y,s}}  -\int \phi d\cH^{d}|_{L_{x,r}} }\\
%		& \stackrel{\eqref{e:angle}}{\lec}  \alpha_{\mu}(x,r) r\mu(B(x,r)) +  \alpha_{\mu}(y,s)s\mu(B(y,s)) + c_{y,s} \alpha_{\mu}(s,r) s^{d+1}\\
%		& \stackrel{\eqref{e:cb-mub} \atop \eqref{e:alpha-monitone}}{\lec} \alpha_{\mu}(y,s)s^{d+1}.
%\end{align*}
%\end{proof}

\begin{rem}\label{remark2}
In the preceding lemma, if we assume that $x,y\in E_0\cap B_0$ and we allow $c_0d(x)\le r\le C_{1}$, with $c_0<2$,
then \rf{e:LcapB}, \rf{e:cb-mub}, \rf{e:angle}, and \rf{e:cx-cy} also hold, with implicit constants
depending on $A,\tau,M,c_0$, assuming $\ve$ small enough.
\end{rem}

\begin{lemma}\label{l:good}
	Under the contradiction assumption for Lemma \ref{l:main} and using the notation above we have that the set 
	\begin{equation}\label{e:good}
	G=\{x\in {E_{0}}: \Theta^d_{\mu}(x,r)\in [2^{-d}\tau,A2^{d}] \,\mbox{ for all }r\in (0,C_{1})\}.
	\end{equation}
	satisfies $\mu(G)=\cH^{d}(G)=0$. In particular, if $Z=\{x:d(x)=0\}$, then $Z\subset G\subset E_0$ and $\mu(Z)=\cH^{d}(Z)=0$.
\end{lemma}

\begin{proof}
	It is easy to see that $\mu|_{G}\ll \cH^{d}|_{G}\ll \mu|_{G}$ since
	\begin{equation}\label{e:t<theta<A}
	2^{-d}\tau\leq \liminf_{r\rightarrow 0} \Theta_{\mu}^{d}(x,r) 
	\leq  \limsup_{r\rightarrow 0} \Theta_{\mu}^{d}(x,r) 
	\leq 2^{d} A \;\; \mbox{for all }x\in G.
	\end{equation}
	See for example \cite[Theorem 6.9]{Mattila}. Given $x\in G$ and $0<r<C_{1}/2$,
consider the function $\phi(y) = \frac1r(2r-|x-y|)_+$. Then we have	
	\begin{align}\label{eq514}
	\beta_{\mu|_{G},1}^{d}(x,r)
	& :=\inf_{L}\frac{1}{r^{d}}\int_{B(x,r)}\frac{\dist(y,L)}{r}\, d\mu|_{G}(y)
	 \leq \frac{1}{r^{d}}\int_{B(x,r)} \phi(y)\,\frac{\dist(y,L_{x,2r})}{r}\,d\mu(y)\\
	& \lec \alpha_{\mu}(x,2r)\,\frac{\mu(B(x,2r))}{r^{d}}
	\stackrel{\eqref{e:muAD}}{\lec_{A}} \alpha_{\mu}(x,2r).\nonumber
	\end{align}
	Thus, $	\int_{0}^{1} \beta_{\mu|_{G},1}^{d}(x,r)^{2}\frac{dr}{r}<\infty$ for each $x\in G$, and so $\mu|_{G}$ is $d$-rectifiable by \cite[Theorem A]{BS16}. Therefore, $\mu(G)=0$ by our assumption at the beginning of the proof that $\mu$ vanishes on any $d$-rectifiable subset of positive measure. Now we just observe that by \eqref{e:muAD}, $Z\subset G$, and so the proof is finished. 
\end{proof}

As explained at beginning of Section \ref{sec5} the goal of the rest of the paper is to show that in fact $\mu(G)>0$.

\vv
% *******************************************************************************************

\section{The approximating surface}\label{approximating}

We will rely on the following theorem.

\begin{theorem}\cite{DT12}\label{teo61}
	For $k\in \N\cup\{0\}$, set $r_{k}=10^{-k}$ and let $\{x_{j,k}\}_{j\in J_{k}}$ be a collection of points so that for some $d$-plane $P_{0}$,
	\[\{x_{j,0}\}_{j\in J_{0}}\subset P_{0},\] 
	\[|x_{i,k}-x_{j,k}|\geq r_{k} \quad\mbox{ for all $i,j\in J_k$},\]
	and, denoting $B_{j,k}=B(x_{j,k},r_{k})$, 
	\begin{equation}
	x_{i,k}\in V_{k-1}^{2}\quad\mbox{ for all $i\in J_k$,}
	\label{V2}
	\end{equation}
	where
	\[ V_{k}^{\lambda}:=\bigcup_{j\in J_{k}}\lambda B_{j,k}.\]
	To each point $x_{j,k}$, associate a $d$-plane $P_{j,k}\subset \R^n$  
	{such that $P_{j,k} \ni x_{j,k}$}
	and set
	\[
	\ve_{k}(x)=\sup\{\hdis{x,10^{4}r_{l}}(P_{j,k},P_{i,l}): j\in J_{k}, |l-k|\leq 2, i\in J_{l}, 
	x\in 100 B_{j,k}\cap 100 B_{i,l}\}.
	\]
	There is $\ve_{1}>0$ such that if $\ve\in (0,\ve_{1})$ and 
	\begin{equation}
	\label{e:vek<ve}
	\ve_{k}(x_{j,k})<\ve \mbox{ for all }k\geq 0\;\mbox{ and }\;j\in J_{k},
	\end{equation}
	then there is a bijection $g:\R^{n}\rightarrow \R^{n}$ so that the following hold
	\begin{enumerate}[(i)]
		\item We have 
		\begin{equation}\label{eqahj0}
		E_{\infty}:=\bigcap_{K=1}^{\infty}\overline{\bigcup_{k= K}^{\infty} \{x_{j,k}\}_{j\in J_{k}}}\subset \Sigma:= g(\R^{d}).
		\end{equation}
		\item $g(z)=z$ when $\dist(z,P_{0})>2$.
		\item {There is some $\tau_0>0$} such that, for $x,y\in \R^{n}$,
		\[ \frac{1}{4}|x-y|^{1+\tau_0}\leq |g(x)-g(y)|\leq 10|x-y|^{1-\tau_0}.\]
		\item We have 
		\begin{equation}
		\label{e:gz-z}
		|g(z)-z|\lec \ve \;\mbox{ for }\;z\in \R^{n}.
		\end{equation}
		\item There is a maximal $\frac{r_{k}}{2}$-separated set $\{x_{j,k}\}_{j\in L_{k}}$ in $\R^{n}\setminus V_{k}^{9}$ such that setting
		\[
		B_{j,k}=B(x_{j,k},r_{k}/10)\; \mbox{ for $\;j\in L_{k}$},\] 
		we have $g(x)=\lim_{k} \sigma_{k}\circ\cdots \sigma_{0}(x)$ all for $x\in P_0$, where 
		$\sigma_k:\R^n\to\R^n$ is defined by
		\begin{equation}
		\sigma_{k}(y)=\psi_{k}(y)y+\sum_{j\in J_{k}}\theta_{j,k}(y)\, \pi_{j,k}(y),
		\label{e:sigmak}
		\end{equation}
		and where $\pi_{j,k}$ is the orthogonal projection onto $P_{j,k}$,
		$\{\theta_{j,k}\}_{j\in L_{k}\cup J_{k}}$ is a partition of unity such that $\chi_{9B_{j,k}}\leq \theta_{j,k}\leq \chi_{10 B_{j,k}}$ for all $k$ and $j\in L_{k}\cup J_{k}$, and $\psi_{k}=\sum_{j\in L_{k}}\theta_{j,k}$. 
		 
		\item \cite[Equation (4.5)]{DT12} For $k\geq 0$,
		\begin{equation}\label{e:v10c}
		\sigma_{k}(y)=y \;\mbox{ and } \;D\sigma_{k}(y)=Id\; \mbox{ for } \;y\in \R^{n}\setminus V_{k}^{10}.
		\end{equation}
		\item\cite[Proposition 5.1]{DT12} Let $\Sigma_{0}=P_{0}$ and
		\[
		\Sigma_{k+1}=\sigma_{k}(\Sigma_{k}).\] 
		There is a function $A_{j,k}: P_{j,k}\cap 49B_{j,k}\rightarrow P_{j,k}^{\perp}$ of class $C^{2}$ such that $|A_{j,k}(x_{j,k})|\lec \ve r_{k}$, $|DA_{j,k}|\lec \ve$ on $P_{j,k}\cap 49B_{j,k}$, and if $\Gamma_{j,k}$ is its graph over $P_{j,k}$, then 
		\begin{equation}\label{e:49graph}
		\Sigma_{k+1}\cap D(x_{j,k},P_{j,k},49r_{k})=\Gamma_{j,k}\cap  D(x_{j,k},P_{j,k},49r_{k}),
		\end{equation}
		where 
		\begin{equation}\label{e:D-is-a-cylinder}
		D(x,P,r)=\{z+w: z\in P\cap B(x,r), w\in P^{\perp} \cap B(0,r)\}. 
		\end{equation}
		(Above $P^\perp$ is the $(n-d)$-plane perpendicular to $P$ going through 0.)
		In particular,
		\begin{equation}
		\hdis{x_{j,k},49r_{j,k}}(\Sigma_{k+1},P_{j,k})\lec\ve.
		\label{e:49r}
		\end{equation}
		\item \cite[Lemma 6.2]{DT12} For $k\geq 0$ and $y\in \Sigma_{k}$, there is an affine $d$-plane $P$ through $y$ and a $C\ve$-Lipschitz and $C^{2}$ function $A:P\rightarrow P^{\perp}$ so that if $\Gamma$ is the graph of $A$ over $P$, then 
		\begin{equation}\label{e:ygraph}
		\Sigma_{k}\cap B(y,19r_{k})=\Gamma\cap B(y,19r_{k}).
		\end{equation}
		\item \cite[Proposition 6.3]{DT12} $\Sigma=g(P_0)$ is $C\ve$-Reifenberg flat in the sense that for all $z\in \Sigma$, and $t\in (0,1)$, there is a $d$-plane $P=P(z,t)$ so that $F_{z,t}(\Sigma,P)\lec \ve$. 
		\item \cite[Equation (6.7)]{DT12} For all $y\in \Sigma_{k}$, 
		\begin{equation}\label{e:sky-y}
		|\sigma_{k}(y)-y|\lec \ve r_{k}.
		\end{equation}
		In particular, it follows that
		\begin{equation}\label{e:ytosigma}
		\dist(y,\Sigma)\lec \ve r_{k}\quad \mbox{for }y\in \Sigma_{k}.
		\end{equation}
		
		\item \cite[Lemma 7.2]{DT12} For $k\geq 0$, $y\in \Sigma_{k}\cap V_{k}^{8}$, choose $i\in J_{k}$ such that $y\in 10 B_{i,k}$. Then
		\begin{equation}\label{e:skCloseInBall}
		|\sigma_{k}(y)-\pi_{i,k}(y)|\lec \ve_{k}(y)r_{k}.
		\end{equation} 
%		and
%		\begin{equation}
%		|D\sigma_{k}(y)-D\pi_{i,k}|\lec \ve_{k}(y) \end{equation}
%		If $T\Sigma_{k}(x)$ denotes the tangent space at $x\in \Sigma_{k}$, then 
%		\begin{equation}\label{e:Tsig}
%		\angle(T\Sigma_{k+1}(\sigma_{k}(x)),P_{i,k})\lec \ve_{k}(y) \mbox{ for }x\in \Sigma_{k}\cap B(x_{i,k},10r_{k}).
%		\end{equation}
%		
		%\item \cite[Proposition 6.3]{DT12} 
		%For all $y\in \Sigma_{k}$, there is a plane $P$ passing through $y$ and a $C\ve$-Lipschitz function $A:P\rightarrow P^{\perp}$ so that $\Sigma_{k}\cap B(y,19r_{k})=\Gamma\cap B(y,19r_{k})$ where $\Gamma$ is the graph of $A$ over $P$. 
%		\item \cite[Lemma 9.1]{DT12} For $x,y\in \Sigma_{k}$, 
%		\begin{equation}
%		\angle(T\Sigma_{k}(x),T\Sigma_{k}(y))\lec \ve \frac{|x-y|}{r_{k}}.
%		\label{e:Tlip}
%		\end{equation}

\item \cite[Proposition 8.3]{DT12} If $g_{k}(x)= \sigma_{k}\circ\cdots \circ \sigma_{0}(x)$ and, for all $x\in P_{0}$,
		\begin{equation}\label{e:DTsum}
			\sum_{k\geq 0} \ve_{k}(g_{k}(x))^{2}\leq \ve,		
		\end{equation}
	then for $\ve$ small enough, $g$ is $\exp(C\ve)$-bi-Lipschitz, and hence $(1+C\wt\ve)$-bi-Lipschitz (this is not stated as such in \cite[Proposition 8.3]{DT12}, but it follows from its proof. To observe this, the crucial inequalities are  (8.10)-(8.11) and (8.22)-(8.23) in \cite{DT12}).
		\item \cite[Lemma 13.2]{DT12} Under the assumption \rf{e:DTsum}, for $x\in \Sigma$ and $r>0$, 
		\begin{equation}
		\cH^{d}_{\infty}(B(x,r)\cap \Sigma)\geq (1-C\ve)\,\omega_{d} r^{d},
		\label{e:sigmalr}
		\end{equation}
		where $\omega_{d}$ is the volume of the unit ball in $\R^{d}$ (this statement is proven in \cite{DT12} with $\cH^{d}$ in place of $\cH^{d}_{\infty}$, but the same proof works for $\cH^{d}_{\infty}$). 
	\end{enumerate}

	%
	%Moreover, $g(x)=\lim_k f_{k}(x)$ where $|f_{k}(x_{j,k})-x_{j,k}|\lec \ve r_{k}$, and 
	%\begin{equation}\label{closetog}
	%\dist(x,g(\R^{d}))\lesssim \ve r_{k}\mbox{\; for all\; }x\in 40B_{j,k}\cap L_{j,k}.\end{equation}
	\label{t:DT}
\end{theorem}

We now apply this result to our situation. For $k\geq0$, let $r_{k}=10^{-k}$ and $\{x_{j,k}'\}_{j\in J_{k}}$ be a maximally $(1+1/10)r_{k}$ separated set in $E_{k}$, where
\begin{equation}\label{def-E_k}
E_{k}:=\{x\in E_{0}\cap B_{0} :d(x)< r_{k}\}\subset E_{0}\cap B_{0}.
\end{equation}
Here $E_0$, $B_0$ and $d(x)$ are as in Section \ref{sec5} and \rf{def-d}.
 Note that by \eqref{e:delta0} %, we may assume that  $\{0\}= \{x_{j,0}'\}_{j\in J_{0}}=\{x_{0,0}\}$, so 
 $E_k=E_0\cap B_0$ for $k=0,1,2,3$. If $E_k=\emptyset$ then $J_k=\emptyset$.
 Let $C_2$ be such that $1<C_{2}^2<C_{1}$.
\begin{equation}\label{coice-x'}
P_{j,k}'=L_{x_{j,k}',C_{2} r_{k}}, \;\;\; \LL_{j,k}'=\LL_{x_{j,k}',C_{2} r_{k}}
\end{equation}
and $P_{0}=P_{0,0}'$. 

These would be good planes and points for the purpose of applying Theorem \ref{t:DT} if each $d$-plane $P_{j,k}$ passed through  $x_{j,k}'$. Since this may fail, some extra care must be taken. 
%By choosing our ball $B_{0}$ differently, we can assume that $x_{0,0}\in P_{0}$ \xavi{I think that the assumption $x_{0,0}\in P_0$ requires more care. Maybe it's better not to assume this and treat this point like the others.}, so we need only worry about the other points. 

Note that $r_{k}>d(x_{j,k}')$, and so by Lemma \ref{l:good-range} and the subsequent remark, arguing as in
\rf{eq514}, we obtain
\begin{align*}
\avint_{B(x_{j,k}', r_{k}/2)\cap E_{0}}\frac{\dist(x,P_{j,k}')}{r_{k}}d\mu
& \lec \alpha(x_{j,k}',2C_{2}r_{k})\frac{\mu(B(x_{i,k}',2C_{2} r_{k}))}{\mu(B(x_{j,k}', r_{k}/2)\cap E_{0}))}\\
& \stackrel{\eqref{e:good-range}}{\leq} \alpha(x_{j,k}',2C_{2}r_{k}) \frac{\mu(B(x_{i,k}',2C_{2} r_{k}))}{\mu(B(x_{j,k}',r_{k}/2))(1- c(M)\ve^{\frac{1}{2}})}\\
&\approx_{A,\tau,M, C_2} \alpha(x_{j,k}',2C_{2}r_{k}).
\end{align*}
Thus, for $\ve>0$ small enough, there is $x_{j,k}\in B(x_{j,k}',r_{k}/10)\cap E_{0}$ so that 
\begin{equation}\label{e:xP'}
\dist(x_{j,k},P_{j,k}')\lec_{A,\tau, C_2}  \alpha(x_{j,k}',2C_{2}r_{k}) \,r_{k}.
\end{equation}
Let $B_{j,k} = B(x_{j,k},r_{k})$, $B_{j,k}'=B(x_{j,k}',r_{k})$, and $V_{k}^{\lambda}$ be as in Theorem \ref{t:DT}. Notice that since $\{x_{j,k}'\}_{j\in J_{k}}$ is a maximal $(1+1/10)r_{k}$-net for $E_{k}\cap B_{0}$, the sequence $\{x_{j,k}\}_{j\in J_{k}}$ is now $r_{k}$-separated (because $\alpha_{\mu}(2C_{2}B_{j,k}')\ll1$), and we have 
\begin{equation}
\label{e:E<V}
E_{k}\cap B_{0} \subset V_{k}^{3/2}.  
\end{equation}
Moreover, since $E_{k+1}\subset E_{k}$, $x_{j,k+1}'\in \bigcup_{i} B(x_{i,k}',r_{k})$, and so
\begin{equation}\label{e:E<V-1}
x_{j,k+1}\subset \bigcup_{i} B(x_{i,k}',r_{k}+C\ve r_{k+1})\subset \bigcup_{i} B(x_{i,k},r_{k}+r_{k}/10+C\ve r_{k})\subset V_{k}^{3/2},
\end{equation}
which ensures that \rf{V2} holds.

Let $P_{j,k}$ be the $d$-plane parallel with $P_{j,k}'$ that passes through $x_{j,k}$ and let
\[
c_{j,k} := c_{B(x_{j,k}',C_{2}r_{k})}=c_{C_{2}B_{j,k}'}.\] 
%\begin{multline}
%\label{e:distPj,k}
%
%\alpha_{\mu}(C_{2} B_{j,k})
% \lec  r_{k}^{-d-1}F_{C_{2}B_{j,k}}(\mu,c_{j,k}\cH^{d} |_{P_{j,k}})  \\
%\leq  r_{k}^{-d-1}F_{C_{2}B_{j,k}}(\mu,c_{j,k}\cH^{d} |_{P_{j,k}'})+r_{k}^{-d-1}F_{C_{2}B_{j,k}}(c_{j,k}\cH^{d} |_{P_{j,k}'},c_{j,k}\cH^{d} |_{P_{j,k}})  \\
%\lec \alpha_{\mu}(C_{2} B_{j,k})
%\stackrel{\eqref{e:angle}, \eqref{e:xP'} \atop \eqref{e:alpha-monitone}}{\lec}_{A,\tau} \ve .
%\end{multline}
Similarly, let $\LL_{j,k}=c_{j,k} \HH^{d}|_{P_{j,k}}$ be the translate of $\LL_{j,k}'$. Note that $B_{j,k}\subset 2B_{j,k}'$, and so 
\begin{align}
F_{C_{2}B_{j,k}}(\mu, \LL_{j,k})
& \leq F_{C_{2}B_{j,k}}(\mu, \LL_{j,k}')
+F_{C_{2}B_{j,k}}(\LL_{j,k}', \LL_{j,k}) \notag \\
& \stackrel{\eqref{e:xP'}}{\lec}_{A,\tau,C_2} F_{2C_{2}B_{j,k}'}(\mu, \LL_{j,k}) +  r_{k}^{d+1}\alpha_{\mu}(2C_{2}B_{j,k}') \notag \\
& \stackrel{\eqref{e:muAD}}{\lec}_{A,\tau,C_2}  r_{k}^{d+1}\alpha_{\mu}(2C_{2}B_{j,k}') .
\label{e:muLjk1}
\end{align}

In the case $k=0$, since $B_{0}=\cnj{B(0,1)}$  we may assume that $\{x_{j,0}\}_{j\in J_0}= \{x_{0,0}\}=\{0\}$ and so $P_{0,0}$ passes through the center of $B_0$.

\begin{lemma}
For $C_{2}$ large enough and $x\in E_k$,
	\begin{equation}\label{e:ve<a}
	\ve_{k}(x) \lec_{A,\tau, C_2} \alpha_{\mu}(x,C_{2}^2 r_{k}).
	\end{equation}
\end{lemma}

Notice that this lemma ensures that \eqref{e:vek<ve} holds (up to a constant). 

\begin{proof}
Let $i,j,k,l$ be such that  $j\in J_{k}$,  $i\in J_{l}$, $l\leq k\leq l+2,$
and $x\in 100 B_{j,k}\cap 100 B_{i,l}$. Then for $C_{2}$ large enough, $\frac{C_{2}}{2} B_{j,k}\subset C_{2} B_{j,k}\cap C_{2} B_{i,l}$, and so 
\begin{align*}
%F_{\frac{C_{2}}{2} B_{j,k}}(\LL_{j,k},\LL_{i,l})
% \leq 
F_{\frac{C_2}{2} B_{j,k}}(\LL_{j,k},\mu)  & +F_{\frac{C_{2}}{2} B_{j,k}}(\mu,\LL_{i,l})
 \leq F_{C_{2}B_{j,k}}(\LL_{j,k},\mu) +F_{C_{2}B_{i,l}}(\mu,\LL_{i,l})\\
& \stackrel{\eqref{e:muLjk1}}{\lec}_{A,\tau,C_2}   r_{k}^{d+1}\alpha_{\mu}(2C_{2}B_{j,k})+ r_{k}^{d+1}\alpha_{\mu}(2C_{2}B_{i,l})\\ &\lec_{A,\tau,C_2} r_{k}^{d+1}\alpha_{\mu}(x,C_{2}^2 r_{k}).
\end{align*}
The lemma now follows from \eqref{e:angle-general1} as $x_{j,k}\in B(x'_{j,k}, r_k)\cap E_0$, and $d(x'_{j,k})<r_k$ then $d(x_{j,k})\le 2r_k$ applying Remark \ref{remark2}.
\end{proof}

Let $\Sigma_{k},\Sigma,\sigma_{k},g$, and so forth be the data we obtain from applying Theorem \ref{t:DT}. 
If $E_k=\emptyset$ then for all $j\ge k$ $E_j=\emptyset$ and the construction stops.

Observe that $V_{k}^{10}\subset V_{0}^{10}= B(0,10)$, and so by \eqref{e:v10c},
\begin{equation}\label{e:planar-sigma}
\Sigma\setminus B(0,10)=P_{0}\setminus B(0,10).
\end{equation}
Observe also  that in our scenario (recalling $B_0$ is closed)
\begin{equation}\label{e-infty=z}
E_{\infty} = Z\subset  B_0,
\end{equation}
which might a priori be empty.
Note that if $x\in V_{k}^{40}$ for infinitely many $k$, then $d(x)=0$ thus we define $k(x)$ for $x\in \Sigma\cap 40B_{0}\setminus Z$ as follows:
\begin{align}\label{k(x)}
\hbox{For }x\in \Sigma\cap 40B_{0}\setminus Z,\hbox{ let }k(x)\hbox{ be the smallest integer }k\hbox{ for which }x\not\in V_{k}^{40}. 
 \end{align}
 %For $x\in \Sigma\cap 40B_{0}\setminus Z$, let $k(x)$ be the smallest integer $k$ for which $x\not\in V_{k}^{40}$.
Since $0\in \{x_{j,k} \}_{j\in J_{0}}$, we know $10B_{0}\subset V_{0}^{40}$ and hence $k(\cdot)>0 $ is well defined. 
 
Since $k(x)$ is minimal, $x\in V_{k(x)-1}^{40}$, and so 
\begin{equation}
\label{e:ing99}
x\in 40 B_{j,k(x)-1}\quad\mbox{ for some $j\in J_{ k(x)-1}$.}
\end{equation}
 Thus, $B(x,r_{k(x)})\subset 41 B_{j,k(x)-1}$. 

\begin{lemma}\label{lem6.3}
For $x\in \Sigma\cap 40B_{0}\setminus Z$, and recalling the notation from Theorem \ref{t:DT}, 
\begin{equation}
\label{e:ingraph}
B(x,r_{k(x)}) \cap \Sigma = B(x,r_{k(x)})\cap \Sigma_{k(x)}=B(x,r_{k(x)})\cap  \Gamma_{j,k(x)-1},
\end{equation}
for some $j\in J_{ k(x)-1}$.
\end{lemma}
\begin{proof}
Indeed, the second equality is from \eqref{e:49graph} (keep in mind for later that $\Gamma_{j,k(x)-1}$ is a $C\ve$-Lipschitz graph over $P_{j,k(x)-1}$). To show the first identity, notice that since $x\not\in V_{k(x)}^{40}$, 
\[
B(x,r_{k(x)})\subset \bR^{n}\setminus V_{k(x)}^{39}.\]
Note that by \eqref{V2}, for $k>k(x)$,
\[
\{x_{i,k}\}_{i\in J_{k}}
\subset V_{k-1}^{2} = \{y\in \bR^{n}: \dist(y,\{x_{i,k}\}_{i\in J_{k-1}})<2r_{k-1}\}.
\]
By iterating this via the triangle inequality and recalling that $r_{k}=10^{-k}$, we get
\begin{align*}
\{x_{i,k}\}_{i\in J_{k}}
& \subset \{y\in \bR^{n}: \dist(y,\{x_{i,k(x)}\}_{i\in J_{k(x)}})<2r_{k-1}+\cdots + 2r_{k(x)}\}\\
& \subset\{y\in \bR^{n}: \dist(y,\{x_{i,k(x)}\}_{i\in J_{k(x)}})<\tfrac{20}{9}r_{k(x)}\}
\subset V_{k(x)}^{3}.
\end{align*}
In particular, $V_{k}^{10}\subset V_{k(x)}^{13}\subset V_{k(x)}^{39}$, hence $B(x,r_{k}(x))\subset \bR^{n}\setminus V_{k}^{10}$ and by \eqref{e:v10c}, $\sigma_{k}$ is the identity on $B(x,r_{k(x)})$ for all $k\geq k(x)$. By Theorem \ref{t:DT} \rf{eqahj0}, \rf{e:sigmak}, \rf{e:v10c}, the first equality of \eqref{e:ingraph} holds and this finishes the claim. 
\end{proof}

\begin{lemma}\label{lem6.4}
We have
\begin{equation}
\label{e:rsimd}
 \frac1{10}\,r_{k(x)}\leq d(x)\leq 60\,r_{k(x)} \;\;\mbox{ for \;$x\in \Sigma\cap  40 B_{0}\setminus Z$}.
\end{equation}
\end{lemma}

\begin{proof}
	Since  $x\in \Sigma\cap  40 B_{0}\setminus Z$, $d(x)>0$ and there is $k=k(x)$ as in \rf{k(x)}. Assume $d(x)<r_{k}/10$. 	
	Let $y\in E_{0}\cap B_{0}$ be such that $\delta(y)+|x-y|< 2d(x)$. Since $d(y)\leq \delta(y)$ then
	$$d(y)\leq 2d(x) <\frac15\,r_k <r_k.$$
	Hence, $y\in E_{k} $ and by \eqref{e:E<V}, there is $x_{j,k}$ so that $|x_{j,k}-y|\leq \frac{3}{2} r_{k}$, thus
	\[
	|x-x_{j,k}|\leq |x-y| + |y-x_{j,k}| \leq \frac{1}{5} r_{k}+\frac{3}{2} r_{k}<2r_{k},\]
	which is a contradiction since $x\not\in V_{k}^{40}$ (by the definition of $k(x)$). Thus, $d(x)\geq r_{k}/10$. 
	
	To prove the upper bound, recall that $x\in 40 B_{i,k(x)-1}$ for some $i\in J_{k-1}$ (see the paragraph before Lemma \ref{lem6.3}). Thus, there is some
 $x'\in \{x_{i,k-1}'\}_{i\in J_{k-1}}$ such that $|x-x'|\leq 41 r_{k-1}$.
 Since $\{x_{i,k-1}'\}_{i\in J_{k-1}}\subset E_{k-1}$, we have $d(x')\leq r_{k-1}=10r_{k}$ by definition. Since $d(\cdot)$ is $1$-Lipschitz, then we get
	\[
	d(x)\leq |x-x'|+d(x')
	\leq  41 r_{k-1}+10r_{k}\leq 60 r_k.
	\]
\end{proof}

Let $\eta=1/1000$ and $\{\mathcal B_j\}_{j=1}^N$ be a Besicovitch subcovering (see \rf{besicovitch-1} and \rf{e:besicovitch}) of the collection 
\begin{equation}\label{besic}
\{B(x, \eta d(x) ):x\in \Sigma\setminus Z\}
\end{equation}
where, by the previous lemma, for our choice of $\eta$,
\begin{equation}\label{e:s<rifxin40}
s_{j}: = \eta \,d(\xi_{j})<\frac{3r_{k(\xi_{j})}}{10}= 3r_{k(\xi_{j})-1}\;\; \mbox{ if } \;\; \xi_{j}\in 40B_{0}\cap \Sigma\setminus  Z.
\end{equation}
Since $d(\xi_j)\leq |d(0)| + |\xi_j|\leq 1 + 40$, we also have
\begin{equation}\label{e:s<rifxin40*}
s_{j}\leq \frac{41}{1000}\leq \frac1{20}.
\end{equation}

For $\xi_{j}\in 40B_{0}$, let $B_{i,k(\xi_j)-1}$ be a ball such that $\xi_j\in 40 B_{i,k(\xi_j)-1}$ (recall \rf{e:ing99}), 
so that $3B_j\subset 49 B_{i,k(\xi_j)-1}$. Denote
\begin{equation}\label{e:btwid}
\wt{B}_{j}= C_2 \,B_{i,k(\xi_j)-1}, \quad \xi_{j}'=x_{i,k(\xi_j)-1}.
\end{equation}
Also, set 
\[ 
P_{j} = P_{i,k(\xi_j)-1},\quad
\Gamma_{j} = \Gamma_{i,k(\xi_j)-1},  \;\mbox{ and }\; \LL_{j} = \LL_{i,k(\xi_j)-1}= c_{\wt{B}_{j}} \HH^{d}|_{P_{j}}
\]
so that by \eqref{e:49graph},  %and \eqref{e:s<rifxin40}, if $\xi_{j}'\in 40B_{0}$, 
$\Gamma_{j}$ is a graph of a $C\ve$-Lipschitz function $A_{j}$ over $P_{j}$ so that
\begin{equation}\label{e:3bjcapSigma}
3B_{j} \cap \Sigma = 3B_{j}\cap \Gamma_{j} 
\end{equation}
and since $\xi_{j}'\in P_{j}$ and $A_{j}$ is $C\ve$-Lipschitz, $\dist(\xi_{j},P_{j})\lec \ve r_{k(\xi_{j})}$.  These facts imply that,
for 
\[
\sigma:= \HH^{d}|_{\Sigma},
\]
we have
\begin{equation}
\label{e:dHPj}
F_{3B_{j}}(\sigma,\cH^{d}|_{P_{j}}) \lec \ve r_{k(\xi_{j})}^{d+1} \stackrel{\eqref{e:rsimd}}{\approx} \ve d(\xi_{j})^{d+1}\approx \ve s_{j}^{d+1}  \;\; \mbox{ if }\xi_{j}\in 40B_{0}.
\end{equation}
%If $x_{j}\not\in 40B_{0}$, then since $\eta=1/1000$,
%\[
%s_{j}=\eta d(x_{j})<\eta(1+|x_{j}|)<\eta |x_{j}| (1/40+1)<2\eta |x_{j}| < \frac{|x_{j}|}{2}\]
%which implies $B_{j}\subset 20B_{0}$, and so $B_{j}\cap \Sigma = B_{j}\cap P_{0}$, so \eqref{e:dHPj} still holds when $x_{j}\not\in 40B_{0}$.
 
Note also that, by \eqref{e:muLjk1},
\begin{equation}\label{e:muLjk}
F_{\wt{B}_{j}}(\mu,\LL_{j})\lec_{A,\tau} \alpha_{\mu}(2\wt{B}_{j})\,r_{k}^{d+1}  \lec \ve r_{k}^{d+1}.
\end{equation}
\begin{remark}\label{r:wheredefined}
Since $\Sigma$ coincides with $P_0$ in $B(0,10)^c$, we do not need to define $\wt{B}_{j}$ and other related terms for $\xi_{j}\not\in 40 B_{0}$.
\end{remark}
 
Next we record the following lemma for later. 

\begin{lemma}\label{l:2Bjinball}
If $2B_{j}\cap 39B_0\neq\varnothing$, then $2B_j\subset 40B_0$. 
\end{lemma}

\begin{proof}
Since $2B_{j}\cap B(0,39)\neq\varnothing$, by \rf{e:s<rifxin40*}
\[
4s_{j} 
=4\eta d(\xi_{j})
\leq 4\eta (|\xi_{j}|+d(0))
\stackrel{\eqref{e:delta0}}{\leq} 4\eta(39+2s_{j}+1)
=160\eta + \frac{8}{20}\eta\le \frac{1}{5} \]
and since $\eta=1/1000$. Thus $\diam 2B_{j} =4s_{j}<1$, and so $2B_{j}\subset B(0,40)$. 
\end{proof}

%$x_{j}=y+A(y)$ for some $y\in P_{0}$ with $|x_{j}-y|\lec \ve r_{k(x)}$, recalling that $x_{j}'\in P_{j}$,
%\[
%\dist(P_{j},x_{j}) 
%=|A(y)|\leq  |A(x_{j}')|+ C\ve |x_{j}'-y|
%\leq 0 + C\ve |x_{j}'-x_{j}|+|x_{j}-y|
%\leq \dist(P_j,x_{j,k})+|x_{j,k}-x_{j}| \lec \ve r_{k(x)}
%\]

\begin{rem}
	It may seem like overkill to invoke Theorem \ref{t:DT} to construct a Lipschitz graph. We could instead construct a graph directly as in \cite{DS}. However, our approach is not very harmful because the condition in (xi)  of Theorem \ref{t:DT} will allow to get nice bounds on the $L^{2}$-norm of the gradient of the graph which will be useful to deal with the
	stopping condition $\BA$.
%	 Indeed, because $\Sigma$ is a $C\ve^{\frac{1}{2}}$-Lipschitz graph, it is immediate that $\sigma$ is AD-regular with constants only depending on $d$, that is,
%\[
%\sigma(B(x,r))\leq (1+C\ve^{\frac{1}{2}}) (2r)^{d} \mbox{ for all }x\in \Sigma \mbox{ and }r>0.
%\]
%
%However, we will need better estimates than this, and this is where \ref{t:DT} comes in. 
\end{rem}

\begin{lemma}\label{lem-sigma-AD} 
Let
\[
\sigma:= \HH^{d}|_{\Sigma}.
\]
For $\ve>0$ small and $C_{1}$ large enough (depending on $C_{2}$ but independent of $\ve$), the map $g$ is $(1+C\ve)$-bi-Lipschitz (with $C$ depending on $A,\tau$, and $C_{1}$). In particular, $\sigma$ is $AD$-regular with constant close to 1. For $\ve>0$ small enough, 
	\begin{equation}
\label{e:sigma-AD}
2^{-1} \cdot(2r)^{d} < (1-C\ve)(2r)^{d}\leq \sigma(B(x,r))\leq (1+C\ve) (2r)^{d}<2 \cdot(2r)^{d}\quad
\mbox{for all $x\in\Sigma$.}
\end{equation}
\end{lemma}

Recall that $\cH^{d}(B(x,r)\cap \R^{d})=(2r)^{d}$ for $x\in \R^{d}$, so \eqref{e:sigma-AD} is saying that surface measure is very close to being uniform like planar surface measure.

\begin{proof}
By Theorem \ref{teo61}, to prove the lemma it suffices to show that 
\begin{equation}\label{sum-bound}
\sum_{k\geq 0} \ve_{k}(g_{k}(y))^{2}\lesssim\ve^2,
\end{equation}
for $g_{k}(y)= \sigma_{k}\circ\cdots \circ \sigma_{0}(x)$ and for all $y\in P_{0}$.

Suppose first that $x:=g(y)\in \Sigma\cap 10B_{0}\setminus  Z$. 
Then $k(x)<\infty$.
%Since $k(x)$ is minimal, $x\in V_{k}^{40}$ for $k<k(x)$, and so $x\in 40B_{j,k}$ for some $j\in J_{k}$. 
By \eqref{e:sky-y},  $x_{k}=g_{k}(y)$ satisfies $x=\lim x_{k}$ and 
\[
|x_{k}-x|\lec \ve r_{k}.
\]
Note that for $k\geq k(x)$, 
 $x_{k}=x$ by \eqref{e:v10c}, taking also into account that $x\not\in V_{k(x)}^{40}$ by the definition of $k(x)$. In fact, all $z\in B(x,r_{k(x)})$ satisfy $z\not\in V_{k(x)}^{39}$ and thus
 $\sigma_{k}$ is the identity map in $B(x,\,r_{k(x)})$ %in some ball $B(x,c\,r_{k(x)})$, for some absolute constant $0<c<1$.
So it follows that $\ve_{k}(x_{k})\neq0$ only for $k\leq Ck(x)$ for some universal constant $C$. \jonas{I added here and below $Ck(x)$ instead of $k(x)$, since $\ve_{k}$ asks for points in a large ball of size like $100r_k$.} %$k\leq c_1 k(x)$ for some universal constant $c_1>1$. Let $k\leq c_1k(x)$.
Let $k\leq Ck(x)$. 

%For $\ve>0$ small enough, $|x_{k}-x|<d(x)\approx r_{k}$. 
%\[
%\frac{32d(x_k) }{C_{0}}
%\leq \frac{32(d(x)+|x-x_{k}|)}{C_{0}}
%<\frac{33 d(x)}{C_{0}}\]
Let $z\in E_{0}\cap B_{0}$ be such that $|x-z|<2d(x)\lec r_{k(x)}$ by Lemma \ref{lem6.4}. Then % since $k\leq c_1 k(x)$,
\[
|x_{k}-z|\leq |x_{k}-x|+|x-z|\lec \ve r_{k}+ r_{k(x)} \lec   r_{k}\]
where the implicit constant is universal, and so we can pick $C_{2}$ large enough so that 
\begin{equation}\label{6.30A}
B(x_{k},C_{2}^2r_{k})
\subset B(z,C_{2}^2 r_{k}+C r_{k})
\subset B(z,2C_{2}^2r_{k}).
\end{equation}
Thus,
\begin{equation}\label{6.30B}
\ve_{k}(x_{k})
\stackrel{\eqref{e:ve<a}}{\lec}_{A,\tau,C_2}\alpha_{\mu}(x_{k},C_{2}^2 r_{k})
\stackrel{\eqref{e:alpha-monitone}}{\lec} \alpha_{\mu}(z,2C_{2}^2 r_{k}).
\end{equation}
Hence, for $C_{1}>2C_2^2$ large enough, using \rf{working-hyp-2A} we obtain
\begin{equation}\label{6.30C}
\sum_{k= 0}^{\infty}\ve_{k}(x_{k})^{2} 
\lec_{A,\tau}  \sum_{k=0}^{Ck(x)} \alpha_{\mu}(z, C_{1} r_{k})^{2} 
\stackrel{\eqref{e:alpha-monitone}}{\lec}_{A,\tau} \int_{0}^{C_{1}} \alpha_{\mu}(z,r)^{2} \frac{dr}{r}<\ve^{2}.
\end{equation}

%{\color{blue} Some additional explanations included below}\jonas{Corrected some things related to this $Ck(x)$ thing I mentioned}

When $x\in \Sigma\setminus 10 B_{0}=P_0\backslash B(0,10)$ (see \rf{e:planar-sigma}), then $\ve_{k}(x_{k})=0$ for $k\geq C$ for some large $C>0$ because $V_0^{10}=10B_0$ and since by \rf{e:E<V-1} $x_{j,k+1}\in V_k^{3/2}$ for all $k$, hence $x_{j,k}\in V_{0}^{3}$ for all $k\geq 0$. But  $(V_1^{10})^c\subset (V_0^3)^c\subset (10B_0)^c$, which means that $\ve_{k}=0$ for $k$ large enough. %
%then $x_{j,1}\in V_0^{3/2}$. Hence 
%by \eqref{e:v10c} $\sigma_1=Id$ on $(V_1^10)^c\subset (V_0^3)^c\subset (10B_0)^c$. By induction $\sigma_j=Id$ on $(V_j^10)^c\subset (10B_0)^c$. 
This proves \eqref{e:DTsum} in the case that $x\not\in Z$.

If $x\in Z$, then $d(x)=0$ and for each $k$ there are $x'_{j,k}\in E_k$ such that  $|x-x'_{j,k}|<r_k(1+1/10)$, and $x_{j,k}$ such that $|x-x_{j,k}|<2r_k$, thus $x\in E_\infty$. Let $y\in \R^d$ be such that $x=g(y)$ and let $x_k=g_k(y)$, recall that $|x-x_k|\lec r_k$. Then by \rf{6.30A} and \rf{6.30B} we have that 
\begin{equation}\label{6.30C}
\ve_k(x_k)\lec \alpha_\mu(x, 2C_2^2r_k).
\end{equation}
We conclude \rf{6.30C} as before with $x$ instead of $z$. Thus \rf{sum-bound} follows.

%is proved similarly. In fact, the arguments are a little easier, and so we will just outline the necessary changes: we have $d(x)=0$, and we just replace $z$ above with $x$ and we set $k(x)=\infty$ and observe that $x\in Z\subset E_{0}$. Thus, for $\ve>0$ small enough, $g$ is $(1+C\ve^2)$-bi-Lipschitz. Now \eqref{e:sigma-AD} follows immediately.
\end{proof}

%{\color{blue} There are a number of important modifications in the proof of Lemma \ref{lem6.8}}
\begin{lemma}\label{lem6.8}
	There is a constant $C=C(n)>0$ such that the surface $\Sigma$ is a $C\ve^{\frac14}$-Lipschitz graph over $P_{0}$, that is, there is a $C\ve^{\frac14}$-Lipschitz function $h:P_{0}\rightarrow P_{0}^{\perp}$ such that 
	\begin{equation}\label{S-graph}
	\Sigma = \{x+h(x):x\in P_{0}\}.
	\end{equation}
\end{lemma}
\begin{proof}
By Lemma \ref{l:good} $\cH^{d}(Z)=0$. Moreover $Z$ is a closed subset of $\Sigma$ since 
$d$ is continuous. In particular, we infer that for $\sigma$-almost every $x\in \Sigma$, $d(x)>0$. Hence, for $\sigma$-almost every $x\in 10B_{0}\cap \Sigma$, $B(x,r_{k(x)})\cap \Sigma$ is a $C\ve$-Lipschitz graph over $P_{j,k(x)-1}$ by \eqref{e:ingraph}. Recalling that $d(x)\approx r_{k(x)}$, by 
 \eqref{e:good-range'}  $P_{j,k(x)-1}$ is 
 a $C\ve^{\frac14}$-Lipschitz graph over $P_{0}$, and hence so is $B(x,r_{k(x)})\cap \Sigma$ (with another constant $C$). On the other hand, if $x\in \Sigma\setminus 10B_{0}$, then $x\in P_{0}$ by \eqref{e:planar-sigma}.  Thus, we can cover $\Sigma$ up 
 to a set of surface measure zero by balls $B_j$ in which $\Sigma$ is a $C\ve^{\frac14}$-Lipschitz graph over $P_{0}$. 
 By the previous lemma, $g:P_{0}\rightarrow \Sigma$ is bi-Lipschitz, and so for a.e.\ $z\in P_{0}$, $g(z)\in \bigcup_{j=1}^N \mathcal{B}_j$. 
 
 {%\color{red} 
 The initial goal is to show that for any $x,y\in \Sigma$, $|\pi_{P_{0}^{\perp}}(x-y)|\lec \ve^{\frac14}\,|x-y|$, which would guarantee that $\Sigma$ is included in a Lipschitz graph with constant bounded above by a constant times $ \ve^{\frac14}$. Let  $x,y\in \Sigma$ and $x', y'\in P_0$ be such that $g(x')=x$ and $g(y')=y$. Note that \rf{e:sky-y} implies that $|x'-g(x')|=|x'-x|\lec\ve$ and $|y'-g(y')|=|y'-y|\lec \ve$. If $|x-y|\ge 1/10$
 then
 \begin{equation}\label{x-y-big}
 |\pi_{P_{0}^{\perp}}(x-y)|\le |\pi_{P_{0}^{\perp}}(x)|+|\pi_{P_0^{\perp}}(y)|\le |x'-x| +|y'-y|\lec \ve\lec\ve|x-y|. 
 \end{equation}
 Thus we assume that $|x-y|< 1/10=r_1$. Hence there exists $k\ge1$ such that $r_{k+1}\le |x-y|<r_k$. We consider two cases: either $\max\{k(x),k(y)\}>k$ or $\max\{k(x),k(y)\}\le k$. 
 
 In the first case we assume
 without loss of generality we assume that $k(x)>k$. If $x\not\in Z$ and $y\in B(x,r_{k(x)})\cap \Sigma$, then 
  \begin{equation}\label{k-small} 
 |\pi_{P_{j,k(x)-1}^{\perp}}(x)-\pi_{P_{j,k(x)-1}^{\perp}}(y)|\lec \ve |\pi_{P_{j,k(x)-1}}(x)-\pi_{P_{j,k(x)-1}}(y)|\lec \ve\,|x-y| 
 \end{equation}
 as $B(x,r_{k(x)})\cap \Sigma$ is a $C\ve$ Lipschitz graph over $P_{j,k(x)-1}$ for some $j\in J_{ k(x)-1}$ by Lemma \ref{lem6.3}. Since $x_{j,k(x)-1}'\in E_0\cap B_0$ by the choice of $x_{j,k}$, $P_{j,k}$ and $x_{j,k}'$ (see \rf{coice-x'} and line above \rf{e:xP'}) we have by Lemma \ref{l:good-range} that $\angle(P_{j,k(x)-1},L_{B_{0}})\leq \ve^{\frac{1}{4}}$. Thus a simple geometric argument ensures that 
 \begin{equation}\label{k-small-2} 
 |\pi_{P_{0}^{\perp}}(x)-\pi_{P_{0}^{\perp}}(y)| \lec\ve^\frac{1}{4}|\pi_{P_{0}}(x)-\pi_{P_{0}}(y)|,
 \end{equation}
 provided $x,y\in \Sigma\backslash Z$ and $\max\{k(x),k(y)\}>k$.
 
 In the case when $\max\{k(x),k(y)\}\le k$, for $g_k$ as in $(xii)$ in Theorem \ref{teo61}, denote by $x_k=g_k(x')\in \Sigma_{k+1}$ and $y_k=g_k(y')\in \Sigma_{k+1}$. Iterating \rf{e:sky-y} we have $|x_k-x|\lec \ve r_k$ and $|y_k-y|\lec \ve r_k$. Thus , for $\ve>0$ small enough,
 \[
 |x_k-y_k|\le |x-y|+x_k-x|+|y_k-y|\le r_k + C\ve r_k\le 2r_k.
 \]
By the construction there is $x_{j,k}$ such that $|x_k-x_{j,k}|\le 10 r_k$ and 
 therefore $|y_k-x_{jk}|\le 12r_k$. Hence 
 \[
 x_k,y_k\in \Sigma_{k+1}\cap B(x{j,k}, 12 r_k)\subset \Sigma_{k+1}\cap D(x_{j,k}, P_{j,k}, 49 r_k)=\Gamma_{j,k}\cap D(x_{j,k}, P_{j,k}, 49 r_k)
 \]
 as in \rf{e:49graph} where
 by $(vii)$ in Theorem \ref{teo61} $\Gamma_{j,k}$ is a graph over $P_{j,k}$ with constant less than $C\ve$. A similar argument to the one used above yields $\angle(P_{j,k},L_{B_{0}})\leq \ve^{\frac{1}{4}}$ where we 
 also appeal to Remark \ref{remark1} with $c_0=1/100$, which ensures that \rf{k-small-2} also holds in this case. 
 
 The inequality \rf{k-small-2}
  proves that there exists $C(n)\ve^{\frac14}$-Lipschitz function $h:P_{0}\rightarrow P_{0}^{\perp}$ such that 
 \begin{equation}\label{S-in-graph}
	\Sigma \backslash Z\subset \{x+h(x):x\in P_{0}\}=\Gamma.
	\end{equation}
For $x\in Z\subset \Sigma$ by \rf{e:sigma-AD} since $\mu(Z)=0$ there exists a sequence $\tilde x_k\in \Sigma\backslash Z$ such that $\tilde x_k\to x$ as $k\to\infty$. By \rf{S-in-graph} there is $y_k\in P_0$ 
such that $\tilde x_k=y_k+h(y_k)\to x$, thus $|y_k-y_\ell|\le |y_k+h(y_k) -(y_\ell+h(y_\ell))| +|h(y_k)-h(y_\ell)|\le |\tilde x_k-\tilde x_\ell| + C\ve^{\frac{1}{4}}|y_k-y_\ell|\le 2|\tilde x_k-\tilde x_\ell|$ which ensures that $\{y_k\}_k$ is a Cauchy sequence. Let $y=\lim_{k\to\infty}y_k\in P_0$. Since $h$ is Lipschitz continuous $\tilde x_k=y_k+h(y_k)\to y+h(y)=x$. Thus $\Sigma\subset \Gamma$.
 Since $\Sigma$ and $\Gamma$ are both closed if here is $x+h(x)\in \Gamma\backslash \Sigma$ with $x\in P_0$ then since $\Sigma\backslash 10B_0=P_0\backslash 10B_0$ there exists $\rho>0$ such that 
 $B(x+h(x),\rho)\cap (\Sigma\cup (10B_0)^c=\empty$. Then the map $\pi_{P_0}\circ g:P_0\rightarrow P_0\backslash \pi_{P_0}(\Gamma \cap B(x+h(x),\rho)$ is bi-Lipschitz and satisfies $\pi_{P_0}\circ g= Id$ on
 $P_0\backslash 10B_0$ which is a contradiction (via a minor degree argument).
 }\end{proof}
 
 {%\color{red}
 It is worth emphasizing that the reason why in this case $\Sigma$ is a Lipschitz graph in contrast with the general $\Sigma$ constructed in Theorem \ref{teo61} is that since we have that $\cH^d(Z)=0$ the 
 construction always stops before the tilt between the original plane $P_0$ and the good approximating plan at a given scale gets larger than $\ve^{\frac{1}{4}}$.}
 
% Hence, for any $x,y\in \Sigma$, it is not hard to find using Fubini's theorem a polygonal curve $L$  connecting $x'=g^{-1}(x)$ and $y'=g^{-1}(y)$ such that for $\cH^{1}$-a.e. $z\in L$, $g(z)\in \bigcup_j B_j$, and also so that $L$ is the image of a bi-Lipschitz map $\gamma_{0}:[0,|x-y|]\rightarrow L$. In particular, if $\gamma= g\circ\gamma_{0}$, this implies $|\grad (\pi_{P_{0}^{\perp}}\circ\gamma)|\lec \ve^{\frac14}$ a.e.
 %(where $\pi_{P_{0}^{\perp}}$ denotes the orthogonal projection onto $P_{0}^{\perp}$, the $(n-d)$-plane orthogonal to $P_0$), hence
%\[
%|\pi_{P_{0}^{\perp}}(x-y)|
%\leq \int_{0}^{|x-y|} |\grad \pi_{P_{0}^{\perp}}\circ \gamma|
%\lec \ve^{\frac14}\,|x-y|.\]
%Since this holds for all $x,y\in \Sigma$, this implies $\Sigma$ is a $C\ve^{\frac14}$-Lipschitz graph.

\begin{lemma}
We have
\begin{equation}
\label{e:h-int}
\int_{P_{0}} |D h|^{2}\,d\HH^d
\lec \ve.
\end{equation}
\end{lemma}

\begin{proof}
Since $\sigma(x)=x$ outside $B(0,10)$ by \eqref{e:v10c}, using some simple degree theory as in the proof of  \cite[Theorem 13.1]{DT12}, we know that 
\[
\pi_{P_{0}}(B(0,10)\cap\Sigma)=B(0,10)\cap P_{0}. 
\]
{%\color{red}
and $h|_{P_{0}\setminus B(0,10)}\equiv 0$. Let the function $f:P_0\to\R^n$ be defined by $f(y)=(y,h(y))$. Since $h$ is a $C\ve^{\frac{1}{4}}$ Lipschitz function, 
by the area formula the generalized Jacobian, $J_f$ of $f$ is given by:
\begin{equation}\label{jacobian}
J_j=\sqrt{\text{det}\,\left(\delta_{ij}+\frac{\partial h}{\partial x_i}\frac{\partial h}{\partial x_j}\right)}\ge 1 +C |D h|^{2},
\end{equation}
where we have used the fact that $ |D h|\le C\ve^{\frac{1}{4}}$ and a Taylor expansion for this type determinant.}\jonas{I'd say maybe do both since then it's a bit more self-contained. I don't have the book (and can't pirate it even), do you want to add the argument here?}

%{\color{blue} 
%}
%we have
%\begin{equation}\label{eqco33}
%|D h(y)|\lesssim (J_{f}-1)(y)\quad\mbox{ for all $y\in P_0$.}
%\end{equation}
%This inequality is well known and has already appeared in previous works, such as \cite{AS16}. However, for the reader's convenience we explain how this can be obtained:
%let $\lambda_{1}$ denote the largest singular value of $Dh(y)$ and let $v_{1}$ denote the associated unit vector. Then $|Dh(y)|=|Dh(y)(v_{1})|=\lambda_{1}$. Since $|Df(y)(v)|=(|v|^{2}+|Dh(y)(v)|^{2})^{\frac{1}{2}}$ %$|Df(y)|=\sqrt{1+\lambda_{1}^{2}}$, which coincides with the largest singular value of $Df(y)$.
%Now the singular values for $|Df(y)|$ are all at least $1$ (since $f$ is a graph), and since the generalized determinant of $Df(y)$ is the product of the singular values, we have 
%$$|Dh(y)|^{2}=|Df(y)|^{2}-1\leq J_{f}(y)^{2}-1= (J_{f}(y)-1)(J_{f}(y)+1)\lesssim J_{f}(y)-1,$$
%taking into account that $f$ is Lipschitz for the last inequality. This concludes the proof of \rf{eqco33}.

From \rf{jacobian}, we get
$$
\int_{P_{0}} |D h|^{2}\,d\HH^d
\lec \int_{B(0,10)\cap P_{0}}  \!\!(J_{f}-1)\,d\HH^d
=\cH^{d}(B(0,10)\cap \Sigma)-\cH^{d}(B(0,10)\cap P_{0})
\stackrel{\eqref{e:sigma-AD}}{\lec} \!\ve,
$$
as wished.
\end{proof}

Note that the argument above can also be reduced by using standard results, see for example the proof of Lemma 23.10 \cite{Maggi}. 
 
Notice that the estimate \rf{e:h-int} follows from the $(1+C\ve)$-bilipschitz character of $f(y)=(y,h(y))$,
which in turn comes from the smallness of the $\alpha$-numbers ensured by the condition \rf{working-hyp-2A}. If, instead, we use the property that $h$ is $C\ve^{1/4}$-Lipschitz coming from the stopping condition $\BA$ involving the angles that the approximating $d$-planes form with $P_0$, we get the worse estimate
$$\int_{P_{0}} |D h|^{2}\,d\HH^d \lesssim \ve^{1/2},$$
which is not useful for our purposes. The sharper inequality
\rf{e:h-int} plays a key role 
later to show that the set of points where $\BA$ holds has small measure.

% {\color{blue} I don't see exactly the connection between the lemma below Tol09 immediately, should we add an explanation?. It is not clear what $0<f(x)\approx1$ means. Could you please clarify?}
% \jonas{I added an explanation below, and also changed 1/4 to 1/2 in equations 6.49 and 6.50 below.}

\begin{lemma}  \label{l:Xavi-lemma}
	Let $B$ be a ball centered on $\Sigma$, and $f$ a function such that 
	$$||f-f(x_{B})||_{L^\infty(3B\cap \Sigma)} \lec \ve^{\frac14}\quad \mbox{ and }f(x)\in (1/C,C)$$ uniformly for all $x\in 3B\cap\Sigma$ for some constant $C>0$. Then
	\begin{equation}\label{e:Xavi-lemma}
	\int_{B}\int_{0}^{r_{B}} \alpha_{f\sigma}(x,r)^{2}\frac{dr}{r}d\sigma(x) \lec \ve^{\frac12} r_{B}^{d}
	\end{equation}
	and there is a plane $P_{B}$ such that \xavi{I replaced $\sigma$ by $f\sigma$ inside $F_B(...)$ here}
	\begin{equation}\label{e:a<e}
	\alpha_{f\sigma}(B)\lec r_{B}^{-d-1} F_{B}( f\sigma,\cH^{d}|_{P_{B}})\lec \ve^{\frac12}.
	\end{equation}
\end{lemma}

{%\color{magenta}
This result is a direct consequence of \cite[Theorem 1.1]{Tol09} and Remark 4.1 immediately proceeding it, taking into account that
$\Sigma$ is a $C\ve^{\frac14}$-Lipschitz graph. The original statement was for $f\equiv 1$, but the same proof works for the preceding lemma. We sketch the adjustments below, using the notation of \cite{Tol09}. First, since $\Sigma$ is a $C\ve^{\frac{1}{4}}$-Lipschitz graph, by Whitney extension we can replace it with a $C\ve^{1/2}$-graph $\Gamma$ that agrees with $\Sigma$ in $3B$ but is constant outside $4B$, and also set $f=f(x_{B})$ outside $3B$. By rotating, we can assume $\Gamma$ is a graph along $\R^{d}$.

At the beginning of the proof of \cite[Theorem 1.1]{Tol09}, replace the function
\[
g(x) := \rho(\tilde{A}(x)) |J(\tilde{A})(x)|.
\]
with $\tilde{g} = f(\tilde{A})g$. There, the graph $\Gamma$ is a graph of a function $A$ and $\tilde{A}(x)=(x,A(x))$, and in our case $A$ is $C\ve^{1/4}$-Lipschitz. Then the proof continues verbetum. As in \cite[Remark 4.1]{Tol09}, we obtain from the proof that 
\[
\sum_{Q\in \mathcal{D}_{\R^{n}}} \alpha(Q)^{2}\mu(Q)
\lec \sum_{Q\in \mathcal{D}_{\R^{n}}} \beta_{1}(2Q)^{2}\mu(Q) + \sum_{I\in \mathcal{D}_{\R^{d}}} ||\Delta_{I} \tilde{g}||_{2}^{2}
\]
and  again, as in \cite[Remark 4.1]{Tol09}, 
\[
 \sum_{Q\in \mathcal{D}_{\R^{n}}} \beta_{1}(2Q)^{2}\mu(Q) \lec ||\grad A||_{2}^{2}.
 \]
 
In our situation, if $\tilde{A}(x)\in 3B$, then since $|g|\lec 1$,
\begin{align*}
|\tilde{g}(x)-1|
& =|f(\tilde{A}(x))g(x) -f(x_{B})|
\leq |f(\tilde{A}(x))-f(x_{B})|\cdot |g(x)| + |f(x_{B})|\cdot |1-g(x)|
\\
& \lec \ve^{\frac{1}{4}}\one_{3B} +|1-g(x)|.
\end{align*}

Hence, since $g\lec 1$ and $|f(x_{B})|\leq  1$, and because $f\equiv f(x_{B})$ outside $3B$ and $\grad A=0$ outside the projection of $4B$ into $\R^{d}$. 
\begin{align*}
\sum_{I\in \mathcal{D}_{\R^{d}}} ||\Delta_{I} \tilde{g}||_{2}^{2}
& \sim ||\tilde{g}-f(x_{B})||^{2} \lec \ve^{\frac{1}{2}} r_{B}^{d} + \int_{\R^{d}} |1-g|^{2}\\
& \lec \ve^{\frac{1}{2}} r_{B}^{d} + ||\grad A||_{2}^{2}
\lec \ve^{\frac{1}{2}}r_{B}^{d}. 
\end{align*}

Now \eqref{e:a<e} follows by applying the same argument to a slightly larger ball, say $\frac{3}{2}B$, then \eqref{e:Xavi-lemma} and Chebychev's inequality imply there must be $r\in (\frac{4}{3}r_{B},\frac{3}{2}r_{B})$ and $x\in \frac{1}{3}B$ such that 
\[
\alpha_{f\sigma}(B)
\lec \alpha_{f\sigma}(x,r) \lec \ve^{\frac{1}{2}}. 
\]

}

\vv
% *******************************************************************************************

\section{The approximating measure}\label{meas-approx}

Let $\theta_{j}$ be a partition of unity subordinated to the balls $B_{j}$, belonging to the Besicovitch subcovering of balls as in \rf{besic}, and satisfying
\begin{equation}\label{partition1}
0\leq \theta_{j} \leq 1 , \;\;
\Lip(\theta_{j})\approx s_{j}^{-1}, \;\; 
\frac{3}{2}B_j\subset\supp \theta_{j}\subset 2B_{j}
\end{equation}
and
\begin{equation}\label{partition2}
\theta:=\sum_{j} \theta_{j} \equiv 1\quad \mbox{ on } \quad\bigcup \tfrac{3}{2}B_{j}=: O.
\end{equation}
Note that, by the finite superposition of the balls $B_j$ we may assume that 
$$\theta_j(x)\approx 1\quad\mbox{ for all $x\in B_j$.}$$
Define 
\begin{equation}
\label{e:cj}
 c_{j}:= \isif{ \int \theta_{j} d\mu/\int \theta_{j}d\sigma & \mbox{ if } 2B_{j}\subset  B(0,10), \\ c_{C_{2} B_{0}} & \mbox{ if } 2B_{j}\not\subset B(0,10),}
\end{equation}
and
\begin{equation}\label{def-nu}
d\nu := \sum_{j} c_{j}\theta_{j} d\sigma.
\end{equation}
Note that $\nu\ll  \sigma=\cH^d|_{\Sigma}$.
Note that by the way we have chosen the $c_{j}$, we have 
\begin{equation}\label{e:nuaway}
\nu|_{B(0,10)^{c}}= c_{C_{2}B_{0}} \cH^{d}|_{P_{0}\setminus B(0,10)} = \LL_{C_{2}B_{0}}|_{B(0,10)^{c}}.
\end{equation}
%
%{\color{blue} Below there are some changes some 2$B_j$ have become $3B_j$ and some$3B_j$ have become $2B_j$. Once again I am assuming $3B(x,r)=B(x, 3r)$, rather than $3B(x,r)=B(3x,3r)$. If this is not the case further adjustments might be necessary. As I mentioned before we should define this early on. I did not do it as I wanted to make sure we were all in the same page. I did not see why Lemma \ref{l:2Bjinball} was necessary : if $2B_{j}\subset B(0,10)$ then $2B_j\subset 10B_0\subset 40B_0$.Please look at the proof}
%\jonas{You're right about the definition of 3B, I'll include a definition in the notation section.}

\begin{lemma}
	For all $j$ such that $2B_{j}\subset B(0,10)$,
	\begin{equation}\label{e:c-cB}
	|c_{j}-c_{\wt B_{j}}|\lec_{A,\tau} \ve .%{\color{blue}\hbox{ Note that the power of }\ve\hbox{ has changed}}
	\end{equation}
\end{lemma}

\begin{proof}
 
%Since $2B_{j}=B(\xi_j, 2s_j)\cap B(0,10)$, 
%Lemma \ref{l:2Bjinball} implies $\xi_{j}\in  2B_{j}\subseteq 40B_{0}$ (recall Remark \ref{r:wheredefined}). Thus, 
Recall that by \eqref{e:3bjcapSigma}, $3B_{j}\cap \Sigma$ is a $C\ve^{}$-Lipschitz graph over $P_{j}$, and also that $3B_j
\subset \wt B_j$, $r_{B_j}\approx r_{\wt B_j}$ (as in \rf{e:btwid}), and $P_j$ passes through the center of $\wt B_j$. Then 
as in \rf{e:dHPj}
\begin{equation}\label{e:sigma-Pj}
F_{3B_{j}}(\sigma,\cH^{d}|_{P_{j}})\lec_{A,\tau} \ve^{} s_{j}^{d+1} .
\end{equation}
Recalling that $\Lip_{1}(\theta_{j})\lec s_{j}^{-1}$, we have
\begin{align*}
s_{j}^{d} \,|c_{j}- c_{\wt{B}_{j}}| &  \lec  |c_{j}- c_{\wt{B}_{j}}|   \int \theta_{j}d\sigma 
= \av{ \int \theta_{j} d\mu -\int \theta_{j} c_{\wt{B}_{j}} d\sigma }\\
& \leq \biggl| \int \theta_{j} d\mu-\!\int \theta_{j} \underbrace{c_{\wt{B}_{j}} d\cH^{d}|_{P_{j}}}_{=d\LL_{j}}\biggr|
+c_{\wt{B}_{j}}   \av{\int \theta_{j} d\cH^{d}|_{P_{j}} -\! \int \theta_{j}d\sigma}
 \stackrel{\eqref{e:sigma-Pj}, \eqref{e:muLjk}}{\lec_{A,\tau} }\!\!
\ve^{} s_{j}^{d}.
\end{align*}
\end{proof}

\begin{lemma}
	The measure $\nu$ is $d$-AD-regular with constants depending on $A$ and $\tau$, that is,
	\begin{equation}
	\label{e:nu-AD}
	\nu(B(x,r))\approx_{A,\tau} r^{d} \quad\mbox{ for all }x\in \Sigma \mbox{ and }r>0.
	\end{equation}
\end{lemma}

\begin{proof}
	Note that by \eqref{e:cb-mub} and the definition of $c_{j}$, using Lemma \ref{lem5.6}, Remark \ref{remark1} and Remark \ref{remark2}, we have $c_{j}\approx_{A,\tau} 1$ for all $j$. Thus, 
	\begin{equation}\label{nu-sigma-comparison}
	d\sigma 
	\lec_{A,\tau} \sum_{j} c_{j} \theta_{j} d\sigma 
	\lec_{A,\tau} d\sigma,
	\end{equation}
	which by Lemma \ref{lem-sigma-AD} ensures that $\nu$ is $AD$-regular.
\end{proof}

\begin{lemma}
	If $2B_{i}\cap 2B_{j}\neq\varnothing$, then
	\begin{equation}
	\label{e:ci-cj}
	|c_{i}-c_{j}|\lec_{A,\tau} \ve^{}.
	\end{equation}
	Thus,
	\begin{equation}\label{e:almost-constant}
	\biggl|\sum_{j}c_{j} \theta_{j}(x)-c_{i}\biggr|\lec \ve^{} \quad\mbox{ for all }x\in \Sigma\cap 2B_{i}. 
	\end{equation}
\end{lemma}

\begin{proof}
Let $B_i$ and $B_j$ be such that $2B_{i}\cap 2B_{j}\neq\varnothing$.
 Then we have
\[
s_{i} = \eta d(\xi_{i}) \leq \eta d(\xi_{j}) + \eta|\xi_{i}-\xi_{j}|
\leq s_{j} + 2\eta (s_{i}+s_{j})
\]
and since $\eta=1/1000$ we get
$
s_{i}\leq (2+4\eta)s_{j} \leq 3s_{j}.
$
Therefore, by symmetry,
$$\frac13\,s_j\leq s_{i}\leq 3s_{j}.$$
Since $ \wt{B}_{i}\supset 2B_i$ and $\wt{B}_{j}\supset 2B_{j}$ for $C_{2}$ large enough, we
derive
\[
 \wt{B}_{i}\subset 4 \wt{B}_{j}\quad \mbox{ and }\quad\wt{B}_{j}\subset 4 \wt{B}_{i}.
\]

First assume both $2B_{i}$ and $2B_{j}$ are contained in $B(0,10)$.
Then,
\begin{align*}
|c_{i}-c_{j}|
& \leq |c_{i} - c_{\wt{B}_{i}}|
+ |c_{\wt{B}_{i}}- c_{4\wt{B}_{j}}|
+|c_{4\wt{B}_{j}}-c_{\wt{B}_{j}}|
+ |c_{\wt{B}_{j}}-c_{j}|\\
& \stackrel{\eqref{e:c-cB}, \eqref{e:cx-cy-general1}}{\lec_{A,\tau}} \alpha_{\mu}(4\wt{B}_{j}) +\ve^{} \lec \ve^{}.
%used to also cite \eqref{e:muLjk} in penultimate inequality, not sure why
\end{align*}

Suppose now that $2B_{i}\subset  B(0,10)$ and $2B_{j}\not\subset B(0,10)$.
From the fact that $s_j\leq 1/20$ by \rf{e:s<rifxin40*}, it follows that $B_j\cap B(0,9)=\varnothing$,
and thus $s_j\approx d(\xi_j)\approx 1$.
So  we have $s_{i}\approx s_{j}\approx 1$.
 Thus, \eqref{e:ci-cj} follows by similar estimates to above using the fact that $c_{j}=c_{C_{2}B_{0}}$. 
 
 Finally, if both $2B_{i}$ and $ 2B_{i}$ are not contained in $B(0,10)$, then $c_{i}=c_{j}= c_{C_{2} B_{0}}$ and so \eqref{e:ci-cj} is trivial. 
\end{proof}
%
%\begin{lemma}
%	If $2B_{i}\cap 2B_{j}$, then
%	\begin{equation}\label{e:c-c}
%	|c_{i}-c_{j}|\lec_{A,\tau} 1.
%	\end{equation}
%\end{lemma}
%
%\begin{proof}
%	Let $i,j$ be as in the lemma and assume both $2B_{i}$ and $2B_{j}$ intersect $B(0,10)$ (the other cases are similar).  Then
%	
%\begin{align*}
%c_{j} \int \theta_{j} d\sigma 
%=\int \theta_{j} d\mu
%& \leq \int \theta_{j}  \underbrace{c_{j}d\cH^{d}|_{L_{j}}}_{d\LL_{C_{1}\wt{B}_{j}}} + \alpha_{\sigma}(C_{1} \wt{B_{j}}) C_{1} s_{j} \sigma(\wt{B}_{j}) \\
%& \stackrel{\eqref{e:a<e}}{\leq} \int \theta_{j} d\mu + \alpha_{\mu}(\wt{B}_{j}) + \ve s_{j}^{d+1}\\
%\end{align*}
%\end{proof}

\begin{lemma}
	For all $x\in B(0,20)$ with $2 d(x)<r<20$, and $\theta$ as in \rf{partition1}
	\begin{equation}
	\label{e:mumutheta}
	F_{B(x,r)}(\mu,\theta \mu)\lec_{A} \ve^{\frac{1}{2}} r^{d+1}. 
	\end{equation} 
\end{lemma}

\begin{proof}
	Let $\phi\in \Lip_{1}(B(x,r))$. Since $4d(x)<r< C_1$, we have
	$$\frac{\mu(B(x,r)\setminus E_{0})}{ \mu(B(x,r))}\lesssim_{M}\ve^{\frac{1}{2}}.
$$ 
	\[
	\av{\int \phi\, d\mu - \int \phi\, d\mu|_{E_{0}}}
	\leq r\mu(B(x,r)\setminus E_{0})
	\stackrel{\eqref{eqe0*}}{\lesssim}_M r\,\ve^{\frac{1}{2}} \mu(B(x,r))\stackrel{\eqref{e:muAD}}{\lec_{A}} \ve^{\frac{1}{2}} r^{d+1} 
	\]
	which implies $F_{B(x,r)} (\mu, \mu|_{E_{0}})\lec_{A,M} \ve^{\frac{1}{2}} r^{d+1}$. Similarly, $F_{B(x,r)}  (\theta \mu, \theta \mu|_{E_{0}})\lec_{A} \ve^{\frac{1}{2}} r^{d+1}$. So it suffices to show
	that $\mu|_{E_{0}}=\theta \mu|_{E_{0}}$, which is equivalent to saying that $\theta\equiv1$ 
	$\mu$-a.e.\ in $E_0$.

	Let $y\in E_{0}\cap B(x,r)\setminus Z\subset B(0,40)$. We wish to show that $y\in\frac32 B_j$ for
	some $j$.
	Let $k=k(y)$. Then $k>0$ since $y\in 40 B_{0}$, and so $y\in V_{k-1}^{40}$, hence $y\in 40B_{j,k-1}$ for some $j\in J_{k-1}$. 
%	
%	If $d(y)\geq 1$, then $1\leq d(y) \leq 1/2+|y|$, and so $|y|\geq 1/2$. 
%	
%	Pick $k\geq 0$ so that 
%	\begin{equation}\label{e:r<d<r}
%	r_{k+1}\leq d(y)< r_{k}.
%	\end{equation}
%	Then $y\in E_{k}$ and so by \eqref{e:E<V} there is $j\in J_{k}$ so that $y\in \frac{3}{2}B_{j,k}$.
%	
%	
%	 Let $t=\dist(y,P_{j,k})$ and first assume $t>2d(y)/C_{0}$, then $t/2>4d(y)/C_{0}$ and we may apply \eqref{e:muAD}. For $C_{1}$ large enough, $B(y,t)\subset C_{1} B_{j,k}$, and so
%	\begin{align*}
%	t^{d+1} 
%	& \lec t\mu(B(y,t/2))
%	\leq \int \dist(z,B(y,t)^{c})d\mu \\
%	& \leq C_{1} r_{k}\mu( C_{1} B_{j,k}) \alpha_{\mu}(C_{1} B_{j,k})+  \int \dist(z,B(y,t)^{c})d\LL_{C_{1}B_{j,k}}\\
%	& \stackrel{\eqref{e:muAD}}{\lec_{A}} (C_{1} r_{k})^{d+1} \ve . 
%	\end{align*}
%	Hence,
%	\[
%	\frac{r_{k}}{C_{0}}=\frac{r_{k-1}}{10 C_{0}}\leq \frac{ d(y)}{10C_{0}}<  \frac{t}{20} \lec_{A,\tau}  C_{1} \ve^{\frac{1}{d+1}}r_{k}
%	\]
%	which is a contradiction for $\ve$ small enough, depending on $A,\tau,C_{0}$, and $C_{1}$. Thus, $t\leq 2d(y)/C_{0}<r_{k}/2$ for $C_{0}$ large enough. 
	For $\ve>0$ small enough  depending on $M$ and $\eta$, by Lemma \ref{l:LcapB}, 
	recalling that $d(y)\approx r_{k-1}$,
		\[
	\dist(y,P_{j,k-1})\leq \min\{\eta^{2}d(y),r_{k-1}\}.
	\]
	Also, since $y\in 40 B_{j,k-1}$, we have
	$
	\pi_{j,k-1}(y) \in 40B_{j,k-1}$. Then,
	for $\ve>0$ small,
	\begin{align*}
	\dist(y,\Sigma)
	& \leq |y-\pi_{j,k-1}(y)| + \dist(\pi_{j,k-1}(y),\Sigma)\\
	& \stackrel{\eqref{e:49r} \atop \eqref{e:ytosigma}}{\leq} \eta^{2}d(y) + C\ve r_{k-1}\\
	& \leq \eta^{2}d(y)+ C\ve d(y)
	< 2\eta^{2}d(y).
	\end{align*}
		Let $z\in \Sigma$ be such that 
	\[
	\dist(y,\Sigma)=|y-z|<2\eta^{2}d(y)\leq\frac{d(y)}{2}.
	\]
	 Then,
	\[
	d(z)\geq d(y) - |y-z| \geq d(y)/2>0.
	\]
	Hence, $z\in B_{j}$ for some $j$ and
	\[
	s_{j} =   \eta   d(\xi_{j})\geq  \eta (d(y) -  |\xi_{j}-z|-|z-y|)
	\geq  \eta \ps{d(y) - s_{j} -\frac{d(y)}{2} }\geq \frac{\eta}{2} d(y) - \eta s_{j}.
	\]
	Therefore, 
	\begin{equation}
	\label{e:sj>eta/4}
	s_{j}\geq  \frac{\eta}{2(1+\eta)}d(y)\geq \frac{\eta}{4} d(y).
	\end{equation}
	In particular, 
    \[
	y\in B(z,2\eta^{2}d(y))
	\subset B\ps{\xi_{j},s_{j}+2\eta^{2}d(y)}
	\stackrel{\eqref{e:sj>eta/4}}{\subset} B\ps{\xi_{j},s_{j}(1+8\eta)}\subset \tfrac{3}{2} B_{j}.\]	
	This proves
	\begin{equation}\label{e:E040}
	E_{0}\cap 40 B_{0} \subset \bigcup_j \tfrac32B_j=O.
	\end{equation}
	So $\theta\equiv 1$ on $O$ and, since $\mu(Z)=0$, we deduce that thus $\theta\equiv 1$ $\mu$-a.e.
	on $E_0$, as wished.
	\end{proof}
	
\begin{lemma}
	For $x\in \Sigma\cap B(0,20)$ and $0<r\leq15$,
	\begin{equation}
	\label{e:nuthetamu}
		F_{B(x,r)}(\nu,\theta \mu) \lec_{A,\tau} \sum_{2B_{j}\cap B(x,r)\neq\varnothing} \ve^{\frac{1}{4}} s_{j}^{d+1}.
	\end{equation}
\end{lemma}

\begin{proof}
Let 
\[
J(x,r) = \{j: 2B_{j}\cap B(x,r)\neq\varnothing\}.
\]
Observe that for $j\in J(x,r)$,  since $x\in B(0,20)$ and $r\leq15$, we have $2B_{j}\cap B(0,35)\neq\varnothing$, so $\xi_{j}\in 40 B_{0}$ by Lemma \ref{l:2Bjinball}. Let $\phi\in \Lip_{1}(B(x,r))$. Then
\begin{multline*}
\int \phi\, d\nu 
=\sum_{j\in J(x,r)} \int \phi \theta_{j} c_{j}\,d\sigma 
=\sum_{j\in J(x,r)}\int (\phi-\phi(\xi_{j})) \theta_{j}c_{j}\,d\sigma +\sum_{j\in J(x,r)}\!\! \phi(\xi_{j}) \int  \theta_{j}c_{j}\,d\sigma \\
=:I_{1}+I_{2}.
\end{multline*}
We will estimate the two sums separately. Note that $\Lip((\phi-\phi(\xi_{j})) \theta_{j})\lec 1$
and $c_j\approx_{A,\tau}1$, and so, with constants $C$ depending on $A$ and $\tau$,
\begin{align*}
I_{1} 
& = \sum_{j\in J(x,r)}\int (\phi-\phi(\xi_{j})) \theta_{j}c_{j}d\sigma
 \stackrel{\eqref{e:dHPj}}{\leq}  \sum_{j\in J(x,r)}\ps{\int (\phi-\phi(\xi_{j})) \theta_{j}c_{j}d \cH^{d}|_{P_{j}}+C\ve s_{j}^{d+1}}\\
 & \stackrel{\eqref{e:c-cB}}{\leq} \sum_{j\in J(x,r)}\ps{\int (\phi-\phi(\xi_{j})) \theta_{j}c_{\wt{B}_{j}}d \cH^{d}|_{P_{j}}+C\ve s_{j}^{d+1}}\\
& \stackrel{\eqref{cL-ppts}}{\leq} \sum_{j\in J(x,r)}\ps{\int (\phi-\phi(\xi_{j})) \theta_{j}d\mu + Cs_{j}\mu(\wt{B}_{j}) \alpha_{\mu}(\wt{B}_{j}) +C\ve s_{j}^{d+1}}\\
& \stackrel{\eqref{e:muAD}\eqref{e:muLjk}}{\leq} \sum_{j\in J(x,r)}\ps{\int (\phi-\phi(\xi_{j})) \theta_{j}d\mu + C\ve s_{j}^{d+1}}.\\
\end{align*}

Let $J_{1}$ be those $j\in J(x,r)$ for which $2B_{j}\subset B(0,10)$ and $J_{2}=J(x,r)\setminus J_{1}$. We split
\[
I_{2} = \sum_{j\in J_{1}}  \phi(\xi_{j}) \int  \theta_{j}c_{j}d\sigma 
+  \sum_{j\in J_{2}}  \phi(\xi_{j}) \int  \theta_{j}c_{j}d\sigma 
=:I_{21}+I_{22}.\]
We now estimate these two terms separately. First,
\[
I_{21}\stackrel{\eqref{e:cj}}{=}\sum_{j\in J_1} \phi(\xi_{j}) \int  \theta_{j}d\mu.
\]
For $I_{22}$, note that if $2B_{j}\cap B(0,10)^{c}\neq\varnothing$, then $s_{j}\approx 1$ and so $r\lec s_{j}$ and  there number of such $j$'s is bounded above by a constant only depending on $n$. Thus, $|\phi(\xi_j)|\lesssim r\lesssim1$ and moreover,
for $C_{2}$ large enough, $\bigcup_{j\in J_2} 2B_{j}\subset C_{2}B_{0}$. Also, since $\Sigma$ is a $C\ve^{\frac{1}{4}}$-Lipschitz graph over $P_{0}$, we know 
\[
F_{C_{2}B_{0}}(\sigma,\cH^{d}|_{P_{0}})\lec \ve^{\frac{1}{4}}.\] 
We can thus estimate
\begin{align*}
I_{22}
& \stackrel{\eqref{e:cj}}{=} \sum_{j\in J_{2}}  \phi(\xi_{j}) \int  \theta_{j}c_{C_{2}B_{0}}d\sigma 
\leq 
\sum_{j\in J_{2}}  \phi(\xi_{j}) \int  \theta_{j}\underbrace{c_{C_{\wt{B}_{j}}}d\cH^{d}|_{P_{0}}}_{\LL_{C_{2}B_{0}}} +C\sum_{j\in J_{2}}  |\phi(\xi_{j})| \, \ve^{\frac{1}{4}}\\
& \leq \sum_{j\in J_{2}}  \phi(\xi_{j}) \int  \theta_{j} d\mu +C\sum_{j\in J_{2}}  |\phi(\xi_{j})|\,(\alpha_{\mu}(C_{2}B_{0})+\ve^{\frac{1}{4}})\\
& \leq \sum_{j\in J_{2}}  \phi(\xi_{j}) \int  \theta_{j} d\mu +C\sum_{j\in J_2} \ve^{\frac{1}{4}} s_{j}^{d+1}.
\end{align*}
Thus,
\[
I_{2} \leq I_{21}+I_{22} = \sum_{j\in J(x,r)} \ps{ \phi(\xi_{j})\int \theta_{j}d\mu + \ve^{\frac{1}{4}} s_{j}^{d+1}}.\]
Hence,
\begin{align*}
\int \phi\, d\nu 
& =I_{1}+I_{2} \\
&\leq \sum_{j\in J(x,r)}\int (\phi-\phi(\xi_{j})) \theta_{j}\,d\mu + C \sum_{j\in J(x,r)}\ve^{\tfrac12} s_{j}^{d+1} + \sum_{j\in J(x,r)} \phi(\xi_{j}) \int  \theta_{j}d\mu\\
&=\int \phi \,\theta\, d\mu + C\sum_{j\in J(x,r)} \ve^{\frac{1}{4}} s_{j}^{d+1} .
\end{align*}
We can similarly show a converse inequality, and this proves the lemma. 
\end{proof}

An immediate consequence of the previous two lemmas is the following.

\begin{lemma}
	For all $x\in \Sigma\cap B(0,20)$ with $2d(x)<r\leq15$, 
	\begin{equation}
	\label{e:dnumu}
	F_{B(x,r)}(\nu,\mu)\lec_{A,\tau}  \ve^{\frac{1}{4}} r^{d+1} + \ve^{\frac14} \sum_{2B_{j}\cap B(x,r)\neq\varnothing}  s_{j}^{d+1}.
	\end{equation}
\end{lemma}

\begin{lemma}
	If $1<r$, $x\in \Sigma$, and $B(x,r) \cap B(0,10)\neq\varnothing$, then
	\begin{equation}\label{e:a<r^-d}
	\alpha_{\nu}(x,r)\lec_{A,\tau} \ve^{\frac{1}{4}} r^{-d}.	
	 	\end{equation}
%		{\color{blue}{ Am I correct that there is no } $r^{-d}$ { here?}. Compare with previous version} \jonas{No, we do need the $r^{-d}$ to say that the alpha numbers decay as the radius goes to infinity so that we can integrate over all radii later, see the estimate for $I_2$ on page 35. The proof is not changed, though.}
\end{lemma}
 
\begin{proof}
	Let $\psi$ be a $1$-Lipschitz function that is zero on $B(0,11)^{c}$ and $1$ on $B(0,10)$. Set
	\[
	\wt c=\frac{\int \psi\, d\nu}{\int \psi\, d\cH^{d}|_{P_{0}}}.
	\]
	{%\color{red}
	Note that the collection $\{B_j\}$ has finite overlap depending only on $n$. Moreover
	if $2B_{j}\cap B(x,r)\neq\varnothing$, with $B_j=B(\xi_j,s_j)$, $\xi_j\in\Sigma$, $s_j=\eta d(\xi_j)\le \eta (d(x) +2s_j+r)\le \eta(3r/2+2s_j)$ with $\eta=10^{-3}$ which yields $s_j\le 2\eta r$ and therefore
	using \rf{e:nu-AD} we have
	\begin{equation}\label{sum-bound-1}
	 \sum_{2B_{j}\cap B(x,r)\neq\varnothing}  s_{j}^{d+1}\lec r\sum_{2B_{j}\cap B(x,r)\neq\varnothing}\sigma(B_j) 
	 \lec r\sigma(B(x, 2r))\lec r^{d+1}.
		\end{equation}
	Then using the fact that $C_2\le C_1$ (see line above \rf{coice-x'}), \rf{pointwise-a-bound-2} and \rf{e:dnumu}, we have
	}
	\begin{align*}
	\wt c\int \psi\, d\cH^{d}|_{P_{0}}  & = \int \psi\, d\nu
	 \stackrel{\eqref{e:dnumu}}{\leq}
	 \int \psi\, d\mu+ C\ve^{\frac{1}{4}}
	\\ & \leq \int \psi\, d\LL_{C_{2}B_{0}} +C\alpha(0,C_{2}) + C\ve^{\frac{1}{4}} 
	\leq c_{C_{2}B_{0}} \int \psi \,d\cH^{d}|_{P_{0}} + C\ve^{\frac{1}{4}}.
	\end{align*}
	Since $\int \psi\, d\cH^{d}|_{P_{0}} \approx 1$, this gives $\wt c\leq c_{C_{2}B_{0}}+C\ve^{\frac{1}{4}}$ for $\ve>0$ small enough, where $C$ depends on $A$ and $\tau$. An opposite inequality can be proved by 
	a similar argument. Thus, $|\wt c-c_{C_{2}B_{0}}|\lec_{A,\tau}\! \ve^{\frac{1}{4}}$. Hence, for $\phi\in \Lip_{1}(B(x,r))$, and since $\nu = \LL_{C_{2}B_{0}}$ in $B(0,10)^c$ and $1<r$ by \eqref{e:nuaway},
	\begin{align*}
	\left|\int \phi \,(d\nu- d\LL_{C_{2}B_{0}})\right|
	& = \left|\int [  (\phi -\phi(0))\psi + \phi(0)\psi + \phi (1-\psi)]\, (d\nu- d\LL_{C_{2}B_{0}})\right|\\
	& \lec \alpha_{\nu}(0,C_{2})+ |\phi(0)| \left|\int \psi \, (d\nu- d\LL_{C_{2}B_{0}})\right| + 0\\
	& \lec \alpha_{\nu}(0,C_{2})+ r \ps{ \ve^{\frac{1}{4}}  +  \left|\int \psi \, (d\nu- \wt c \,d \cH^{d}|_{P_{0}})}\right|
	 \lec  r\ve^{\frac{1}{4}}.
	\end{align*}
Thus, \eqref{e:a<r^-d} follows by this and \eqref{e:nu-AD}.
\end{proof}

% ******************************************************************************************************
\vv
 
\section{$\Lambda$-estimates}

For the rest of the paper we denote by $\phi:\R^n\to\R$ a radial $C^\infty$ function such that $\chi_{B(0,1/2)}\leq\phi\leq 
\chi_{B(0,1)}$.
We also set
\[
\phi_{r}(x) = r^{-d} \phi(r^{-1}x)
\]
and 
\[
\psi_{r}(x) = \phi_{r}(x) -\phi_{2r}(x).\]

Let $\pi$ be the orthogonal projection onto $P_{0}$ and let $\pi[\nu]$ denote the image measure
of $\nu$ by $\pi$, that is, the measure such that
\begin{equation}\label{pi-nu}
\pi[\nu](G)=\nu(\pi^{-1}(G))
\end{equation}
for any Borel subset $G\subset P_0$. 

%{\color{blue} Since I am still a bit confused about what the exact definition of $\pi[\nu]$ really is I would like to make sure that we include it here. 
%Is 
%\begin{equation}\label{pi-nu}
%\pi[\nu](G)=\nu(\pi^{-1}(G))
%\end{equation}
%for any Borel subset $G\subset P_0$?}
%\jonas{Yes, I now added your definition}

{%\color{red} 
Note that \rf{nu-sigma-comparison} ensures that $\nu$ and $\sigma$ are comparable measures on $\Sigma$. Since $\Sigma$ is a Lipschitz graph over $P_0$ (see Lemma \ref{lem6.8}) then
$\HH^{d}|_{P_{0}}$ and $\pi[\nu]$ are mutually absolutely continuous, in fact they are comparable.}

The goal of this section is to prove the following.

\begin{lemma}\label{lem8.1}
	Let $f=\frac{d\pi[\nu]}{d\HH^{d}|_{P_{0}}}$. Then
\begin{equation}
\label{e:nu-int}
\|f-c_{C_{2}B_{0}}\|_{L^{2}(\HH^d|_{P_0})}^{2} \approx \int_{P_0}  \int_{0}^{\infty} |\psi_{r}*\pi[\nu](z)|^{2} \,\frac{dr}{r}\, d\HH^d(z) \lec_{A,\tau}  \ve^{\frac{1}{4}}. 
\end{equation}
\end{lemma}

The first comparison above is a classical result from harmonic analysis (see \cite[Section I.6.3]{Big-Stein}), so we just will focus on the second inequality. 

For a measure $\lambda$ and $x\in \R^{n}$, we define
\[\wt{\psi}_{r}(x) = \psi_{r}\circ \pi(x) \cdot \phi\bigl((5r)^{-1}x\bigr)\]
and
\[
\Lambda_{\lambda}(x,r)
=\av{\int \wt{\psi}_{r}(y-x) \,d\lambda(y)}.\]

Since $\Sigma$ is a $C\ve^{\frac{1}{4}}$-Lipschitz graph over $P_{0}$, we claim that for $\ve>0$ small enough, then 
%\[
%\wt{\psi}_{r}(y-x) = \psi_{r}\circ \pi(y-x)  \mbox{ when }\dist(y,\Sigma)<r \mbox{ and }x\in \Sigma.\]
\begin{equation}\label{e:lambda-psi}
\Lambda_{\nu}(x,r) = |\psi_{r}* \pi[\nu](\pi(x))| \quad\mbox{ for all $x\in \Sigma$}.
\end{equation}
Indeed, it suffices to show that
\begin{equation}\label{eqpsi1}
\wt{\psi}_{r}(y-x) =\psi_{r}(\pi(y-x))
\quad\mbox{ for all $x,y\in\Sigma$.}
\end{equation}
To this end, by the definition of $\wt{\psi}_{r}$ it suffices to check that $\phi\bigl((5r)^{-1}(y-x)\bigr)
= 1$ whenever $\psi_{r}(\pi(y-x))\neq0$. Note that the latter condition implies that
 $\phi_{2r}(\pi(x-y))\neq0$ and so $|\pi(x)-\pi(y)|\leq 2r$. In fact if $\phi_{2r}(\pi(x-y))=0$ then $\phi_{r}(\pi(x-y))=0$ and $\psi_{r}(\pi(y-x))=0$. Thus, since $\Sigma$ is a $C\ve^{\frac{1}{4}}$-Lipschitz graph, 
\[
|x-y|\leq 2(1+C\ve^{\frac{1}{4}})r<\frac{5}{2}\,r.\] 
Thus, $y\in B(x,5r/2)$, which implies that $
\phi\bigl((5r)^{-1}(y-x)\bigr)=1$, as wished.

Consequently, since $\pi$ is bi-Lipschitz between $\Sigma$ and $P_{0}$, we have
\begin{align*}
\int_{P_0}  \int_{0}^{\infty} |\psi_{r}*\pi[\nu](z)|^{2} \,\frac{dr}{r}\, d\HH^d(z) &\approx
\int_{P_0}  \int_{0}^{\infty} |\psi_{r}*\pi[\nu](z)|^{2} \,\frac{dr}{r}\, d\pi[\sigma](z)\\
&=\int_{\Sigma}  \int_{0}^{\infty} \Lambda_{\nu}(x,r)^{2} \,\frac{dr}{r}\, d\sigma(z).
\end{align*}
Therefore, to complete the proof of Lemma \ref{lem8.1}, it suffices to show that
\begin{equation}
\label{e:lambda-int}
\int_{\Sigma}  \int_{0}^{\infty} \Lambda_{\nu}(x,r)^{2}\, \frac{dr}{r} \,d\sigma(x) \lec_{A,\tau} \ve^{\frac{1}{4}}.
\end{equation}
The rest of this section is devoted to proving this estimate. 

First we need the following auxiliary result.

\begin{lemma} 
	\label{l:tt-lemma}
	For a finite Borel measure $\lambda$, denote 
\[
	T\lambda(x) = \ps{\int_{0}^{\infty} \bigl|\wt\psi_{r}*\lambda(x)\bigr|^{2}
	 \,\frac{dr}{r}}^{\frac{1}{2}},
	\]
and for $f\in L^2(\sigma)$,
	set $T_\sigma f=T(f\sigma)$. Then $T_\sigma$ is bounded in $L^p(\sigma)$ for $1<p<\infty$  and $T$ is bounded from 
$M(\R^{n})$ to $L^{1,\infty}(\sigma)$. Further, the norms $\|T_\sigma\|_{L^p(\sigma)\to L^p(\sigma)}$ and
 $\|T\|_{M(\R^n)\to L^{1,\infty}(\sigma)}$ are bounded above by some absolute constants depending only on $p$, $n$, and $d$.
\end{lemma}

The proof of this lemma is quite standard in Calder\'on-Zygmund theory. First one shows that 
$T_\sigma$ is bounded in $L^2(\sigma)$, taking in to account
\rf{eqpsi1}. By a suitable Calder\'on-Zygmund decomposition, one can derive then the boundedness of
$T$ from $M(\R^n)$ to $L^{1,\infty}(\sigma)$, which implies the boundedness of $T_\sigma$ in $L^p(\sigma)$
for $1<p<2$ by interpolation. The boundedness in $L^p(\sigma)$
for $2<p<\infty$ can be deduced by interpolation from its boundedness from $L^\infty(\sigma)$ to $BMO(\sigma)$.
See \cite[Theorem 5.1]{TT15} and  \cite[Proposition 13.7]{Tol17} for quite similar  (but somewhat more difficult) results. We skip the
details.

\begin{lemma}
We have
	\begin{equation}\label{e:nu-int1}
	\int_{\Sigma \cap 20 B_{0}}\int_{\eta^2 d(x) }^{1}  \Lambda_{\nu}(x,r)^{2}\,\frac{dr}{r}\, d\sigma(x) \lec \ve^{\frac{1}{4}}.
	\end{equation}
\end{lemma}

\begin{proof}
For each $x\in\Sigma \cap 20 B_{0}$, we split
\begin{align}\label{eqoq1}
 &\int_{\eta^2 d(x)}^{1}   \Lambda_{\nu}(x,r)^{2}\,\frac{dr}{r} \\
&  \lesssim \int_{\eta^2 d(x)}^{1} \!(\Lambda_{\nu}(x,r)-  \Lambda_{\theta \mu}(x,r))^{2}\,\frac{dr}{r} +\int_{\eta^2 d(x)}^{1} \!\Lambda_{(1-\theta) \mu}(x,r)^{2} \, \frac{dr}{r} 
+\int_{\eta^2 d(x)}^{1} \!\Lambda_{ \mu}(x,r)^{2}\,\frac{dr}{r},\nonumber
\end{align}
and denote\[
H=\ck{x\in \Sigma \cap 20 B_0 : \int_{\eta^2 d(x)}^{1}   \Lambda_{(1-\theta)\mu}(x,r)^{2}\,\frac{dr}{r} >\ve^{\frac14}}.
\]
Now write
\begin{eqnarray}\label{eqoq2}
\int_{\Sigma \cap 20 B_{0}}\int_{\eta^2 d(x) }^{1}  \Lambda_{\nu}(x,r)^{2}\,\frac{dr}{r}\, d\sigma(x) &=&
\int_{H}\int_{\eta^2 d(x) }^{1}  \Lambda_{\nu}(x,r)^{2}\,\frac{dr}{r}\, d\sigma(x) \nonumber\\ &&\quad +
\int_{\Sigma \cap 20 B_{0}\setminus H}\int_{\eta^2 d(x) }^{1}   \Lambda_{\nu}(x,r)^{2}\,\frac{dr}{r}\, d\sigma(x)
	\end{eqnarray}
To estimate the first integral on the right hand side, 
note that if $x\in 20 B_{0}$ and $r<1$, then $B(x,r)\subset 40B_{0}$. So
 applying Lemma \ref{l:tt-lemma} with $\lambda =(\mu- \theta \mu)|_{40B_{0}}$, using the fact that $\theta\equiv 1$ on $O\supset E_{0}\cap 40 B_{0}$ by \eqref{e:E040}, by the definition of $E_{0}$, and \rf{working-hyp-3A} (provided $C_1>40$) we get 
\begin{align*}
\sigma(H) &\leq \sigma\left(\left\{x\in\Sigma: T( \chi_{40B_0}(\mu-\theta \mu))>\ve^{\frac18}\right\}\right)\\
&
\lec \ve^{-\frac{1}{8}}\,\| \mu-\theta \mu\|_{40 B_{0}} \leq\ve^{-\frac{1}{8}} \mu(40B_{0}\setminus E_{0})\leq 
\ve^{\frac{7}{8}} \ve \,\mu(40B_{0})
\lesssim \ve^{\frac{1}{2}} . 
\end{align*}
Consider the function 
$q = \frac{d\nu|_{C_{1}B}}{d\sigma}$. Taking into account that $\|q\|_{L^4(\sigma)}\lesssim_A1$
and using the $L^4(\sigma)$ boundedness of $T_\sigma$ (and recalling $T_{\sigma}(q)=T(\nu)$), we obtain
\begin{equation}\label{eqoq*h}
\int_{H}\int_{\eta^2 d(x)}^{1} \Lambda_{\nu}(x,r)^{2} \,\frac{dr}{r}\, d\sigma (x)
\leq \int_H |T_\sigma q(x)|^2\,d\sigma(x) \lesssim \sigma(H)^{\frac12}\,\|T_\sigma q\|_{L^4(\sigma)}^2
\lesssim_A \ve^{\frac{1}{4}}.
\end{equation}

Next we consider the second integral on the right hand side of \rf{eqoq2}. By the
definition of $H$, we have
$$\int_{\Sigma \cap 20 B_{0}\setminus H}
\int_{\eta^2 d(x)}^{1} \!\Lambda_{(1-\theta) \mu}(x,r)^{2} \, \frac{dr}{r}\lesssim \ve^{\frac14},$$
and so, by  \rf{eqoq1}, 
\begin{align}\label{eqoq3}
 \int_{\Sigma \cap 20 B_{0}\setminus H}\int_{\eta^2 d(x) }^{1}   \Lambda_{\nu}(x,r)^{2}\frac{dr}{r}\,d\sigma(x) 
&  \lesssim  \int_{\Sigma \cap 20 B_0} \int_{\eta^2 d(x)}^{1} \!(\Lambda_{\nu}(x,r)-  \Lambda_{\theta \mu}(x,r))^{2}\frac{dr}{r}\,d\sigma(x) \\
&\quad
+ \int_{\Sigma \cap 20 B_0}\int_{\eta^2 d(x)}^{1} \!\Lambda_{ \mu}(x,r)^{2}\frac{dr}{r}\,d\sigma(x) + \ve^{\frac14}\nonumber\\
& =: I_1 + I_2 + \ve^{\frac14}.\nonumber
\end{align}

We will now bound $I_2$.  Given $x\in\Sigma \cap 20 B_{0}
\setminus Z$, by Lemma \ref{lem6.4} we have
$d(x)\approx r_{k(x)}$ and by the definition of $k(x)$, $x\in V_{k(x)-1}^{40}$.
Thus there exists some ball $B_{j,k(x)-1}$ such that $x\in 40B_{j,k(x)-1}$.
So for all $r\in(\eta^2 d(x),1)$ there exists some ball $B_{j,k}$ such that
$B(x,5r)\subset 3B_{j,k}$ and  $r\approx r_{B_{j,k}}$. 
Then,
taking into account that $|\nabla \wt\psi_r|\lesssim r^{-d-1}$ and \rf{e:muLjk1},
\begin{align}
\label{e:L<a}
\Lambda_{\mu}(x,r) &
\leq  \av{\int \wt{\psi}_{r}(y-x) \,(d\mu(y)-d \LL_{j,k})(y)}
 + \av{\int \wt{\psi}_{r}(y-x) \, d\LL_{j,k}(y)}\\
&\lesssim \frac{r_{2C_2B_{j,k}}\,\mu(2C_2B_{j,k})}{r^{d+1}}\,\alpha_{\mu}(2C_2B_{j,k}) + \av{\int \wt{\psi}_{r}(y-x) \, d\LL_{j,k}(y)}.\nonumber
\end{align}
The first term on the right hand side satisfies
$$\frac{r_{2C_2B_{j,k}}\,\mu(2C_2B_{j,k})}{r^{d+1}}\,\alpha_{\mu}(C_2B_{j,k})\lesssim_A \alpha_{\mu}(C_2B_{j,k})\lesssim_A \alpha_{\mu}(x_{j,k},C_3r),$$
for a suitable constant $C_3>1$, so that in particular $C_3\,r> d(x_{j,k})$.

Now we claim that the last integral on the right hand side of \rf{e:L<a} vanishes. To prove this, first we will check that $\wt{\psi}_{r}(y-x) =\psi_r\circ\pi(y-x)$ for $x\in\Sigma$ and $y\in P_{j,k}$ with
$y-x\in\supp{\psi}_{r}\circ\pi$. Indeed, the latter condition implies that
$$|\pi(y)-\pi(x)|\leq 2r.$$
Also using that $\Sigma$ is a $C\ve^{\frac14}$-Lipschitz graph, that $\angle(P_{j,k},P_{0})<\ve^{\frac{1}{4}}$, that
$x\in 3B_{j,k}$, \rf{e:49r}, and \rf{e:ytosigma}, it easily follows that
$$|\pi^\perp(y)-\pi^\perp(x)|\lesssim\ve^{\frac14}r,$$
assuming $\ve$ small enough.
Then we infer that
$$|x-y|\leq 2r + C\ve^{\frac14}r\leq \frac52\,r.$$
Thus $\phi((5r)^{-1}(x-y))=1$ and so
$$\wt{\psi}_{r}(y-x) =\phi((5r)^{-1}(x-y))\,\psi_r\circ\pi(y-x) = \psi_r\circ\pi(y-x).$$
Now we derive
\begin{equation}\label{8.10A}
\int \wt{\psi}_{r}(y-x) \, d\LL_{j,k}(y) =
\int \psi_{r}(\pi(y)-\pi(x)) \, d\LL_{j,k}(y) = \int \psi_{r}(y'-\pi(x)) \, d\pi[\LL_{j,k}](y')=0,
\end{equation}
taking into account the definition of $\psi_r$ and that $\pi[\LL_{j,k}]$ coincides with $d$-dimensional
Lebesgue measure on $P_0$ modulo a constant factor. 

Consequently, from \rf{e:L<a} we deduce that
$$\Lambda_{\mu}(x,r)\lesssim_A \alpha_{\mu}(x_{j,k},C_3r).$$
Therefore, for $C_1$ big enough,
\[
 \int_{\eta^2 d(x)}^{1} \Lambda_{ \mu}(x,r)^{2}\,\frac{dr}{r} \lesssim \int_{0}^{1} \alpha_{\mu}(x_{j,k},C_3r)^{2}\, \frac{dr}{r}  
\lec \int_{0}^{C_{1}}\alpha_{\mu}(x_{j,k},t)^{2} \,\frac{dt}{t}  
< \ve,\]
and thus using that fact that $\mu(B_0)=1$, \rf{working-hyp-1A} and taking $C_1>40$ we have
$$I_2 = \int_{\Sigma \cap 20 B_0}\int_{\eta^2 d(x)}^{1} \!\Lambda_{ \mu}(x,r)^{2}\frac{dr}{r}\,d\sigma(x)  \lesssim \ve.$$

Finally we handle the integral $I_1$ in \rf{eqoq3}.
 Recall that 
\begin{align*}
|\Lambda_{\nu}(x,r)-\Lambda_{\theta\mu}(x,r)|^{2}
&\lec \ps{r^{-d-1} F_{B(x,5r)}(\nu,\theta\mu)}^{2}
& \!\!\stackrel{\eqref{e:nuthetamu}}{ \lec}  \Biggl(\,\sum_{2B_{j}\cap B(x,5r)\neq\varnothing} \ve^{\frac14}\, \frac{s_{j}^{d+1}}{r^{d+1}}\Biggr)^{2}.
\end{align*}
Observe that if $B(x,5r)\cap 2B_{j}\neq\varnothing$, for $r\in (\eta^2d(x),1)$ (recall $\eta=10^{-3}$) then
%{\color{blue} There are some changes of constants here as it is not true that $\eta d(x)<r$ we are in the regime $\eta^2 d(x)<r$ some large numbers appear. My impression is that the cutoff could have been made at
%$\eta d(x)$ rather than $\eta^2 d(x)$ but since so much of the section is written in terms of $\eta^2 d(x)$ I decided to not make any modifications}
%\jonas{This was my bad, I think I may have had $\eta^{2}$ just to ensure that, for $r<\eta^{2}d(x)$, that $B(x,r)$ was much smaller than the $B_j$ (so that $\nu$ is very smooth there), but this could have been an overly conservative precaution. I say leave as is.}
\[
s_{j}=\eta\, d(\xi_{j}) \leq \eta(d(x)+2s_{j} + 5r) \leq 10^3r + 2\eta s_{j}+ r=1001r+2\eta s_{j}\]
and so $s_{j}\leq 1100r$ since $\eta<1/4$. Then, it follows easily that $B_j\subset C_4B_0$, for a large enough
constant $C_4>1$.
Hence, by the Cauchy-Schwarz inequality, plus an argument along the lines of the one used to show \rf{sum-bound-1}, we have
\begin{align*}
\biggl(\,\sum_{2B_{j}\cap B(x,5r)\neq\varnothing}  \frac{s_{j}^{d+1}}{r^{d+1}}\biggr)^{2}
& \leq 
\Biggl(\,\sum_{2B_{j}\cap B(x,5r)\neq\varnothing} \frac{s_{j}^{d+2}}{r^{d+2}}\Biggr)
\Biggl(\,\sum_{2B_{j}\cap B(x,5r)\neq\varnothing} \frac{s_{j}^{d}}{r^{d}}\Biggr)\\
& \lesssim \sum_{2B_{j}\cap B(x,5r)\neq\varnothing} \frac{s_{j}^{d+2}}{r^{d+2}}.
\end{align*}
This and the fact that $s_{j}\leq 1100r$ imply
\begin{align*}
I_1 &= \int_{\Sigma \cap 20 B_0} \int_{\eta^2 d(x)}^{1} \!(\Lambda_{\nu}(x,r)-  \Lambda_{\theta \mu}(x,r))^{2}\,\frac{dr}{r}\,d\sigma(x)\\
& \lec \ve^{\frac12} \int_{\Sigma\cap 20 B_{0}}\int_{\eta^2 d(x)}^{1} \sum_{2B_{j}\cap B(x,5r)\neq\varnothing}  s_{j}^{d+2} \frac{dr}{r^{d+3}}\, d\sigma(x) \\
& = \ve^{\frac12} \;\sum_{j} s_{j}^{d+2}  \int_{\Sigma\cap 20 B_{0}}\int_{\eta^2 d(x)}^{1}\chi_{B(x,5r)\cap 2B_{j}\neq\varnothing}(x)\, \frac{dr}{r^{d+3}}\, d\sigma(x) \\
& \lec \ve^{\frac12} \sum_{j: B_{j}\subset C_4B_0} s_{j}^{d+2}  \int_{\Sigma\cap 20 B_{0} }\int_{s_{j}/1100}^{1}\chi_{B(x,5r)\cap 2B_{j}\neq\varnothing}(x)\,\frac{dr}{r^{d+3}} \,d\sigma(x).
\end{align*}
Observe now that if $ B(x,5r)\cap 2B_{j}\neq\varnothing$, then 
$$x\in B(\xi_j,5r+2s_j)\subset B(\xi_j,2250r),$$ 
recalling that $s_j\leq 1100r$. Therefore, 
\begin{align*}
I_1
 & \lec \ve^{\frac12} \sum_{j: B_{j}\subset C_4B_0} s_{j}^{d+2}  \int_{s_{j}/4}^{1}\int_{B(\xi_{j},13r)}  d\sigma(x)\,\frac{dr}{r^{d+3}} \\
  & \lec_A \ve^{\frac12} \sum_{j: B_{j}\subset C_4B_0 } s_{j}^{d+2}  \int_{s_{j}/4}^{1} \frac{dr}{r^{3}}  
  % \lec \ve^{2} \sum_{j: 2B_{j}\cap 20B_{0}\neq\varnothing } s_{j}^{d+2}  s_{j}^{-2}   
  \\
   &    \lec \ve^{\frac12} \sum_{j: B_{j}\subset C_4B_0 } s_{j}^{d}  \lesssim\ve^{\frac12}\sigma(C_4B_{0}) 
     \lec \ve^{\frac12}
\end{align*}
Gathering  \rf{eqoq1}, \rf{eqoq*h}, and
the estimates obtained for $I_1$ and $I_2$, the lemma follows.
\end{proof}

\begin{lemma}
For all $j$,
\begin{equation}
\label{e:small-scales}
\int_{B_{j}}\int_{0}^{\eta^2 d(x)}  |\Lambda_{\nu}(x,r)|^{2}\frac{dr}{r}d\sigma(x) \lec_{A,\tau} \ve^{1/4} s_{j}^{d}.
\end{equation}
\end{lemma}

\begin{proof}
	Let $g= \sum_{i} c_{i}\theta_{i}(x) $. Note that for $x\in \tfrac32B_{j}$, 
	\[
	|g(x) - c_{j}| =\av{\sum_{i} (c_{i}-c_{j})\theta_{i} }
	\stackrel{\eqref{e:ci-cj}}{\lec} \ve .
	\]
	For $x\in B_{j}$ we have
	\[
	\eta^2 d(x) \leq \eta^2(d(\xi_{j}) + |x-\xi_{j}|) < 2 \eta s_{j}.\]
	Also, using the $d$-AD-regularity of $\nu$ and that $\Lambda_{\nu}(x,r)=|\psi_{r}* \pi[\nu](\pi(x))| $ for
all $x\in\Sigma$ (see \rf{e:lambda-psi}), it is easy to check, by an argument similar to the one used in \rf{e:L<a}, that 
\begin{equation}\label{eqkgeq245}
\Lambda_{\nu}(x,r)\lesssim_{A,\tau}\alpha_{\nu}(x,5r).
\end{equation}	
	Thus, by Lemma \eqref{l:Xavi-lemma}, and recalling Lemma \ref{lem8.1}we get
\[
\int_{B_{j}}\int_{0}^{\eta^2 d(x)} |\Lambda_{\nu}(x,r)|^{2} \,\frac{dr}{r} \,d\sigma(x) 
\lec_{A,\tau} \int_{B_{j}} \int_{0}^{2\eta s_{j}}\alpha_{\nu}(x,5r)^{2} \,\frac{dr}{r} \,d\sigma(x)
\lec \ve^{\frac14} s_{j}^{d}.
\]
\end{proof}

We now complete the proof of \eqref{e:lambda-int}. To this end, we split the domain of integration by setting
\[
A_{1} = \{(x,r): B(x,r)\cap B(0,10)=\varnothing\},
\]	
\[
A_{2}=\{(x,r): B(x,r)\cap B(0,10)\neq\varnothing, r>1\},
\]
\[
A_{3} = \{(x,r): B(x,r)\cap B(0,10)\neq\varnothing,\eta^2 d(x)<r\leq 1\},
\]
\[
A_{4} = \{(x,r): B(x,r)\cap B(0,10)\neq\varnothing, r\leq \eta^2 d(x) \},
\]
and then we write 
\[
I_{i}=\iint_{A_{i}} \Lambda_{\nu}(x,r)^{2}\, \frac{dr}{r} \,d\sigma(x) .
\]

Note that for $(x,r)\in A_{1}$,
\[
\Lambda_{\nu}(x,r) = \Lambda_{\LL_{C_{2} B_{0}}} (x,r)=0.
\]
and so
$
I_{1}=0.
$

For $(x,r)\in A_{2}$, since $B(x,r)\cap B(0,10)\neq\varnothing$ and $r>1$, 
\[
|x| \leq r +10<11r.
\]
and so 
\[ r\geq \max\{\tfrac1{11}|x|,1\}.\]
Using again \rf{eqkgeq245}, which still holds in this case, plus the fact that $\sigma$ is $d$-AD regular 
 we get
\begin{align*}
I_{2} 
& \leq \int_{\Sigma} \int_{\max\{\frac1{11}|x|,1\}}^{\infty} \Lambda_{\nu}(x,r)^{2} \,\frac{dr}{r}\,d\sigma(x) 
\lec \int_{\Sigma}\int_{\max\{\frac1{11}|x|,1\}}^{\infty} \alpha_{\nu}(x,5r)^{2} \,\frac{dr}{r}\,d\sigma(x) \\
&\!\!\!  \stackrel{\eqref{e:a<r^-d}}{\lec}  \ve^{\frac12} \int_{\Sigma}\int_{\max\{\frac1{11}|x|,1\}}^{\infty}\, \frac{dr}{r^{2d+1}}\,d\sigma(x)
\lec  \ve^{\frac12} \int_{\Sigma} \frac{1}{\max\{\frac1{11}|x|,1\}^{2d}}\,d\sigma(x)\\
&\!\!\! \lec\ve^{\frac12} \left(\int_{\Sigma\cap B_0} \,d\sigma(x)+ \sum_{k=1}^\infty2^{-2dk}\int_{\Sigma\cap (B_{2^k}\backslash B_{2^{k-1}})}\,d\sigma(x)\right)
\lec \ve^{\frac12}. 
\end{align*}
 
If $B(x,r)\cap B(0,10)\neq\varnothing$ and $r<1$, then $x\in B(0,20)$, and so
\[
I_{3} 
\leq \int_{B(0,20)}\int_{\eta^2 d(x)}^{1}  \Lambda_{\nu}(x,r)^{2} \,\frac{dr}{r} \,d\sigma(x) 
\stackrel{\eqref{e:nu-int1}}{\lec} \ve^{\frac{1}{4}}. 
\]
Thus, all that is left is $I_{4}$. Note that if $r<\eta^2 d(x)$ and $x\in B_{j}$, then $B(x,r)\subset 2B_{j}$, hence
\begin{align*}
I_{4} 
& =  \int_{B(0,20)}\int_{0}^{\eta^2 d(x)} \Lambda_{\nu}(x,r)^{2} \,\frac{dr}{r}\, d\sigma(x) 
 =\sum_{B_{j}\cap B(0,20)\neq\varnothing} \int_{B_{j}}\int_{0}^{\eta^2 d(x)} \Lambda_{\nu}(x,r)^{2} \,\frac{dr}{r}\, d\sigma(x)\\
 & \!\!\!\stackrel{\eqref{e:small-scales}}{\lec} \sum_{B_{j}\cap B(0,20)\neq\varnothing}  \ve^{\frac14} s_{j}^{d}
 \lec \ve^{\frac14}\sum_{B_{j}\cap B(0,20)\neq\varnothing}\sigma(B_{j})
 \lec \ve^{\frac14} \sigma(B(0,40)) \lec \ve^{\frac14}. 
\end{align*}
Combining the estimates above finishes the proof of Lemma \ref{lem8.1}.

\vv
% *******************************************************************************************

\section{The end of the proof}\label{end}

Recall that by our choice to $\tau$ and $A$, Lemma \ref{l:good-range}, $\delta(x)\leq 10^{-3}$ for all points $x\in E_{0}\cap B_{0}$.
If moreover $\delta(x)>0$, let 
\begin{equation}\label{def-Bx}
B_{x}= B(x,\delta(x))
\end{equation}
Recall that for $x\in E_{0}\cap B_{0}\setminus Z$ with $Z$ as in Lemma \ref{l:good} $d(x)>0$ and therefore $\delta(x)>0$ since $E_0\cap B_0$ is a closed set. Hence $x\in E_{0}\cap B_{0}\setminus Z$
satisfies one of the following conditions:

\begin{itemize}
	\item[\ND:] $\mu(B_{x}\setminus E_{0})\geq\ve^{\frac12} \mu(B_{x})$,
	\item[\LD:] %$x\not\in \ND$ and 
	$\Theta_{\mu}^{d}(B_{x})\leq\tau$,
	\item[\HD:]  %$x\not\in \ND$ and %
	$\Theta_{\mu}^{d}(B_{x})\ge A$, or
	\item[\BA:] %$x\not\in \ND\cup\LD\cup\HD$ and %
	$\angle(L_{x,r},P_{B_{0}})\geq\ve^{\frac{1}{4}}$. 
\end{itemize}

\noindent Recall the abbreviations stand for ``low density", ``high density", and ``big angle", respectively. 

What we show now is that each of these sets has measure much smaller than $\mu(E_{0}\cap B_{0})$, and thus there must be a subset $G'\subset E_{0}\cap B_{0}$ of positive measure for which $\Theta^d_{\mu}(x,r) \in [\tau,A]$ for all $r>0$ small. 
This contradicts the conclusion of Lemma \ref{l:good}, obtained under the assumption that 
there is no set $E'\subset E$ with $\mu(E')>0$ such that $\mu|_{E'}$ is $d$-rectifiable. 
This will conclude the proof of
Lemma \ref{l:main} and hence of Theorem \ref{thmi}.

\begin{lemma}\label{lemld}
	$\mu(\ND)\lec \ve^{\frac{1}{2}}$. 
\end{lemma}

\begin{proof}
Let $\{\mathcal{B}_j\}_{j=1}^{N}$ be a Besicovitch subcovering of the collection $\{B(x,\delta(x)): x\in \ND\}$. Since $\delta(x)\leq C_1$ by definition,
 we have $B\subset 3C_1B_0$ for each $B\in \mathcal{B}_j$ and every $j$, and so by our definition of $E_{0}$,
\begin{align*}
	\mu(\ND) &
	\leq \sum_{j=1}^N \sum_{B\in\mathcal{B}_j}\mu(B)
	\leq\ve^{-\frac{1}{2}} \sum_{j}\mu( B(j)\setminus E_{0})\\&
	\leq \ve^{-\frac{1}{2}} \mu(3C_1B_{0}\setminus E_{0})
	\leq \ve^{-\frac{1}{2}} \ve\, \mu(3C_1B_{0}) = \ve^{\frac{1}{2}}. 
\end{align*}
\end{proof}

\begin{lemma}\label{l:LD}
	$\mu(\LD)\lec_{M,C_1} \tau$. 
\end{lemma}

\begin{proof}
For any $x\in \LD$ we have $\delta(x)\geq d(x)>0$.
 Then there exists some $k$ such that $r_{k-1}\leq \delta(x) < r_{k}$ and so $x\in V_k^2$.
 For $i\in J_k $ such that $x\in 2 B_{i,k}$ we have
$$\dist(x,\Sigma)\leq |x-x_{i,k}| + \dist(x_{i,k},\Sigma)\lesssim r_k +\ve\, r_k\lesssim \delta(x).$$
Let $z(x)\in \Sigma$ be such that 
\begin{equation}\label{eqzx1}
|x-z(x)|=\dist(x,\Sigma)\leq C_5\,\delta(x),
\end{equation}
for a suitable constant $C_5>1$.
Then we have
\begin{equation}\label{eqzx2}
B(x,\delta(x))\subset B(z(x),(C_5+1)\delta(x))\subset B(z(x),2C_5\delta(x))\subset B(x,4C_5\delta(x)),
\end{equation}
so that
\begin{equation}\label{eqzx3}
\mu(B(z(x),2C_5\delta(x)))\leq \mu(B(x,4C_5\delta(x)))\lesssim_M \mu(B(x,\delta(x)))\lesssim_{M}\tau \,\delta(x)^d.
\end{equation}

Let $\{\mathcal{B}_j\}_{j=1}^{N}$ be a Besicovitch subcovering from the collection $\{B(x,4C_5\delta(x)):x\in \LD\}$, and let $x_{B}$ be
the center of each $B\in \cup_{j=1}^N\mathcal{B}_j$. 
We deduce that
\begin{align*}
\mu(\LD)&
\leq \sum_{j=1}^N\sum_{B\in\mathcal{B}_j}\mu(B) \lesssim_{M}\tau \sum_{j=1}^N\sum_{B\in\mathcal{B}_j}\delta(x_B)^d \\
&\lesssim_{M}\tau \sum_{j=1}^N\sum_{B\in\mathcal{B}_j}\sigma(B)
\lec \tau \sigma(3C_1B_{0})
\lec_{M,C_1} \tau,
\end{align*}
by \rf{eqzx3} and because the balls $B(z(x_B), 2C_5\delta(x_B))$ have finite superposition.
\end{proof}

\begin{lemma}\label{lemhd}
	$\mu(\HD)\lec_{A,\tau} \ve^{\frac{1}{4}}$. 
\end{lemma}

\begin{proof}
Take $x\in\HD$ and note again that $\delta(x)\geq d(x)>0$. Then arguing as in Lemma \ref{l:LD}, there is a point $z(x)\in\Sigma$
satisfying \rf{eqzx1} and \rf{eqzx2}.

Observe that for any ball $B_j$ such that $2B_{j}\cap B(z(x),4C_5\delta(x))\neq\varnothing$, by
\rf{eqzx1} we have
\[
\delta(x) \geq d(x) \geq d(\xi_{j})-|x-z(x)|-|z(x)-\xi_{j}|\geq d(\xi_{j}) -C_5\delta(x)-2s_{j}-4C_5\delta(x)
\]
and so 
\begin{equation}\label{e:delta>si}
\delta(x)\geq  \frac{d(\xi_{j})-2s_{j}}{1+5C_5}
=\frac{\eta^{-1}s_{j}-2s_{j}}{1+5C_5}
>\frac{s_{j}}{C_5}.
\end{equation}
This and the fact that $2B_{j}\cap B(z(x),4C_5\delta(x))\neq\varnothing$ imply
\begin{equation}\label{e:Bi<C0d}
B_{j}\subset B(z(x),4C_5\delta(x)+4s_{j}) \subset   B(z(x),8C_5\delta(x)).
\end{equation}

Consider now the function 
$\lambda(y) = (4C_5\delta(x)-|y-z(x)|)_{+}$. Observe that this is $1$-Lipschitz and, by \rf{eqzx2}, satisfies
$$2C_5\delta(x)\,\chi_{B(x,\delta(x))}\leq 2C_5\delta(x)\,\chi_{B(z(x),2C_5\delta(x))} \leq\lambda\leq 4C_5\delta(x)\,\chi_{B(z(x),4C_5\delta(x))}.$$
Thus, with constants $C$ depending only on $A$ and $\tau$, and for $\ve>0$ small enough, we get
\begin{align*}
4C_5\delta(x)  &\,\nu(B(z(x),4C_5\delta(x)))
 \geq \int \lambda(y)\,d\nu(y)\\
&\!\!\! \stackrel{\eqref{e:nuthetamu}}{\geq} \int\lambda(y)\theta(y) \,d\mu(y) -\sum_{2B_{j}\cap B(z(x),4C_5\delta(x))\neq\varnothing} \ve^{\frac14} s_{j}^{d+1}\\
& \!\!\!\!\!\! \stackrel{\eqref{e:mumutheta},\eqref{e:delta>si}}{\geq} \int \lambda(y)\,d\mu(y) - C\ve^{\frac{1}{2}} \delta(x)^{d+1} - C\ve^{\frac14} \!\sum_{2B_{j}\cap B(z(x),4C_5\delta(x))\neq\varnothing} \!\! \delta(x)\,\sigma(B_{j})\\
& \!\!\stackrel{\eqref{e:Bi<C0d}}{\geq} \delta(x)\, \mu(B(x,\delta(x))) -C\ve^{\frac{1}{2}} \delta (x)^{d+1} -  C\ve^{\frac14}\, \delta(x) \,\sigma( B(z(x),8C_5\delta(x)))\\
& \geq A\delta(x)^{d+1} - C\ve^{\frac{1}{4}} \delta(x)^{d+1}
\geq \frac{A}{2}\,\delta(x)^{d+1} .
\end{align*}
Hence, by \rf{pi-nu}
\[
\pi[\nu](B(\pi(z(x)),4C_5\delta(x))) \geq \nu(B(z(x),4C_5\delta(x)))
\geq \frac{A}{8C_5}\, \delta(x)^{d}. 
\]
In particular, for $y\in B(\pi(z(x)),C_5\delta(x))\cap P_{0}$,
\[
\frac{\pi[\nu](B(y,5C_5\delta(x)))}{\cH^{d}(P_{0}\cap B(y,5C_5\delta(x)))}
\gtrsim \frac{\pi[\nu](B(\pi(z(x)),4C_5\delta(x)))}{\delta(x)^{d}}
\gtrsim A.\]
So choosing $A\gg c_{C_{2}B_{0}}$ (recall $c_{C_{1}B_{0}} \stackrel{\eqref{e:cb-theta}}{\lec}  \Theta^d_{\mu}(C_{1}B_{0}) =1$ by assumption), and since  $f=\frac{d\pi[\nu]}{d\cH^{d}|_{P_{0}}} $,
\[
\cM\ps{f - c_{C_{2}B_{0}}}(y)
\geq \frac{\pi[\nu](B(y,5C_5\delta(x)))}{\cH^{d}(P_{0}\cap B(y,5C_5\delta(x)))} -c_{C_{1}B_{0}}
\gtrsim A,
\]
where $\cM$ is the Hardy-Littlewood maximal function on $P_{0}$. Therefore, for any $x\in\HD$ we have
\begin{equation*}
B(\pi(z(x)),C_5\delta(x))\cap P_0\subset \ck{y\in P_0: \cM\ps{f - c_{C_{1}B_{0}}}(y) >1}=:\Gamma.
\end{equation*}
On the other hand, since $\cM$ is of weak type $(2,2)$ with respect to $\HH^d|_{P_0}$, by Lemma \ref{lem8.1} we have
\begin{align}\label{e:pib-Lambda}
\HH^d(\Gamma) &\lesssim \|\cM(f-c_{C_{2}B_{0}})\|_{L^{2}(\HH^d|_{P_0})}^{2} \lesssim \|f-c_{C_{2}B_{0}}\|_{L^{2}(\HH^d|_{P_0})}^{2}
 \lec_{A,\tau}  \ve^{\frac{1}{4}}.
 \end{align}
%{\color{blue} Note that with initial definition of Besicovitch subcovering a number of things change here}

Let $\{\mathcal{B}_j\}_{j=1}^{N}$ be a Besicovitch subcovering from the collection $\{B(x,4C_5\delta(x)):x\in \HD\}$. 
There is $j_0\in\{1,\cdots,N\}$ such that the disjoint subcollection $\mathcal{B}_{j_0}$ satisfies
$$\mu(\HD)\lesssim \sum_{B\in\mathcal{B}_{j_0}}\mu(B).$$
Denote by $x_B$ the center of each $B$, and note that the balls $B(z(x_B),2C_5\delta(x_B))$, $B\in\mathcal{B}_{j_0}$, are also pairwise disjoint,
by \rf{eqzx2}.
Since $\Sigma$ is a Lipschitz graph with a very small constant and the balls $B(z(x_B),2C_5\delta(x_B))$ are centered in $\Sigma$ are
pairwise disjoint, it follows that the $d$-dimensional balls $B(\pi(z(x_B)),C_5\delta(x_B))$, $B\in\mathcal{B}_{j_0}$, are also pairwise disjoint.
Then we get
\begin{align*}
\mu(\HD)&
\lesssim \sum_{B\in\mathcal{B}_{j_0}} \mu(B)
  \lesssim_{A} \sum_{B\in\mathcal{B}_{j_0}} r_{B}^{d} 
\lec \sum_{B\in\mathcal{B}_{j_0}} \cH^{d}\bigl(B(\pi(z(x_B)),C_5\delta(x_B))\bigr)\\
&\leq \cH^{d}\biggl( \bigcup_{B\in\mathcal{B}_{j_0}}B(\pi(z(x_B)),C_5\delta(x_B)))\biggr)\leq \HH^d(\Gamma)
\stackrel{\eqref{e:pib-Lambda}}
\lec_{A,\tau}
\ve^{\frac{1}{4}}.
\end{align*}
\end{proof}

\begin{lemma}
	$\mu(\BA)\lec_{A,\tau,M} \ve^{\frac{1}{2}}$.  
\end{lemma}

\begin{proof}
As in Lemma \ref{lemld}, for each $x\in\BA$ there is a point $z(x)\in\Sigma$
satisfying \rf{eqzx1} and \rf{eqzx2}.
Let $\{\mathcal{B}_j\}_{j=1}^{N}$ be a Besicovitch subcovering 
 from the collection $\{B(x,4C_5\delta(x)): x\in \BA\}$ with centers $x_B$ and radii $t_B$.
Again as in Lemma \ref{lemld}, for each $x_B$ we take $k$ such that $r_{k-1}\leq \delta(x_B) < r_{k}$, and so there exists
some $i\in J_k$ such that $x_B\in 2B_{i,k}$. Then, using  \eqref{e:angle} and the subsequent
Remark \ref{remark2},
	\begin{equation}\label{angle-estimate}
	\angle(P_{i,k},P_{0})\geq \angle(L_{x_B,\delta(x_B)},P_{0})-\angle(P_{i,k},L_{x_B,\delta(x_B)})
	\gec \ve^{\frac{1}{4}} - C \ve \gec \ve^{\frac{1}{4}}. 
	\end{equation}
	By \eqref{e:49r} and \eqref{e:sky-y}, taking into account \rf{eqzx2},
	\begin{equation}\label{h-dist-estimate}
	\dist_{H}(B \cap \Sigma, B\cap P_{i,k})\lec \ve r_{k}\approx \ve\, \delta(x_B)\approx \ve\,
	r_{B}.\end{equation}
	Thus, by \rf{e:3bjcapSigma}, Lemma \ref{lem6.8}, \rf{angle-estimate} and \rf{h-dist-estimate}, we have
	\begin{align*}
	\avint_{\pi(B\cap \Sigma)} |Dh|^{2} \,d\HH^d &\gec \;\avint_{\pi(B\cap \Sigma)} \frac{|h-h(\pi(x_B))|^{2}}{r_{B}^2} \,d\HH^d\\
	&\gec  \;\avint_{\pi(B\cap \Sigma)} \frac{|A_{j,k}-h(\pi(x_B))|^{2}}{r_{B}^2}\,d\HH^d - C \ve^2
	\gec \ve^{\frac12} -C\ve^2 \gec \ve^{\frac12}. 
	\end{align*}
	 Then
	\begin{align*}
	\mu(\BA)
	& \leq \sum_{j=1}^N\sum_{B\in\mathcal B_{j}}\mu(B)
	\lesssim_{A} \sum_{j=1}^N\sum_{B\in\mathcal B_{j}}r_{B}^{d} 
	\lec \sum_{j=1}^N\sum_{B\in\mathcal B_{j}} \HH^d(\pi(B\cap \Sigma))\\
	& \lec \ve^{-\frac12} \sum_{j=1}^N\sum_{B\in\mathcal B_{j}} \int_{\pi(B\cap \Sigma)} |Dh|^{2} \,d\HH^d
 	\lec \ve^{-\frac{1}{2}} \int |Dh|^{2} \lec \ve^{-\frac12}\ve\lec \ve^{\frac12},
	\end{align*}
	by \rf{e:h-int}.
\end{proof}
\vv

Putting these lemmas together, we obtain that for $\tau$ and $\ve$ small enough,
\[
\mu( \HD\cup\LD\cup \ND\cup \BA)\leq C\tau + C(A,\tau) \ve^{\frac{1}{4}}  <\mu(E_{0}),
\]
which implies that there exists a subset $G'\subset E_{0}\cap B_{0}$ of positive measure for which $\Theta^d_{\mu}(x,r) \in [\tau,A]$ for all $r>0$ small. 
This contradicts Lemma \ref{l:good} and concludes the proof of
Lemma \ref{l:main} and of Theorem \ref{thmi}.

\vv

% **********************************************************************************************************************

\section{Proof of Theorem \ref{teocount1}}

In this section we assume $n=2$ and $d=1$. That is, we are in the plane and we consider $1$-dimensional $\alpha$-numbers.
Our objective is to construct a measure $\mu$ such that
	\begin{equation}\label{eq**1}
	\int_0^1\alpha_\mu(x,r)^2\,\frac{dr}r<\infty\quad \mbox{for all $x\in \supp\mu$},
	\end{equation}
	and such that 
	\begin{equation}\label{10.1A}
	\lim_{r\to 0}\frac{\mu(B(x,r))}r =0\quad\mbox{ for all $x\in\supp\mu$.}
	\end{equation}

Given $h>0$ and a horizontal line $L\subset\R^2$, we denote by $L(h)$ the line parallel to $L$ and above $L$ which is at a distance $h$ from $L$. That is, $L(h) = h\,e_2+ L$, where $e_2=(0,1)$.
Our measure $\mu$ will be a weak limit of measures $\mu_k$ of the form
\begin{equation}\label{10.1B}
\mu_k=\sum_{j=1}^{n_k} c_j^k \,\HH^1|_{L_j^k},
\end{equation}
where $L_j^k$, $j=1,\ldots,n_k$ are horizontal lines.

We consider two decreasing sequences of positive numbers $\{a_k\}_k$, $\{h_k\}_k$, tending to $0$, which will be chosen later, and so that $0<a_k,h_k<1/2$. The reader should think that $h_k$ will tend to zero much faster than $a_k$.
First we set $\mu_0 = \HH^1|_{L_0^0}$, where $L_0^0$ is just the horizontal axis. Inductively, $\mu_{k+1}$ is constructed from $\mu_k$ as in \rf{10.1B}, as follows:
\begin{equation}\label{def-mk}
\mu_{k+1} = \sum_{j=1}^{n_k} c_j^k \,\bigl[(1-a_{k+1}) \HH^1|_{L_j^k} + a_{k+1}\,\HH^1|_{L_j^k(h_{k+1})}\bigr]
.\end{equation}
So roughly speaking, at each step $k+1$, each line $L_j^k$ is split into the two lines $L_j^k$ and 
$L_j^k(h_{k+1})$ and the total mass is distributed so that a fraction $(1-a_{k+1})$ is kept in $L_j^k$, while the other fraction $a_{k+1}$ is transferred to $L_j^k(h_{k+1})$. Further, one should think that $h_k$ goes to $0$ 
very quickly, and $h_{k+1}$ is much smaller than any of the mutual distances among the lines $L_j^k$, $j=1,\ldots,n_k$.
We claim that $a_k$ and $h_k$ can be chosen so that \rf{eq**1} holds but $\mu$ is singular with respect to
$\HH^1$.

First we just analyse how the $\alpha$ coefficients evolve from $\mu_0$ to $\mu_1$. Consider the measure 
$$\mu_1= (1-a)\,\HH^1|_L + a\,\HH^1|_{L'},$$
where $a=a_1$, $L=L_0^0$, and $L'=L(h)$, with $h=h_1$.
Consider a $1$-Lipschitz function $\phi$ supported on $B(x,r)$, with $x\in \supp \mu_1$ and estimate the following integral using a change of variable
\begin{align}\label{eq**15}
\left|\int \phi \,d\bigl(\HH^1|_L - \mu_1)\right| & = \left|\int \phi \,d(a\,\HH^1|_{L} - a\,\HH^1|_{L'})
\right|\\
& = a \left|\int \bigl(\phi(x) - \phi(x+h)\bigr) \,d\HH^1|_{L} \right| \lesssim a\,h\,r. \nonumber
\end{align}

First we estimate $\alpha_{\mu_1}(x,r)$ for $x\in L$. To this end, note that since $a<1/2$, $\mu(B(x,r))
\approx r$ for all $r>0$. 
Further, $\alpha_{\mu_1}(x,r)=0$ if $r\leq h$. 
For $r>h$, we write
$$\alpha_{\mu_1}(x,r))\lesssim \frac1{r^2}\,\dist_{B(x,r)}(\mu_1,\HH^1|_L).$$
 
Hence by \rf{eq**15} we obtain
\begin{equation}\label{eq**2}
\alpha_{\mu_1}(x,r))\lesssim a\,\frac hr,
\end{equation}
and thus 
\begin{equation}\label{eq**25}
\int_0^\infty\alpha_\mu(x,r)^2\,\frac{dr}r\lesssim a^2\int_h^\infty\frac{h^2}{r^2}\,\frac{dr}r\approx a^2.
\end{equation}

Next we turn our attention to the case when $x\in L'$. Again, for $r\leq h$, $\alpha_{\mu_1}(x,r)=0$.
On the other hand, for $r>2h$, $\mu_1(B(x,r))\approx r$, and almost the same calculations as 
above (i.e. for $x\in L$ and $r>h$) show 
that
$$\alpha_{\mu_1}(x,r)\lesssim a\,\frac hr,$$
as in \rf{eq**2}. This yields
\begin{equation}\label{eq**3}
\int_{2h}^\infty 
\alpha_\mu(x,r)^2\,\frac{dr}r\lesssim a^2\int_{2h}^\infty\frac{h^2}{r^2}\,\frac{dr}r\approx a^2.
\end{equation}

Assume now that $x\in L'$ and $h<r<2h$. An easy geometric argument shows that 
$$\HH^1(L\cap B(x,r)) = 2\sqrt{r^2-h^2}.$$
Hence
$$\mu_1(B(x,r)) = 2(1-a)\sqrt{r^2-h^2} + 2ar.$$
By \rf{eq**15} we have
$$\alpha_{\mu_1}(x,r)\lesssim\frac{ahr}{r\mu_1(B(x,r))}\lesssim \frac{a\,h}{2(1-a)\sqrt{r^2-h^2} + 2ar}
\lesssim \frac{a\,h}{\sqrt{r^2-h^2} + ar},$$
and so
$$\alpha_{\mu_1}(x,r)^2\lesssim \frac{a^2\,h^2\,}{r^2-h^2 + (ar)^2} 
= \frac{a^2\,h^2\,}{(1+a^2)r^2-h^2 }
.$$
Therefore, 
\begin{align*}
\int_{h}^{2h} 
\alpha_\mu(x,r)^2\,\frac{dr}r &\lesssim 
\int_{h}^{2h}\frac{a^2\,h^2\,}{(1+a^2)r^2-h^2 }\,\frac{dr}r \lesssim
\int_{h}^{2h}\frac{a^2\,r\,}{(1+a^2)r^2-h^2 }\,dr\\
& = \frac{a^2}{2(1+a^2)}\,\biggl[\log((1+a^2)r^2-h^2)\biggr]_h^{2h} = 
\frac{a^2}{2(1+a^2)}\,\log \frac{3+4a^2}{a^2}\\
&\lesssim a^2\,\log\frac{4}{a^2}\approx a^2\,|\log a|
.
\end{align*} 
Together with \rf{eq**3}, this yields for $x\in L'$ since $0<a<\frac{1}{2}$ that
\begin{equation}\label{10.5A}
\int_{0}^\infty 
\alpha_\mu(x,r)^2\,\frac{dr}r\lesssim a^2\,|\log a|.
\end{equation}
Further, because of \rf{eq**25}, this estimate is also valid for $x\in L$.
\vv

By the same arguments, in the  step $k+1$, denoting 
$d_{k}$ the minimal distance among the lines $L_j^{k}$, $j=1,\ldots,n_{k}$, that form the support of $\mu_{k}$, and assuming that $h_{k+1}\ll d_k$,
 we obtain as before that
\begin{equation}\label{eq**5}
\int_0^{d_{k}/2}
\alpha_{\mu_{k+1}}(x,r)^2\,\frac{dr}r 
\lesssim a_{k+1}^2\,|\log a_{k+1}|\quad\mbox{ for all $x\in\supp\mu_{k+1}$.}
\end{equation}
On the other hand, for $r\geq d_k/2$, we claim that if we choose $h_{k+1}$ small enough so that
\begin{equation}\label{eqhk}
h_{k+1}\leq a_{k+1}^{4(k+1)}\,d_k,
\end{equation}
then we have
\begin{equation}\label{eq*50}
\alpha_{\mu_{k+1}}(x,r)\leq \biggl(1+ \frac{h_{k+1}^{1/4}}{d_k^{1/4}}\,\biggr)\,\alpha_{\mu_k}(x',r')
+ C\,\frac{h_{k+1}}r,
\end{equation}
where $x'$ is the closest point to $x$  from $\supp\mu_k$ and $r'=r+h_{k+1}$. We defer the proof of this estimate and
show how obtain \rf{eq**1} provided \rf{eqhk} and \rf{eq*50} hold.

\vv
For any $0<\ve_k<1/2$, using that $h_{k+1}\ll d_k$, a change of variable and the fact that the function $\frac{s}{s-h_{k+1}}$ is decreasing, we have
\begin{align*}
\int_{d_{k}/2}^\infty\!
\alpha_{\mu_{k+1}}(x,r)^2\,\frac{dr}r & \leq 
(1+\ve_k)\biggl(1+ \frac{h_{k+1}^{1/4}}{d_k^{1/4}}\,\biggr)^2\int_{d_{k}/2}^\infty
\alpha_{\mu_{k}}(x',r+h_{k+1})^2\,\frac{dr}r \\& \quad+ C\,\ve_k^{-1} \int_{d_{k}/2}^\infty\frac{h_{k+1}^2}{r^2}\,\frac{dr}r\\
& \leq (1+\ve_k)\biggl(1+ 3\frac{h_{k+1}^{1/4}}{d_k^{1/4}}\,\biggr)\frac{\frac12d_k}{\frac12 d_k -h_{k+1}}\int_0^\infty\!\!
\alpha_{\mu_{k}}(x',s)^2\,\frac{ds}s + C\,\frac{h_{k+1}^2}{\ve_k\,d_{k}^2}.
\end{align*}

Together with \rf{eq**5}, and using that
$$\frac{\frac12d_k}{\frac12 d_k -h_{k+1}}\leq 1+ C\,\frac{h_{k+1}}{d_k} \leq 1+ C\,\frac{h_{k+1}^{1/4}}{d_k^{1/4}},$$
this gives
\begin{align}\label{eq*55}
\int_0^\infty
\alpha_{\mu_{k+1}}(x,r)^2\,\frac{dr}r &
\leq C\,a_{k+1}^2\,|\log a_{k+1}| \\ &\quad+ (1+\ve_k)
\biggl(1+ C\,\frac{h_{k+1}^{1/4}}{d_k^{1/4}}\,\biggr)\int_0^\infty
\alpha_{\mu_{k}}(x',r)^2\,\frac{dr}r + C\,\frac{h_{k+1}^2}{\ve_k\,d_{k}^2}.\nonumber
\end{align}
Choosing $\ve_k=2^{-k}$ and assuming also that
\begin{equation}\label{eqas490*'}
\frac{h_{k+1}^{1/4}}{d_k^{1/4}}\leq 2^{-k},
\end{equation}
 iterating the estimate \rf{eq*55} it follows that
$$\int_0^\infty
\alpha_{\mu_{k+1}}(x,r)^2\,\frac{dr}r \lesssim \sum_{j=1}^{k+1} a_j^2\,|\log a_j| +\sum_{j=1}^k
\frac{2^j\,h_{j+1}^2}{d_j^2} \lesssim 1+ \sum_{j=1}^{k+1} a_j^2\,|\log a_j|.$$
Since this estimate is uniform on $k$, we derive
\begin{equation}\label{10.10A}\int_0^\infty
\alpha_{\mu}(x,r)^2\,\frac{dr}r \lesssim 1 + \sum_{j\geq1} a_j^2\,|\log a_j| .
\end{equation}

Let $\{a_j\}_j$ be a sequence such that the last sum in \rf{10.10A} is finite (which guaranties that \rf{eq**1} holds)
but so that $\sum_j a_j=\infty$
(such as $a_j = 1/j$, for example). Further choose $\{h_k\}$ inductively so that
both \rf{eqas490*'} and \rf{eqhk} hold.
Recall that by construction, $\mu_k$ has constant $1$-dimensional density on each line $L_i^k$, $i=1,\ldots,n_k$, and further this density is at most
$\prod_{j=1}^k (1-a_j).$
The condition $\sum_j a_j=\infty$ implies that this product tends to $0$ as $k\to\infty$. In turn this ensures that $\mu$ has  vanishing upper density at all points, and thus $\mu$ is singular with respect
to $\HH^1$. This completes the construction of the counterexample, modulo the proof of claim \rf{eq*50}.

\vv
{\bf Proof of \rf{eq*50}}.
Consider a $1$-Lipschitz function $\phi$ supported on $B(x,r)$ and
denote $B=B(x,r)$ and $B'=(x',r')$, with $r'=r+h_{k+1}$ (note that $|x-x'|\leq h_{k+1}$).
Let $c_{B'},L_{B'}$ some pair for which the minimum is attained in the definition of $\alpha_{\mu_k}(B')$ as in \rf{cL-ppts}.
Then we have
\begin{align}\label{eq*51}
\left|\int \phi\,d(c_{B'}\HH^n|_{L_{B'}} - \mu_{k+1})\right| & \leq
\left|\int \phi\,d(c_{B'}\HH^n|_{L_{B'}} - \mu_{k})\right| +
\left|\int \phi\,d(\mu_k - \mu_{k+1})\right|\\
&\leq F_{B'}\bigl(c_{B'}\HH^n|_{L_{B'}},\mu_{k}\bigr) + \left|\int \phi\,d(\mu_k - \mu_{k+1})\right|,
\nonumber
\end{align}
taking into account that $\supp\phi\subset B\subset B'$ in the last inequality.
To deal with the last integral note that
\begin{align*}
\mu_k - \mu_{k+1} & = \sum_{j=1}^{n_k} c_j^k \,\HH^1|_{L_j^k}  - \sum_{j=1}^{n_k} c_j^k \,\bigl[(1-a_{k+1}) \HH^1|_{L_j^k} + a_{k+1}\,\HH^1|_{L_j^k(h_{k+1})}\bigr]
\\
& = a_{k+1}
 \sum_{j=1}^{n_k} c_j^k \,\bigl[\HH^1|_{L_j^k} - \HH^1|_{L_j^k(h_{k+1})}\bigr]
 = a_{k+1} \,\bigl(\mu_k -  \,P_k[\mu_k]\bigr),
 \end{align*}
where $P_k$ stands for the translation $P_k(y) = y+h_{k+1}$.
Therefore,
$$\left|\int \phi\,d(\mu_k - \mu_{k+1})\right| = a_{k+1}
\left|\int \bigl(\phi(y)-\phi(y+h_{k+1})\bigr)\,d\mu_k\right| \leq a_{k+1}\,h_{k+1}\,\mu_k(B'),$$
where we used that $\phi$ is $1$-Lipschitz and $\supp\phi\cup\supp\phi(\cdot+h_{k+1})\subset B'$.
From this inequality and \rf{eq*51} we deduce that
\begin{align}\label{eqal99*}
\alpha_{\mu_{k+1}}(x,r)&\leq \frac{F_{B'}\bigl(c_{B'}\HH^n|_{L_{B'}},\mu_{k}\bigr) + a_{k+1}\,h_{k+1}\,\mu_k(B')}
{r\,\mu_{k+1}(B)}\\
& = \frac{r'\,\mu_k(B')}{r\,\mu_{k+1}(B)}\,
\alpha_{\mu_k}(x',r') + \frac{a_{k+1}\,h_{k+1}\,\mu_k(B')}
{r\,\mu_{k+1}(B)}.\nonumber
\end{align}
Next we need to estimate $\mu_k(B')$ in terms of $\mu_{k+1}(B)$. To this end, 
recall that
$$\supp\mu_k =\bigcup_{j=1}^{n_k} \bigl(L_j^k \cup L_j^k(h_{k+1})\bigr)$$
and denote by $T_k:\supp\mu_{k+1}
\to\supp\mu_k$ the map which equals the identity on each line $L_j^k$ and 
coincides with the orthogonal projection from $L_j^k(h_{k+1})$ to $L_j^k$ on each $L_j^k(h_{k+1})$.
By construction, $\mu_k=T_k[\mu_{k+1}]$, and thus  
$\mu_k(B') = \mu_{k+1}(T_k^{-1}(B'))$. Moreover note that
$$T_k^{-1}(B'\cap \supp\mu_k)\subset B'\cup (B'-h_{k+1}) \subset B(x,r+2h_{k+1}).$$
Therefore,
\begin{equation}\label{10.12A}
\mu_k(B') \leq  \mu_{k+1}(B(x,r+2h_{k+1})) = \mu_{k+1}(B) + \mu_{k+1}(A(x,r,r+2h_{k+1})),
\end{equation}
where $A(x,r,r+2h_{k+1})$ denotes the annular region of center $x$ and between radii $r$ and $r+2h_{k+1}$.
By geometric considerations, it is immediate to check that for any line $L\subset\R^d$,
$$\HH^1(L\cap A(x,r,r+2h_{k+1})) \leq 2\sqrt{(r+2h_{k+1})^2-r^2} =  2\sqrt{4h_{k+1}r + 4h_{k+1}^2}\leq 
2\sqrt{8h_{k+1}r}, 
$$
since we have $h_{k+1}\ll d_k\leq r$.
Therefore,
\begin{equation}\label{10.12B}
\mu_{k+1}(A(x,r,r+2h_{k+1}))\leq 2 \sum_{i=1}^{n_{k+1}} c_i^{k+1}\sqrt{8h_{k+1}r}\leq 2\sqrt{8h_{k+1}r},
\end{equation}
since by construction $\sum_{i=1}^{n_{k+1}} c_i^{k+1}=1$.
Combining \rf{10.12A} and \rf{10.12B}, we derive
$$\mu_k(B') \leq  \mu_{k+1}(B) + C\,\sqrt{h_{k+1}r}.$$
Plugging this estimate into \rf{eqal99*} we obtain
\begin{equation}\label{eq*60}
\alpha_{\mu_{k+1}}(x,r)\leq  \frac{r+h_{k+1}}r \,\left(1 + C\,\frac{\sqrt{h_{k+1}r}}{\mu_{k+1}(B)}
\right)\,
\alpha_{\mu_k}(x',r') + \frac{a_{k+1}\,h_{k+1}}r \left(1+\frac{\sqrt{h_{k+1}r}}{\mu_{k+1}(B)}\right).
\end{equation}
Since $B$ is centered on some line $L_i^{k+1}$, we can use the following brutal estimate for 
$\mu_{k+1}(B)$: 
$$\mu_{k+1}(B)\geq a_1\ldots a_{k+1} \,r\geq a_{k+1}^{k+1}\,r.$$
So by the assumption \rf{eqhk} we infer that
$$\frac{\sqrt{h_{k+1}r}}{\mu_{k+1}(B)}\leq \frac{h_{k+1}^{1/2}}{a_{k+1}^{k+1} \,r^{1/2}}
\leq \frac{h_{k+1}^{1/4}}{d_k^{1/4}}
.$$
Our claim \rf{eq*50} is an immediate consequence of this estimate and \rf{eq*60}.
\vvv

\bibliographystyle{alpha}

\def\cprime{$'$}

\end{document}